\documentclass[a4paper,twoside,11pt,leqno]{article}

\usepackage[utf8]{inputenc}
\usepackage[T1]{fontenc}
\usepackage{amsmath, amssymb, amsthm}
\usepackage{bm, mathtools}
\usepackage{newtxtext, newtxmath}
\usepackage{nicefrac}
\usepackage[babel=true]{csquotes}
\usepackage[english]{babel}
\usepackage[margin=3.5cm]{geometry}
\usepackage{fancyhdr}

\usepackage{graphicx}
\usepackage{pinlabel}

\usepackage{enumitem}
\usepackage{makecell}

\usepackage{upref}
\usepackage{hyperref}
\usepackage{caption}
\usepackage{cite}

\hypersetup{%
   colorlinks,
   breaklinks,
   linkcolor=black,
   citecolor=black,
   filecolor=black,
   urlcolor=black
}

\mathtoolsset{%
   showonlyrefs,
   showmanualtags
}

\numberwithin{equation}{section}

\fancypagestyle{plain}{%
   \fancyhf{} 
}

\fancypagestyle{main}{%
   \fancyhf{} 
   \fancyhead[LE,RO]{\textbf{\thepage}}
   \fancyhead[RE]{\textit{\shortauthors}}
   \fancyhead[LO]{\textit{\shorttitle}}
}

\newcommand{\figref}[1]{Fig.~\ref{#1}}
\newcommand{\defref}[1]{Def.~\ref{#1}}
\newcommand{\teqref}[1]{Eq.~\eqref{#1}}

\newcommand{\secref}[1]{Sec.~\ref{#1}}

\newcommand{\thmref}[1]{Thm.~\ref{#1}}
\newcommand{\propref}[1]{Prop.~\ref{#1}}
\newcommand{\lemref}[1]{Lem.~\ref{#1}}

\newcommand{\probref}[1]{Problem~\ref{#1}}
\newcommand{\remref}[1]{Remark~\ref{#1}}

\def\etal{\emph{et al.}}

\def\ie{\emph{i.e.}}
\def\eg{\emph{e.g.}}

\newcommand\blfootnote[1]{%
  \begingroup
  \renewcommand\thefootnote{}\footnote{#1}%
  \addtocounter{footnote}{-1}%
  \endgroup
}

\theoremstyle{definition}
\newtheorem{theorem}{Theorem}

\newtheorem{lemma}{Lemma}[section]
\newtheorem{corollary}[lemma]{Corollary}
\newtheorem{proposition}[lemma]{Proposition}

\newtheorem{propdef}[lemma]{Proposition\,\&\,Definition}
\newtheorem{definition}[lemma]{Definition}

\newtheorem{problem}[lemma]{Problem}

\theoremstyle{remark}
\newtheorem{remark}{Remark}[section]

\newcommand{\RR}{\mathbb{R}}

\newcommand{\HH}{\mathbb{H}}
\renewcommand{\SS}{\mathbb{S}}

\newcommand{\polyhedra}{\mathcal{C}}
\newcommand{\surf}{\Sigma}
\newcommand{\toposurf}{\mathcal{S}}

\DeclareMathOperator{\SO}{SO}

\DeclareMathOperator{\Diff}{D}
\DeclareMathOperator{\laplacian}{\Delta}
\newcommand{\diff}[2]{\mathrm{d}_{#1}{#2}}
\newcommand{\tdiff}[1]{\frac{\diff{}{}}{\diff{}{#1}}}

\DeclareMathOperator{\acosh}{acosh}
\DeclareMathOperator{\cl}{cl}
\DeclareMathOperator{\dist}{dist}
\DeclareMathOperator{\vol}{Vol}
\DeclareMathOperator{\area}{area}
\DeclareMathOperator{\lin}{span}
\DeclareMathOperator{\HE}{\mathcal{H}}

\newcommand{\ip}[2]{\langle #1, #2 \rangle}

\newcommand{\ee}{\mathrm{e}}
\newcommand{\tri}{\mathcal{T}}
\newcommand{\verts}{V}
\newcommand{\edges}{E}
\newcommand{\faces}{F}
\newcommand{\corner}[2]{\vphantom{\phi}_{#2}^{#1}}
\newcommand{\len}{\ell}
\newcommand{\cotw}{w}

\begin{document}
\thispagestyle{plain}
\title{
    Decorated discrete conformal equivalence\\
    in non-Euclidean geometries
}
\newcommand{\shorttitle}{%
    Decorated discrete conformal equivalence in non-Euclidean geometries
}

\author{%
   \textsc{Alexander~I.~Bobenko}\\[0.1cm]
   \textsc{Carl~O.~R.~Lutz}
}
\newcommand{\shortauthors}{%
   A.\,I.\,Bobenko and C.\,O.\,R.\,Lutz
}

\date{}

\maketitle

\begin{abstract}
    We introduce decorated piecewise hyperbolic and spherical surfaces and
    discuss their discrete conformal equivalence. A decoration is a
    choice of circle about each vertex of the surface. Our decorated surfaces
    are closely related to inversive distance circle packings, canonical
    tessellations of hyperbolic surfaces, and hyperbolic polyhedra.

    We prove the corresponding uniformization theorem. Furthermore, we show
    that on can deform continuously between decorated piecewise hyperbolic,
    Euclidean, and spherical surfaces sharing the same fundamental discrete
    conformal invariant. Therefore, there is one master theory of discrete
    conformal equivalence in different background geometries. Our approach
    is based on a variational principle, which also provides a way to
    compute the discrete uniformization and geometric transitions.\\


    \blfootnote{\textit{Date}: \today}
\end{abstract}


\pagestyle{main}

\section{Introduction}
\label{sec:introduction}
The Poincar\'e--Koebe uniformization theorem of Riemann surfaces
is undoubtedly one of the main results of 19th century mathematics.
The impact of the new ideas and techniques developed for
its poof reached far into the 20th century (for more information on the history
of this theorem see, \eg, \cite{Abikoff1981, DeSaint-Gervais2016}).

At the beginning of the 21th century theories of discrete analytic functions
and conformal maps began were developed. They come in two major flavors: one is build
on circle packings; the other on deforming the edge-lengths of
triangulated piecewise Euclidean surfaces by scale factors associated to the
vertices. Circle packings were already investigated
by \textsc{P.~Koebe} \cite{Koebe1936}. However modern interest started with
\textsc{W.~Thurston's} idea to approximate the classical Riemann map using
circle patterns which was later proved by \textsc{B.~Rodin} and \textsc{D.~Sullivan}
\cite{RS1987}. The vertex-scaling approach first appeared in the context of
relativity theory \cite{RW1984}. Its Riemannian geometric equivalent was introduced by
\textsc{F.~Luo} \cite{Luo2004}. Both the notion of a circle and a triangle are not
limited to Euclidean geometry but are also natural objects in spherical or hyperbolic
space. This leads to definitions of circle packings and piecewise surfaces in different
\emph{background geometries} and corresponding discrete uniformization theorems
mimicking their smooth counterpart \cite{BS1990, GLS+2018, GGL+2018}.

The aim of this article is threefold: one motivation is develop the discrete
conformal theory of decorated piecewise hyperbolic and spherical surfaces. Amongst
others we will present a
variational principle for the corresponding discrete mapping problem,
explain the connections to decorated Teichm\"uller theory and hyperbolic polyhedra,
and prove the corresponding uniformization theorem. Another motivation is to
provide a bridge between the two different approaches to discrete conformal analysis.
Our notion of decorated discrete conformal equivalence incorporates both
\emph{inversive distance circle packings} as well as the \emph{vertex-scaling approach}
as special cases. Last but not least, we wish to explain how the definitions in
different background geometries are related. Indeed, instead of three separate
theories there is only one \emph{master theory} of discrete conformal equivalence. We
will show that one can continuously deform decorated surfaces between different
background geometries within one discrete conformal class based on our functional.

\subsection{Statements, related work, and open problems}
\paragraph{Hyperbolic background geometry.}
Informally speaking, a triangulated piecewise hyperbolic surface is a topological
surface $\toposurf_g$ glued edge-to-edge from hyperbolic triangles
(see \secref{sec:basic_definitions}). It is determined by a triangulation
$\tri$ with vertex set $\verts\subset\toposurf_g$
and \emph{edge-lengths} $\len_{ij}$ on the edges $ij$ of $\tri$.
Two discrete hyperbolic metrics $(\tri, \len)$ and $(\tri, \tilde{\len})$
are said to be \emph{discrete conformally equivalent} if there are
\emph{logarithmic scale factors} $u_i$ on the vertices $\verts$ such that
\begin{equation}\label{eq:classical_hyperbolic_dce}
    \sinh\frac{\tilde{\len}_{ij}}{2}
    \;=\;
    \ee^{\nicefrac{(u_i+u_j)}{2}}
    \sinh\frac{\len_{ij}}{2}
\end{equation}
for all edges $ij$ \cite[Sec.~6]{BPS2015}. One can derive convex variational principles
and uniqueness results for the corresponding discrete mapping problems. The reason
for this is the intimate connection of discrete conformal equivalence to decorated
Teichm\"uller spaces of hyperbolic cusp surfaces and hyperbolic polyhedra.

But, in general, the existence of solutions cannot be guarantee
if the combinatorics are fixed. Therefore, one extends the notion of discrete conformal
equivalence to variable combinatorics. At the heart of this extension lies
the aforementioned interpretation in terms of hyperbolic geometry.
It can be formulated using a sequence of Delaunay triangulations \cite{GGL+2018}.
Furthermore, it only depends on the piecewise hyperbolic metric $\dist_{\toposurf_g}$
on the marked surface $(\toposurf_g, \verts)$, \ie, the complete hyperbolic path-metric
on $\toposurf_g$ determined by $(\tri, \len)$.

\textsc{X.~Gu} \etal\ proved the corresponding prescribed cone-angle theorem
\cite{GGL+2018}. It states: let $\dist_{\toposurf_g}$ be a piecewise hyperbolic metric
on the marked surface $(\toposurf_g, \verts)$ and $\Theta\in\RR_{>0}^{\verts}$
satisfying the \emph{hyperbolic Gau{\ss}--Bonnet condition}
\begin{equation}\label{eq:hyperbolic_gauss_bonnet_condition}
   \frac{1}{2\pi}\sum\Theta_i
   \;<\;
   2g-2\,+\,|\verts|.
\end{equation}
There exists a unique second piecewise hyperbolic metric
$\widetilde{\dist}_{\toposurf_g}$ on $(\toposurf_g, \verts)$ which is
discrete conformally equivalent to $\dist_{\toposurf_g}$ and has the desired
angle sums $\Theta_i$ at the vertices.

A decoration of a piecewise hyperbolic metric is a choice of circle about each vertex.
They are determined by the piecewise hyperbolic metric and radii $r_i\geq0$ at
the vertices. The \emph{undecorated case} corresponds to the special choice
$r_i\equiv0$. We discuss two ways to approach the discrete conformal equivalence
of such surfaces (\secref{sec:ddce}). One is a direct generalization of the
vertex-scaling approach \eqref{eq:classical_hyperbolic_dce}. The other is based on
M\"obius-transformation associated to the faces of the triangulation.
Discrete conformal equivalence leaves invariant the \emph{inversive distance}
induced on the edges of $\tri$ (see \secref{sec:connections_hyperbolic_polyhedra})
showing the close connection to \textsc{P.~Bowers}' and \textsc{K.~Stephenson}'s
\emph{inversive distance circle patterns} \cite{BS2004, BH2003}.

The uniqueness of inversive distance circle patterns, and thus decorated discrete
hyperbolic metrics, was extensively studied \cite{Guo2011a, Luo2011, Xu2018}.
Local rigidity results were also obtained in \cite{GT2017} using
\textsc{D.~Glickenstein's} approach via \emph{duality structures}.
Yet, as in the undecorated case, the dependence on a fixed triangulation poses a strong
obstruction to general existence results. Weighted Delaunay tessellations are the
analogue of Delaunay tessellations for decorated surfaces
(\secref{sec:weighted_delaunay_tessellations}). They always exist and are uniquely
determined by the decorated piecewise hyperbolic metric if we consider
hyperideal decorations, \ie, choices of circles at the vertices such
that no pair intersects. Moreover, each decorated discrete hyperbolic metric induces
a complete hyperbolic surface $\surf_g$ with cusps and complete
ends of infinite area (\secref{sec:connections_hyperbolic_polyhedra}).
We call $\surf_g$ the \emph{fundamental discrete conformal invariant} of the decorated
surface if $\surf_g$ is induced by its weighted Delaunay tessellation. Relating
weighted Delaunay tessellations to canonical tessellations of hyperbolic surfaces
(\secref{sec:canonical_tessellations}), we see that two decorated
piecewise hyperbolic surfaces are discrete conformally equivalent if and only
if their fundamental discrete conformal invariant coincide
(\secref{sec:convex_polyhedra}). This leads to the main result about piecewise
hyperbolic surfaces of the paper.

\begin{theorem}[hyperbolic prescribed cone-angle problem]
    \label{theorem:realisation_hyperbolic}
    Let $(\dist_{\toposurf_g}, r)$ be a hyperideally decorated piecewise hyperbolic
    metric on the marked genus $g\geq0$ surface $(\toposurf_g, \verts)$. Denote by
    $\surf_g$ its fundamental discrete conformal invariant, \ie, the complete hyperbolic
    surface induced by any triangular refinement of its unique weighted Delaunay
    tessellation.
    \begin{enumerate}[label={\arabic*.}, wide=0.8\parindent, topsep=1pt, itemsep=0pt]
        \item\emph{(existence \& uniqueness)}
            There exists a unique decorated piecewise hyperbolic metric discrete
            conformally equivalent to $(\dist_{\toposurf_g}, r)$ realizing
            $\Theta\in\RR_{>0}^{\verts}$ if and only if $\Theta$ satisfies
            the hyperbolic Gau{\ss}--Bonnet condition
            \eqref{eq:hyperbolic_gauss_bonnet_condition}.
        \item\emph{(variational principle)} The logarithmic scale factors, which give
            to the change of metric, correspond to the maximum point of the strictly
            concave discrete Hilbert--Einstein functional $\HE_{\surf_g,\Theta}^{-}$
            (see \secref{sec:local_variational_principle}).
    \end{enumerate}
\end{theorem}

The discrete Hilbert--Einstein functional $\HE_{\surf_g,\Theta}^{-}$ is a
twice continuously differentiable function. It can be explicitly expressed using
the dilogarithm function (see \remref{remark:formula_for_hyperbolic_volume}). This
provides an effective method to compute discrete conformally equivalent metrics
with prescribed cone-angles. An important special case of this
\thmref{theorem:realisation_hyperbolic} is the analogue
of the classical Poincar\'e--Koebe uniformization theorem for decorated discrete
metrics.

\begin{corollary}[discrete uniformization of decorated piecewise hyperbolic surfaces]
    \label{thm:discrete_uniformization_hyperbolic}
    For each hyperideally decorated piecewise hyperbolic metric
    $(\dist_{\toposurf_g}, r)$ on the marked genus $g\geq2$ surface
    $(\toposurf_g, \verts)$ there is a unique discrete conformally equivalent
    decorated metric realizing the uniform angle distribution
    $\Theta_i\equiv2\pi$.
\end{corollary}

\paragraph{Spherical background geometry.}
Triangulated piecewise spherical surfaces and their decorations can essentially be
defined in the same way as their hyperbolic counterpart. The analog of
\eqref{eq:classical_hyperbolic_dce} in the spherical case is given by
\begin{equation}\label{eq:classical_spherical_dce}
    \sin\frac{\tilde{\len}_{ij}}{2}
    \;=\;
    \ee^{\nicefrac{(u_i+u_j)}{2}}
    \sin\frac{\len_{ij}}{2}
\end{equation}
for all edges $ij$ \cite{BSS2016}. We generalize
\eqref{eq:classical_spherical_dce} to decorations and also discuss the
equivalent approach via M\"obius-transformations.

In \secref{sec:weighted_delaunay_tessellations} and Appendix \ref{sec:spherical_wDt}
we will prove that hyperideally decorated piecewise spherical surfaces admit
unique weighted Delaunay tessellations. This allows us to introduce their
fundamental discrete conformal invariants $\surf_g$ which again characterize discrete
conformally equivalent surfaces. Decorated piecewise spherical surfaces are closely
connected to the geometry of \emph{hyperideal polyhedral cones} over $\surf_g$. These
are collections of hyperideal pyramids, \ie, hyperbolic pyramids with all vertices of the
base triangle lying \enquote{outside} of hyperbolic $3$-space. The spherical link
of their apex defines a decorated piecewise spherical surface
(\secref{sec:hyperbolic_polyhedra_spherical}). Furthermore, the convexity
of hyperideal polyhedral cones is equivalent to the weighted Delaunay property
(\secref{sec:convex_polyhedra}). Thus, the discrete uniformization of decorated
piecewise spherical surfaces is related to the rigidity of hyperideal polyhedra
studied by \textsc{I.~Rivin} \cite{Rivin1994a} and \textsc{J.-M.~Schlenker}
\cite{Schlenker1998}. This leads to our main theorem in the spherical case.

\begin{theorem}[discrete uniformization of decorated piecewise spherical surfaces]
    \label{thm:discrete_uniformization_spherical}
    For each hyperideally decorated piecewise spherical metric
    $(\dist_{\toposurf_0}, r)$ on a marked genus $0$ surface
    $(\toposurf_0, \verts)$ there is a unique
    discrete conformally equivalent decorated metric, up to M\"obius-transformations,
    realizing the uniform angle distribution
    $\Theta_i\equiv2\pi$.
\end{theorem}

Recently, \textsc{I.~Izmestiev} \etal\ \cite{IPW2023} solved the prescribed cone-angle
problem for undecorated convex piecewise spherical metrics on the topological sphere,
\ie, piecewise spherical metrics with all cone-angles less than $2\pi$. They used a
dimension-reduction technique similar to the one discussed in \cite{BI2008} to show
that the non-degeneracy of the spherical discrete Hilbert--Einstein functional
$\HE_{\surf_g,\Theta}^{+}$ (see \secref{sec:local_variational_principle}).
It seams possible to generalize their arguments to the decorated case.

\paragraph{Alexandrov type realizations.}
Decorated piecewise hyperbolic surfaces and their Delaunay tessellations are also
related to hyperbolic polyhedra (\secref{sec:hyperbolic_polyhedra_hyperbolic}).
A discrete subgroup $G$ of isometries of hyperbolic $3$-space is \emph{Fuchsian} if
it acts freely cocompactly on a totally geodesic plane and a hyperideal polyhedron
$P$ is \emph{invariant} under $G$ if $G(P) = P$. Furthermore, the polyhedron
\emph{realizes} a hyperbolic surface $\surf_g$ if $\partial P/G$ is isometric to
$\surf_g$. Hence, our \thmref{theorem:realisation_hyperbolic} provides
a new variational proof of the following

\begin{corollary}
    Let $(\toposurf_g, \verts)$ be a marked surface of genus $g\geq2$. Each
    complete hyperbolic metric on $(\toposurf_g, \verts)$ with cusps and
    complete ends of infinite area at $\verts$ can be realized as a unique
    convex hyperideal polyhedron, up to hyperbolic congruence, invariant under the
    action of a Fuchsian group.
\end{corollary}

A proof of this theorem first appeared in \textsc{F.~Fillastre}'s work
\cite{Fillastre2008}, where he uses the \emph{method of continuity}, by
\textsc{A.~Alexandrov} \cite{Alexandrov2005}. Note that this method is
not constructive. The special case of realizing hyperbolic cusp surfaces using
the discrete Hilbert--Einstein functional $\HE_{\surf_g,\Theta}^{-}$ was
discussed by \textsc{R.~Prosanov} \cite{Prosanov2020a}.

\paragraph{Geometric transitions.}
The definition of discrete conformal equivalence for different background geometries
raises the question of their relationship. A \emph{geometric transition} is a
continuous deformation of the geometric structure on a manifold. They play an
important role, \eg, in the proof of the orbifold theorem \cite{CHK2000}, or the
recent characterization of polyhedra inscribed in quadrics \cite{DMS2020},
extending \textsc{I.~Rivin's} characterization of polyhedra inscribed in the sphere.

In \secref{sec:geometric_transitions} we will show that the spaces of
decorated piecewise hyperbolic and spherical metrics with the same
fundamental invariant $\surf_g$ are incident to the same \emph{ideal boundary}, thus
connected. Their common ideal boundary corresponds to decorated piecewise Euclidean
surfaces with the fundamental invariant $\surf_g$. Furthermore, the discrete
Hilbert--Einstein functional $\HE_{\surf_g, \Theta}^{\pm}$ extends continuously
through the ideal boundary, providing an explicit means to construct geometric
transitions.

\paragraph{Circle packings.}
In \cite{BS2004} \textsc{P.~Bowers} and \textsc{K.~Stephenson} introduced inversive
distance circle packings and asked for their existence and uniqueness properties.
For Euclidean and hyperbolic background geometry, \textsc{R.~Guo} showed their
local \cite{Guo2011a}, and \textsc{F.~Luo} their global rigidity \cite{Luo2011}.
But fixed combinatorics pose an obstruction to existence. If all inversive
distances are greater than $1$, we can associate a complete hyperbolic surface
with ends to the circle packing (see
\lemref{lemma:lambda_length_to_inversive_distance_spherical} and
\lemref{lemma:lambda_length_to_inversive_distance_hyperbolic}). Hence,
\thmref{theorem:realisation_hyperbolic} and \thmref{theorem:geometric_transitions}
yield

\begin{corollary}
    Denote by $I\colon\edges_{\tri}\to(1,\infty)$ inversive distances on
    the edges $\edges_{\tri}$ of a triangulated marked surface
    $(\toposurf_g, \verts, \tri)$. Let $\Theta\in\RR^{\verts}$ either satisfying the
    hyperbolic \eqref{eq:hyperbolic_gauss_bonnet_condition}, or Euclidean
    \eqref{eq:euclidean_gauss_bonnet_condition} Gau{\ss}--Bonnet condition.
    There is a unique decorated piecewise hyperbolic, respectively Euclidean,
    metric on $(\toposurf_g, \verts)$ realizing $\Theta$ such that its
    fundamental discrete conformal invariant $\surf_g$ coincides with the
    hyperbolic surface induced by $I$. It can be obtained by maximizing the
    discrete Hilbert--Einstein functional $\HE_{\surf_g, \Theta}^{-}$.
\end{corollary}

The situation for inversive distance circle packings in the sphere is more complicated.
This is related to the non-concavity of the discrete Hilbert--Einstein functional
$\HE_{\surf_g,\Theta}^{+}$. Moreover, \textsc{J.~Ma} and \textsc{J.-M.~Schlenker}
showed that they are, in general, non-rigid \cite{MS2012}. Recently,
\textsc{J.~Bowers} \etal\ introduced subclasses of such spherical
circle packings which do not suffer this problem \cite{BBP2019}.

\paragraph{Discrete curvature problems.}
Characterizing which functions can arise as Gaussian curvatures of a
Riemannian metric on a given surface within some given conformal class
is a classical problem, solved by \textsc{J.~Kazdan} and \textsc{F.~Warner}
\cite{KW1974,KW1975}. In the literature the angle-defects
$\kappa_i \coloneq 2\pi-\theta_i$ of a discrete metric are often interpreted as
its discrete curvature. Here, $\theta_i$ denotes the cone-angle at the vertex $i$.
In \cite[App.~B]{BPS2015} it was discussed that $\kappa_i$
is rather the discrete analog of the curvature form than the Gaussian curvature.
This corresponds to the fact that the angle-defects $\kappa_i$ are scale-invariant
while the smooth Gaussian curvature is not. Thus, it is natural to ask whether there
are alternatives for the notion of discrete curvature.

In \cite{DalPozKourimska2023} \textsc{H.~Dal\,Poz\,Kou\v{r}imsk\'a} studied the
quotients of angle-defects by a choice of \emph{dual areas} associated to
vertices, \ie, the areas of Voronoi cells dual to the Delaunay tessellations.
As weighted Delaunay tessellations also possess dual decompositions, \ie, weighted
Voronoi decompositions, this notion
directly translates to decorated surfaces \cite{XZ2023}. Another option are
\textsc{H.~Ge}'s and \textsc{X.~Xu}'s \emph{$\alpha$-curvatures}. They
are given by the quotient of the angle-defects by some power of radii associated
to the vertices. Initially defined for circle packings \cite{GX2021},
$\alpha$-curvatures have recently been extended to decorated surfaces
\cite{XZ2023a}. Both notions support existence and, to some extend, uniqueness
results of Kazdan--Warner type. They are closely related to
\thmref{theorem:realisation_hyperbolic} and \cite[Thm.~A]{BL2023}.

\newpage

\section{Decorated discrete conformal equivalence and hyperbolic polyhedra}
\label{sec:ddce_and_polyhedra}
\subsection{Decorated piecewise spherical and piecewise hyperbolic surfaces}
\label{sec:basic_definitions}
By a \emph{marked surface} $(\toposurf_g, \verts)$ we mean a closed oriented
$2$-dimensional manifold $\toposurf_g$ of genus $g$ with a finite set
$\verts\subset\toposurf_g$, the \emph{marking}. A \emph{tessellation} of the marked
surface $(\toposurf_g, \verts)$ is a cell-complex $\tri$ which is homeomorphic to
$\toposurf_g$ with $0$-cells given by $\verts$. The tessellation $\tri$ is called a
\emph{triangulation} if all $2$-cells are incident to exactly three $1$-cells.
We also refer to the $0$-cells as \emph{vertices}, the $1$-cells as \emph{edges}, and
the $2$-cells as \emph{faces}. The set of edges and faces is denoted by
$\edges_{\tri}$ and $\faces_{\tri}$, respectively. If the tessellation is clear from
the context, we sometimes simplify this to $\edges$ and $\faces$.

There is no need that a triangulation $\tri$ is
a simplicial complex, \ie, self-glueings of a single triangle or glueing two
triangles along multiple edges is allowed. Indeed, it is essential for our
constructions in \secref{sec:weighted_delaunay_tessellations} and
\secref{sec:canonical_tessellations} that we work with the general
class of triangulations. Still, we adopt the following notation: we denote by $ij$
the edge with vertices $i$ and $j$, by $ijk$ the face incident to the vertices $i$,
$j$, and $k$, and by $\corner{i}{jk}$ the corner at vertex $i$ in the triangle $ijk$.
This notation is simple but only unambiguous if $\tri$ is a simplicial complex.
Nonetheless, we think that the risk of confusion is small compared to the gain of
conciseness of our expositions.

\begin{remark}
   Most of our considerations will start with a single triangle or
   a quadrilateral given by two adjacent triangles. After prescribing a discrete
   metric (see below), these can always be realized in the (metric) $2$-sphere or
   hyperbolic plane, respectively. There our notation is unambiguous. The
   quantities we obtain in this way can then be projected back to the surface.
   Thus, the problem with the adopted notation amounts \enquote{merely} to an exercise
   in bookkeeping.
\end{remark}

A \emph{discrete hyperbolic metric} $(\tri, \len)$ of the marked surface
$(\toposurf_g, \verts)$ is a triangulation $\tri$ together with a positive
function $\len\colon\edges_{\tri}\to\RR_{>0}$ such that the triangle
inequalities are satisfied for each triangle $ijk\in\faces_{\tri}$, \ie,
\begin{equation}\label{eq:triangle_inequalities}
    \len_{ij}\,+\,\len_{jk} \;>\; \len_{ki}
\end{equation}
for all cyclic permutations of $(i, j, k)$. Similarly, a
\emph{discrete spherical metric} is determined by a function
$\len\colon\edges_{\tri}\to(0,\pi)$ satisfying the triangle inequalities
\eqref{eq:triangle_inequalities} for each triangle. Additionally, we require for
a discrete spherical metric that the perimeter of each triangle $ijk\in\faces_{\tri}$
is smaller than $2\pi$, \ie,
\begin{equation}\label{eq:perimeter_condition}
    \len_{ij}\,+\,\len_{jk}\,+\,\len_{ki} \;<\; 2\pi.
\end{equation}

If we are given a discrete spherical/hyperbolic metric $\len$, each triangle
$ijk\in\faces_{\tri}$ can be uniquely realized, up to isometries, as a
convex hyperbolic or spherical triangle, respectively. Here, a triangle in the
$2$-sphere is said to be convex if it is contained in an open half-sphere.
In the hyperbolic plane any triangle is convex. The identifications of these
triangles along edges are isometries. Thus, prescribing $(\tri, \len)$ endows the
marked surface $(\toposurf_g, \verts)$ with
a complete path metric $\dist_{\toposurf_g}$ (see, \eg, \cite{CHK2000}). We call
$\dist_{\toposurf_g}$ a \emph{piecewise spherical/hyperbolic metric}. With respect
to this metric each point of $\toposurf_g\setminus\verts$ is locally isometric to
the hyperbolic plane or $2$-sphere, respectively. At the vertices $\verts$ the metric
$\dist_{\toposurf_g}$ may have cone-like singularities
$\theta\colon\verts\to\RR_{>0}$. Indeed, we can compute
the angle $\theta_{jk}^i$ at each corner abutting the vertex $i\in\verts$, \eg, using
the spherical/hyperbolic law of cosines. Then the \emph{cone-angle} $\theta_i$ at $i$
is given by sum of these angles $\theta_{jk}^i$. These cone-angles may be different
than $2\pi$.

\begin{remark}
    In this article we only consider piecewise spherical/hyperbolic metrics
    which can be obtained from a discrete spherical/hyperbolic metric.
    For piecewise hyperbolic metrics this is always the case. Yet, in the spherical
    case there are piecewise spherical metrics with no discrete counterpart. For
    example, we can obtain the standard (smooth) metric on the $2$-sphere by gluing
    a convex spherical triangle and its complement in the $2$-sphere. To this
    construction corresponds no discrete spherical metric on a genus $0$ surface
    $\toposurf_0$ with three marked points.
\end{remark}

\subsection{Local geometry in the spherical case}
\label{sec:local_spherical_case}
Consider a single convex spherical triangle $ijk$ determined by the discrete metric
$(\tri, \len)$. It can be realized in the (metric) $2$-sphere
\begin{equation}
    \SS^2
    \;\coloneq\;
    \left\{x\in\RR^3 \;:\; 1=\|x\|^2=x_0^2+x_1^2+x_2^2\right\}.
\end{equation}
A \emph{decoration} of this triangle is a choice of circle about each vertex.
We refer to the circles as \emph{vertex-circles}. They are determined by their radii
$r_i\in[0,\nicefrac{\pi}{2})$ at the vertices.

The conformal automorphisms of the $2$-sphere are \emph{M\"obius transformations}.
They can be identified with the elements of $\SO^{+}(3,1)$, \ie, the
identity component of the isometry group of the Minkowski metric
\begin{equation}
    \ip{X}{Y}_{3,1} \;=\; x_0y_0+x_1y_1+x_2y_2-x_3y_3.
\end{equation}
M\"obius transformations act bijectively on the set of points $\SS^2$
and the set of spherical circles and geodesics (great circles). Indeed, given a
spherical circle with center $p\in\SS^2$ and radius $r\in[0,\nicefrac{\pi}{2})$ we
can identify it with
\begin{equation}\label{eq:spherical_moebius_lift}
    C \,=\, (p, \cos r) \in \RR^{3,1}.
\end{equation}
Its Minkowski norm $\|C\|_{3,1}^2 = 1-\cos^2(r)$ is $>0$ for spherical
circles and $=0$ for points. Thus, we obtain an identification, up to scale, of
circles and points with elements of $\{\|X\|_{3,1}^2>0\}$ and
$\{\|X\|_{3,1}^2=0\}$, respectively. For more information on M\"obius
geometry and this identification we refer the reader to
\cite{Cecil1992, Hertrich-Jeromin2003}.

We say that two decorated spherical triangles are \emph{M\"obius-equivalent} if there
is a M\"obius transformation mapping the vertex-circles of the first triangle
to the corresponding vertex-circles of the second triangle. Not every pair of
decorated spherical triangles is M\"obius-equivalent because
a pair of circles cannot be mapped arbitrarily to two other circles by a M\"obius
transformation. For a pair of spherical circles
\begin{equation}\label{eq:spherical_inversive_distance}
    I_{ij}^{+}
    \;\coloneq\;
    \frac{\cos r_i \cos r_j - \cos\len_{ij}}{\sin r_i \sin r_j}
\end{equation}
is a M\"obius-geometric invariant. It is the spherical version of the
\emph{inversive distance} (see the following lemma and, \eg, \cite{Coxeter1966a}).

\begin{lemma}\label{lemma:spherical_moebius_and_inversive_distance}
   Consider two decorated spherical triangles given by $(\len, r)$ and
   $(\tilde{\len}, \tilde{r})$, respectively. Suppose that $r_i>0$ and
   $\tilde{r}_i>0$ for all $i\in\{1, 2, 3\}$. They are M\"obius-equivalent
   if and only if the spherical inversive distances
   \eqref{eq:spherical_inversive_distance}
   for corresponding edges coincide.
\end{lemma}
\begin{proof}
    Let $C_i,\tilde{C}_i\in\RR^{3,1}$ be the representatives for the vertex-circles
    according to \eqref{eq:spherical_moebius_lift}. Finding a M\"obius transformation
    mapping the first decorated triangle to the second is equivalent to
    finding an element of $\SO^{+}(3,1)$ that maps $C_i$ to $\tilde{C}_i$ up to scale.
    That is, we are looking for $u_i\in\RR$ such that
    \begin{equation}
       \big\|\tilde{C}_i\big\|_{3,1}^2
       \;=\;
       \big\|\ee^{u_i}C_i\big\|_{3,1}^2
       \qquad\text{and}\qquad
       \ip{\tilde{C}_i}{\tilde{C}_j}_{3,1}
       \;=\;
       \ip{\ee^{u_i}C_i}{\ee^{u_j}C_j}_{3,1}.
    \end{equation}
    From elementary spherical geometry we see that
    \begin{equation}
       \|C_i\|_{3,1}^2
       \;=\;
       \sin^2 r_i
       \qquad\text{and}\qquad
       \ip{C_i}{C_j}_{3,1}
       \;=\;
       \cos\len_{ij}- \cos r_i \cos r_j.
    \end{equation}
    Hence, the first equality gives
    $\ee^{u_i}=\nicefrac{\|\tilde{C}_i\|_{3,1}}{\|C_i\|_{3,1}}$
    and the second equality shows that the compatibility conditions are
    \begin{equation}\label{eq:inversive_distances_via_inner_product}
       -I_{ij}^{+}
       \;=\;
       \frac{\ip{C_i}{C_j}_{3,1}}
       {\|C_i\|_{3,1}\|C_j\|_{3,1}}
       \;=\;
       \frac{\ip{\tilde{C}_i}{\tilde{C}_j}_{3,1}}
       {\|\tilde{C}_i\|_{3,1}\|\tilde{C}_j\|_{3,1}}
       \;=\;
       -\tilde{I}_{ij}^{+}.\qedhere
    \end{equation}
\end{proof}

The $u_i$ introduced in the previous proof are called
\emph{(discrete) logarithmic scale factors}. To extend the previous
result to decorated spherical triangles with vanishing vertex-radii we
consider $\epsilon$-families $(\len^{\epsilon}, r^{\epsilon})$ of decorated
triangles. We assume that the length and radii depend smoothly on the real
parameter $\epsilon$. A decorated spherical triangle with an
\emph{infinitesimal circle at the vertex $i$}
is an $\epsilon$-family $(\len^{\epsilon}, r^{\epsilon})$ of decorated
spherical triangles such that
\begin{equation}
   r_i^{\epsilon} \;=\; \epsilon R_i \,+\, \mathrm{o}(\epsilon).
\end{equation}
Here, $R_i>0$ and $\mathrm{o}$ is a function satisfying
$\lim_{\epsilon\to0}\nicefrac{\mathrm{o}(\epsilon)}{\epsilon}=0$.

\begin{lemma}\label{lemma:local_moebius_to_factors_spherical}
    Let $(\len, r)$ and $(\tilde{\len}, \tilde{r})$ be two decorated spherical
    triangles. They are M\"obius-equivalent if and only if
    there are $u_i\in\RR$ for $i\in\{1,2,3\}$ such that
    \begin{align}
        \sin\tilde{r}_i
        &\;=\; \ee^{u_i}\sin r_i,\\
        \cos\tilde{\len}_{ij}
        &\;=\; \ee^{u_i+u_j}
                  \big(\!
                  \cos \len_{ij}
                  - \cos r_i\cos r_j
                  \big)\\
        &\phantom{xxxxx}
                  +\sqrt{\big(1-\ee^{2u_i}\sin^2r_i\big)
                      \big(1-\ee^{2u_j}\sin^2r_j\big)}
    \end{align}
    for all edges $ij$. In particular, the $u_i$ can be computed via
    \begin{equation}\label{eq:radii_to_log_factors_spherical}
        \ee^{u_i}
        \;=\;
        \begin{cases}
           \sin\tilde{r}_i/\sin r_i &,\text{if }r_i\neq0,\\
           |M'(p_i)| &,\text{if }r_i=0.
        \end{cases}
    \end{equation}
    Here, $M'$ is the derivative of $M$, and $p_i\in\SS^2$ represents
    the corresponding vertex of the first triangle.
\end{lemma}
\begin{proof}
    This lemma follows from the considerations in
    the proof of \lemref{lemma:spherical_moebius_and_inversive_distance}.
    The only problem is that, in general, we cannot compute the logarithmic scale
    factors as quotients of Minkowski norms. As an alternative we can use
    \begin{equation}
        \label{eq:triangle_conformal_factors}
        \ee^{2u_i}
        \;=\;
        \frac{\ip{\tilde{C}_i}{\tilde{C}_j}_{3,1}}{\ip{C_i}{C_j}_{3,1}}
        \frac{\ip{C_j}{C_k}_{3,1}}{\ip{\tilde{C}_j}{\tilde{C}_k}_{3,1}}
        \frac{\ip{\tilde{C}_k}{\tilde{C}_i}_{3,1}}{\ip{C_k}{C_i}_{3,1}}
    \end{equation}
    which follows from the conditions on the mixed inner products.

    It is only left to show that $u_i=\log|M'(p_i)|$ if $r_i=0$. Consider
    an infinitesimal circle $(\len^{\epsilon}, r^{\epsilon})$ at the vertex $i$.
    Applying $M$ to this $\epsilon$-family, we obtain an $\epsilon$-family
    $(\tilde{\len}^{\epsilon}, \tilde{r}^{\epsilon})$ at $i$ with
    $(\tilde{\len}^0, \tilde{r}^0) = (\tilde{\len}, \tilde{r})$.
    Since $M$ is a conformal map we have
   \[
      \lim_{\epsilon\to0}\, \frac{\tilde{r}_i^{\epsilon}}{r_i^{\epsilon}}
      \;=\;
      |M'(p_i)|.
      \qedhere
   \]
\end{proof}

\subsection{Local geometry in the hyperbolic case}
In this section we are considering a single hyperbolic triangle $ijk$
determined by the discrete metric $(\tri, \len)$. It can be realized in
the hyperbolic plane which we represent by
\begin{equation}
    \HH^2
    \;\coloneq\;
    \left\{x\in\RR^{2,1} \;:\; 1=\|x\|_{2,1}^2=x_1^2+x_2^2-x_3^2\right\},
\end{equation}
using the hyperboloid model.
As in the spherical case a \emph{decoration} of this triangle is a choice of circle
about each vertex. This time they are determined by radii $r_i\in\RR_{\geq0}$.

M\"obius transformations act also on hyperbolic circles since there are models
of the hyperbolic plane embedded into $\SS^2$, \eg, the upper half-sphere model.
A hyperbolic circle with center $p\in\HH^2$ and radius $r\in\RR_{\geq0}$ can be
lifted to Minkowski $4$-space using
\begin{equation}\label{eq:hyperbolic_moebius_lift}
    C \,=\, (\cosh r, p)\in\RR^{3,1}.
\end{equation}
Again, the Minkowski norm $\|C\|_{3,1}^2 = \cosh^2(r) - 1$ is $>0$ for hyperbolic
circles and $=0$ for points. This relation is preserved by the action
of $\SO^{+}(3,1)$. More information can be found in \cite{Cecil1992,Lutz2023}.

M\"obius-equivalence of decorated hyperbolic triangles is defined as for
their spherical counterparts. The hyperbolic version of the inversive distance
is
\begin{equation}\label{eq:hyperbolic_inversive_distance}
    I_{ij}^{-}
    \;\coloneq\;
    \frac{\cosh\len_{ij} - \cosh r_i \cosh r_j}{\sinh r_i \sinh r_j}.
\end{equation}
It is also a M\"obius-geometric invariant.

\begin{lemma}\label{lemma:hyperbolic_moebius_and_inversive_distance}
    Consider two decorated hyperbolic triangles given by $(\len, r)$ and
    $(\tilde{\len}, \tilde{r})$, respectively. Suppose that $r_i>0$ and
    $\tilde{r}_i>0$ for all $i\in\{1, 2, 3\}$. They are M\"obius-equivalent
    if and only if the hyperbolic inversive distances
    \eqref{eq:hyperbolic_inversive_distance}
    for corresponding edges coincide.
\end{lemma}
\begin{proof}
    This is proved in the same way as
    \lemref{lemma:hyperbolic_moebius_and_inversive_distance}. This time
    we are using \eqref{eq:hyperbolic_moebius_lift} to obtain representatives
    of the circles in $\RR^{3,1}$. In particular,
    \begin{equation}
       \|C_i\|_{3,1}^2
       \;=\;
       \sinh^2 r_i
       \qquad\text{and}\qquad
       \ip{C_i}{C_j}_{3,1}
       \;=\;
       \cosh r_i \cosh r_j - \cosh\len_{ij}.\qedhere
    \end{equation}
\end{proof}

We can define logarithmic scale factors and infinitesimal circles for
decorated hyperbolic triangles similar to the spherical case.
Hence, we can again extend our considerations to decorated hyperbolic
triangles with vanishing vertex-radii.

\begin{lemma}\label{lemma:local_moebius_to_factors_hyperbolic}
   Given two decorated hyperbolic triangles $(\len, r)$ and
   $(\tilde{\len}, \tilde{r})$. They are M\"obius-equivalent if and only if
   there are $u_i\in\RR$ for $i\in\{1,2,3\}$ such that
   \begin{align}
      \sinh\tilde{r}_i
      &\;=\; \ee^{u_i}\sinh r_i,\\
      \cosh\tilde{\len}_{ij}
      &\;=\; \ee^{u_i+u_j}
                \big(\!
                \cosh \len_{ij}
                - \cosh r_i\cosh r_j
                \big)\\
      &\phantom{xxxxx}
              +\sqrt{\big(1+\ee^{2u_i}\sinh^2r_i\big)
                    \big(1+\ee^{2u_j}\sinh^2r_j\big)}
   \end{align}
   for all edges $ij$. In particular, the $u_i$ can be computed via
   \begin{equation}\label{eq:radii_to_log_factors_hyperbolic}
      \ee^{u_i}
      \;=\;
      \begin{cases}
         \sinh\tilde{r}_i/\sinh r_i &,\text{if }r_i\neq0,\\
         |M'(p_i)| &,\text{if }r_i=0.
      \end{cases}
   \end{equation}
   Here, $M'$ is the derivative of $M$, and $p_i\in\HH^2$ represents
   the corresponding vertex of the first triangle.
\end{lemma}
\begin{proof}
    The proof is the same as for \lemref{lemma:local_moebius_to_factors_spherical}.
\end{proof}

\subsection{Decorated discrete conformal equivalence}
\label{sec:ddce}
The notion of decoration discussed in the previous sections extends to
piecewise spherical/hyperbolic metrics: a \emph{decoration} of
$(\toposurf_g, \verts, \dist_{\toposurf_g})$ with geodesic triangulation $\tri$
is a choice of decoration of each triangle such that it is consistent along
edges of pairs of neighbouring faces. A decoration is determined by the radii
\begin{equation}
    r\colon\verts\to\begin{cases}
        [0, \nicefrac{\pi}{2}) &, \text{ spherical case,}\\
        \RR_{\geq0} &, \text{ hyperbolic case}.
    \end{cases}
\end{equation}
Thus, it only depends on the metric $\dist_{\toposurf_g}$ and not on the
triangulation $\tri$. We call a decoration \emph{hyperideal} if no pair of
vertex-circles intersects. This is equivalent to
\begin{equation}\label{eq:hyperideal_decoration_condition}
    r_i+r_j \,<\, \dist_{\toposurf_g}(i, j)
\end{equation}
for all $(i,j)\in\verts^2$. Here, $\dist_{\toposurf_g}(i, i)$ is the length of
a smallest geodesic loop starting and ending in $i$. In the following sections
we will only consider hyperideal decorations. So we might drop the
\enquote{hyperideal}. We call the pair $(\dist_{\toposurf_g}, r)$ a
\emph{decorated piecewise hyperbolic/spherical metric}. It is determined by the
\emph{decorated discrete hyperbolic/spherical metric} $(\tri, \len, r)$
if a triangulation is prescribed.

\begin{definition}[DCE via M\"obius transformations]
    \label{def:dce_via_moebius}
    Let $\tri$ be a triangulation of the marked surface $(\toposurf_g, \verts)$.
    Two decorated discrete hyperbolic/spherical metrics $(\tri, \len, r)$ and
    $(\tri, \tilde{\len}, \tilde{r})$ are called
    \emph{discretely conformally equivalent (DCE)} if
    \begin{enumerate}[label=(\roman*)]
        \item
            for each face $ijk\in\faces_{\tri}$ the decorated triangles
            $(\len|_{ijk},r|_{ijk})$ and $(\tilde{\len}|_{ijk}, \tilde{r}|_{ijk})$
            are M\"obius-equivalent, and
        \item
            for each $i\in\verts$ with $r_i=0$ and adjacent faces
            $ijk, ijl\in\faces_{\tri}$ the M\"obius transformations satisfy
            \[
                |M_{ijk}'(p_i)| \;=\; |M_{ijl}'(p_i)|.
            \]
    \end{enumerate}
\end{definition}

The definitions of logarithmic scale factors and infinitesimal circles naturally
extend to decorated piecewise hyperbolic/spherical metrics. Hence,
\lemref{lemma:local_moebius_to_factors_spherical} and
\lemref{lemma:local_moebius_to_factors_hyperbolic}
lead to the following alternative definitions of discrete conformal equivalence.

\begin{propdef}[DCE via scale factors --- spherical case]
   \label{prop:conformal_change_spherical}
   Let $\tri$ be a triangulation of the marked surface $(\toposurf_g, \verts)$.
   Two decorated discrete spherical metrics $(\tri, \len, r)$ and
   $(\tri, \tilde{\len}, \tilde{r})$ are discrete conformally equivalent if and
   only if there is a $u\in\RR^{\verts}$ such that
   \begin{equation}\label{eq:conformal_change_formulas_spherical}
      \begin{aligned}
      \sin\tilde{r}_i
      &\;=\; \ee^{u_i}\sin r_i\\
      \cos\tilde{\len}_{ij}
      &\;=\; \ee^{u_i+u_j}
                \big(\!
                \cos\len_{ij}
                - \cos r_i\cos r_j
                \big)\\
      &\phantom{xxxxx}
                +\sqrt{\big(1-\ee^{2u_i}\sin^2r_i\big)
                    \big(1-\ee^{2u_j}\sin^2r_j\big)}
      \end{aligned}
   \end{equation}
   for all $ij\in\edges_{\tri}$.
\end{propdef}

\begin{remark}
    From standard trigonometry we know that $1-\cos(x) = 2\sin^2(\nicefrac{x}{2})$.
    Thus, if all vertex-circles vanish, \ie, $r_i\equiv0$, the equations
    \eqref{eq:conformal_change_formulas_spherical} give the
    \enquote{classical} formulation of discrete conformal
    equivalence in the spherical case \eqref{eq:classical_spherical_dce}.
\end{remark}

\begin{propdef}[DCE via scale factors --- hyperbolic case]
   \label{prop:conformal_change_hyperbolic}
   Let $\tri$ be a triangulation of the marked surface $(\toposurf_g, \verts)$.
   Two decorated discrete hyperbolic metrics $(\tri, \len, r)$ and
   $(\tri, \tilde{\len}, \tilde{r})$ are discrete conformally equivalent if and only
   if there is a $u\in\RR^{\verts}$ such that
   \begin{equation}\label{eq:conformal_change_formulas_hyperbolic}
      \begin{aligned}
      \sinh\tilde{r}_i
      &\;=\; \ee^{u_i}\sinh r_i\\
      \cosh\tilde{\len}_{ij}
      &\;=\; \ee^{u_i+u_j}
                \big(\!
                \cosh\len_{ij}
                - \cosh r_i\cosh r_j
                \big)\\
      &\phantom{xxxxx}
              +\sqrt{\big(1+\ee^{2u_i}\sinh^2r_i\big)
                    \big(1+\ee^{2u_j}\sinh^2r_j\big)}
      \end{aligned}
   \end{equation}
   for all $ij\in\edges_{\tri}$.
\end{propdef}

\begin{remark}
    For hyperbolic functions we have the identity
    $\cosh(x) - 1 = 2\sinh^2(\nicefrac{x}{2})$. Hence, if all vertex-circles vanish,
    \ie, $r_i\equiv0$, the equations
    \eqref{eq:conformal_change_formulas_hyperbolic} yield the
    \enquote{classical} formulation of discrete conformal
    equivalence in the hyperbolic case \eqref{eq:classical_hyperbolic_dce}.
\end{remark}

\subsection{Connections to hyperbolic polyhedra}
\label{sec:connections_hyperbolic_polyhedra}
\subsubsection{The spherical case}
\label{sec:hyperbolic_polyhedra_spherical}
For a hyperideally decorated convex spherical triangle $ijk$ there is a unique
spherical circle $C_{ijk}$ which is simultaneously orthogonal to all vertex-circles
of the triangle. We call it the \emph{face-circle} of the decorated triangle $ijk$.
If we interpret the disk bounded
by $C_{ijk}$ as the hyperbolic plane, modeled in the Poincar\'e disk model, then
the triangle $ijk$ becomes a \emph{hyperideal triangle}. The vertices outside the
hyperbolic plane are called \emph{hyperideal} ($r_i>0$) and those on the
circle $C_{ijk}$ are \emph{ideal vertices} ($r_i=0$). Thus, for a hyperideally
decorated discrete spherical metric $(\tri, \len, r)$ we obtain a hyperbolic metric
on each (abstract) triangle. They fit together along the edges, giving a
complete hyperbolic surface with ends $\surf_g$ which is homeomorphic to
$\toposurf_g\setminus\verts$. This surface is equipped with a geodesic
triangulation combinatorially equivalent to $\tri$. The ends of finite area
(\emph{cusps}) and infinite area of $\surf_g$ correspond to ideal and hyperideal
vertices, respectively.

\begin{definition}[fundamental discrete conformal invariant --- spherical case]
   \label{def:fundamental_discrete_conformal_invariant_spherical}
   Let $(\tri, \len, r)$ be a hyperideally decorated discrete spherical metric.
   We call the complete hyperbolic surface $\surf_g$ constructed above its
   \emph{(spherical) fundamental discrete conformal invariant}.
\end{definition}

Let us show that $\surf_g$ is indeed an invariant under discrete conformal
equivalence. Towards this end interpret $\SS^2$ as the boundary at infinity
(\emph{ideal boundary}) of hyperbolic $3$-space in the Poincar\'e ball model, \ie,
$\HH^3_{\mathrm{B}} = \{v\in\RR^3:\|v\|<1\}$.
To each spherical circle and geodesic correspond a totally geodesic plane in
hyperbolic $3$-space (see, \eg, \cite[\S4.5]{Ratcliffe1994}). In the former case
it is represented by the intersection of $\HH^3_{\mathrm{B}}$ with the unique
sphere which intersects $\SS^2$ orthogonally
in this circle. In the latter case it is given by the intersection of
$\HH^3_{\mathrm{B}}$ with (euclidean) planes passing through the origin $0\in\RR^3$.
Now, the hyperbolic planes over the edges of $ijk$ together with the plane over
$C_{ijk}$ bound a \emph{semi-hyperideal tetrahedron} with one distinguished vertex
at $0$. We also call it a \emph{hyperideal pyramid} to emphasise that
we distinguished a vertex in hyperbolic $3$-space (see
\figref{fig:hyperideal_pyramid}). The \emph{base face} of this pyramid, \ie,
the face which is not incident to $0$, is a hyperideal triangle. It is related to
the previously considered hyperideal triangle via a sphere-inversion.
In fact, using an appropriate
normalization, this is the well-known relationship between the Poincar\'e disk model
and the upper half-sphere model of the hyperbolic plane via stereographic projection
(see, \eg, \cite[Sec.~7]{CFK+1997}). So both constructions induce the same hyperbolic
metric on the triangle $ijk$, and thus on $\toposurf_g\setminus\verts$.

\begin{figure}[t]
    \centering
    \includegraphics[width=0.6\textwidth]{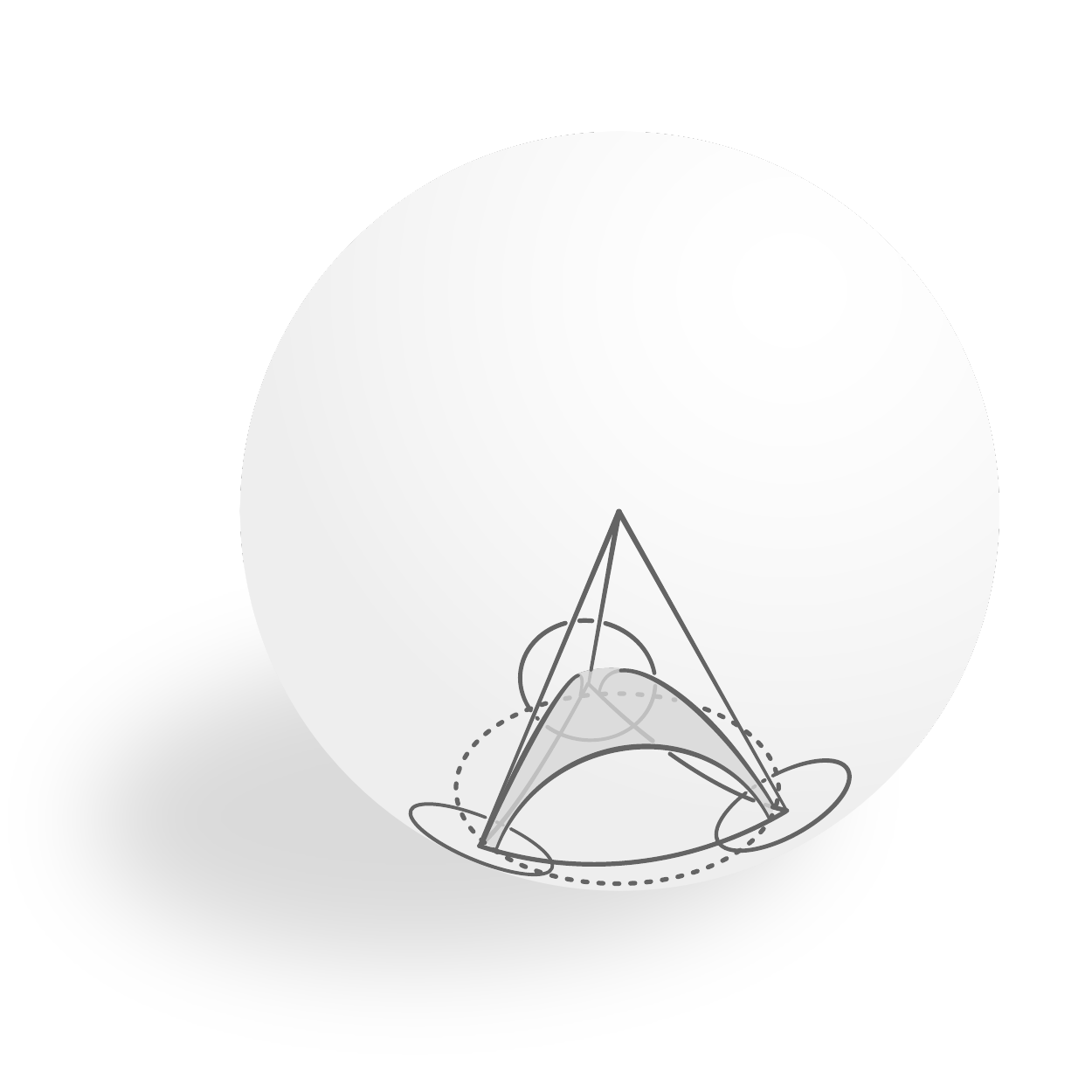}
    \caption{%
        Sketch of a hyperideal pyramid over a decorated convex spherical triangle.
        The hyperideal triangle giving the base face of the pyramid is shaded.
        Note that the dihedral angles at the edges incident with distinguished vertex
        $0$ coincide with the interior angles $\theta_{ij}^k$ of the spherical triangle.
    }
    \label{fig:hyperideal_pyramid}
\end{figure}

M\"obius-trasformations of $\SS^2$ are in one-to-one correspondence to the
isometries of hyperbolic $3$-space (see, \eg, \cite[Thm.~4.5.2]{Ratcliffe1994}).
Because discrete conformal equivalence is defined via M\"obius-transformations
acting on each triangle $ijk$ (\defref{def:dce_via_moebius}) we obtain

\begin{proposition}
    \label{prop:dce_via_invariant_spherical}
    Two hyperideally decorated discrete spherical metrics on the same
    triangulation $\tri$ of the
    marked surface $(\toposurf_g, \verts)$ are discrete conformally equivalent
    if and only if they induce the same fundamental discrete conformal invariant.
\end{proposition}

The fundamental discrete conformal invariant $\surf_g$ of a decorated discrete
spherical metric $(\tri, \len, r)$ is characterized by the \emph{$\lambda$-lengths}
of its geodesic triangulation $\tri$. They are given by the \emph{(truncated) lengths}
$\lambda\colon\edges_{\tri}\to\RR$ of its edges. If both vertices of an edge
$ij$ are hyperideal, then $\lambda_{ij}>0$ is the distance between the hyperbolic
planes associated to the vertices (see \figref{fig:notation_hyperideal_pyramid}, left).
Should one of the vertices be ideal instead, we equip it with an
\emph{auxiliary horosphere}. Now, the distance $\lambda_{ij}$ is measured relative to
this horosphere. It is positive if the horosphere is disjoint with the hyperbolic
plane, or horosphere, at the other vertex and negative otherwise (see
\figref{fig:notation_hyperideal_pyramid}, right). In the first case the $\lambda_{ij}$
is related to the inversive distance \eqref{eq:spherical_inversive_distance}
of the edge $ij$.

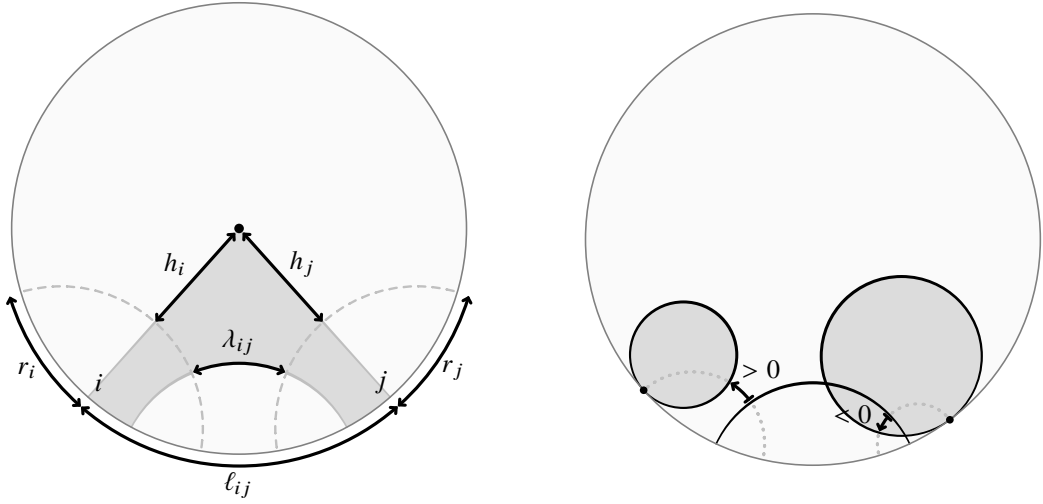
\begin{figure}[t]
    \centering
    \labellist
    \small\hair 2pt
    \pinlabel $i$ at 125 120
    \pinlabel $j$ at 480 120
    \pinlabel $h_i$ at 220 270
    \pinlabel $h_j$ at 380 270
    \pinlabel $\lambda_{ij}$ at 300 170
    \pinlabel $\len_{ij}$ at 300 -10
    \pinlabel $r_i$ at 35 135
    \pinlabel $r_j$ at 570 135
    \endlabellist
    \includegraphics[width=0.45\textwidth]{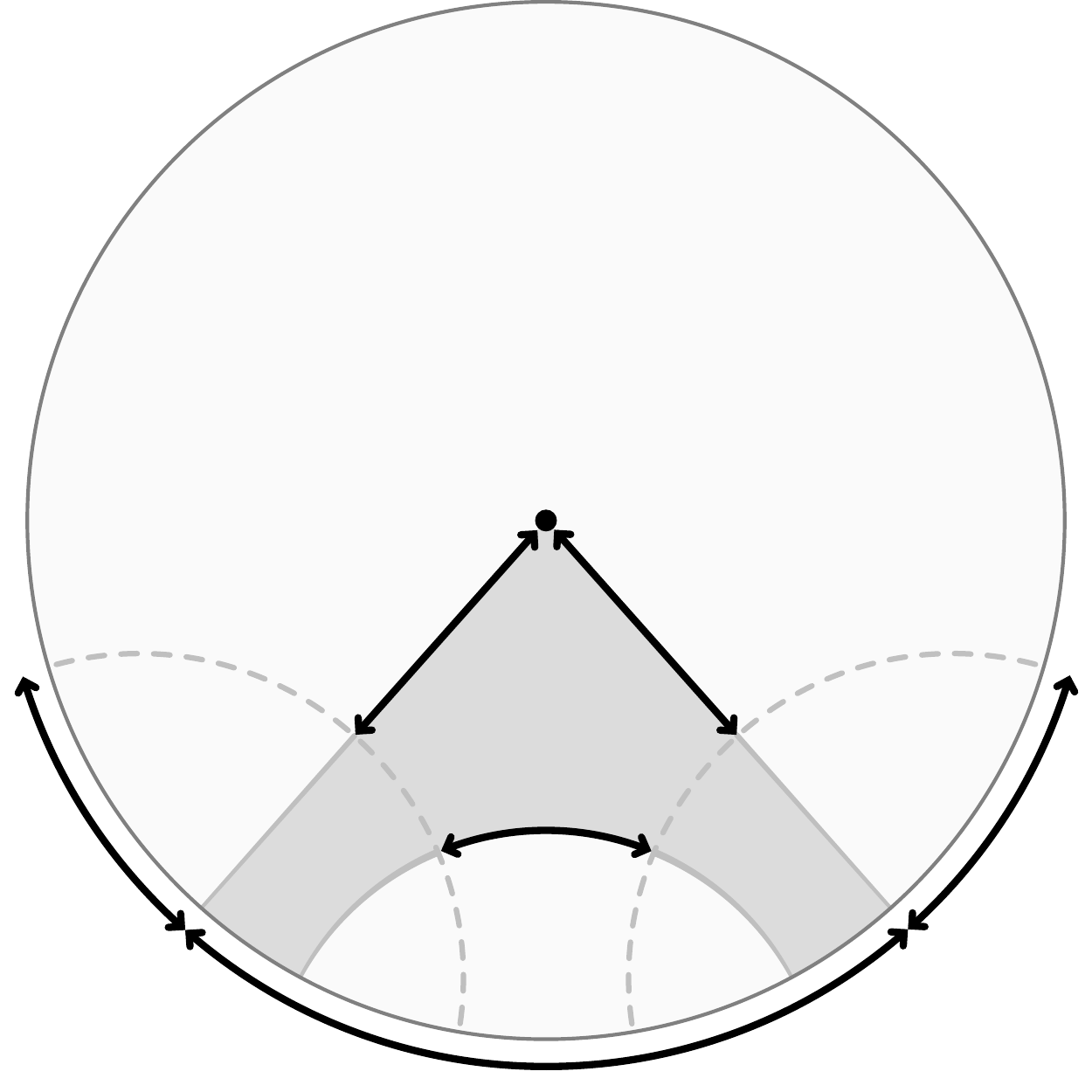}
    \quad\qquad
    \labellist
    \small\hair 2pt
    \pinlabel $>0$ at 235 137
    \pinlabel $<0$ at 350 80
    \endlabellist
    \includegraphics[width=0.45\textwidth]{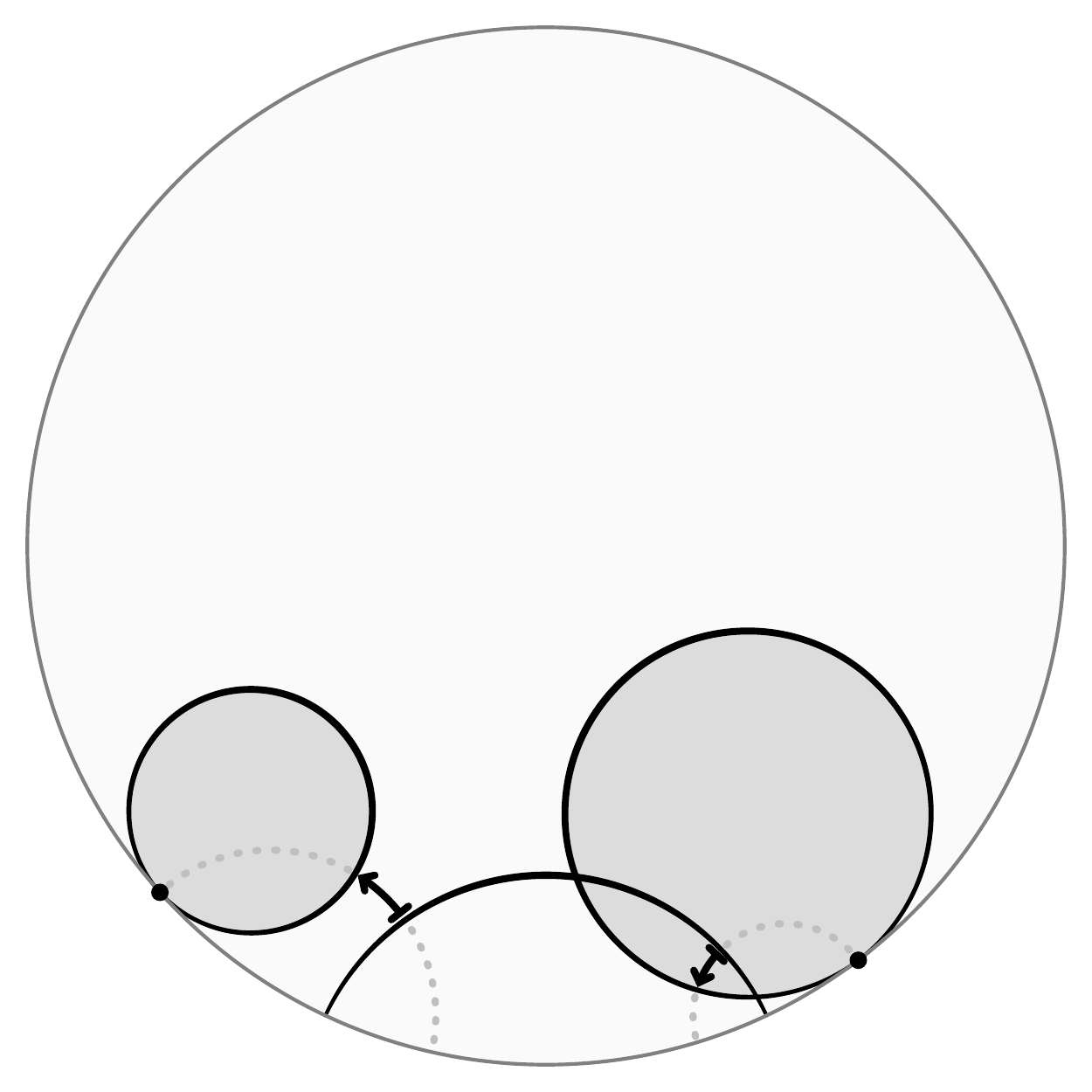}
    \caption{%
        \textsc{Left}: notation for hyperideal pyramids.
        \textsc{Right}: the hyperbolic distance between a horocycle and a hyperbolic
        line is defined to be positive if the horocycle and line do not intersect
        and negative otherwise. The higher dimensional analog is defined similarly.
    }
    \label{fig:notation_hyperideal_pyramid}
\end{figure}

\begin{lemma}\label{lemma:lambda_length_to_inversive_distance_spherical}
    Let $(\tri, \len, r)$ be a hyperideally decorated discrete spherical metric.
    Consider an edge $ij\in\edges_{\tri}$ and suppose that both $r_i>0$ and
    $r_j>0$. Then
    \begin{equation}\label{eq:lambda_length_to_inversive_distance_spherical}
        I_{ij}^{+} \;=\; \cosh\lambda_{ij}.
    \end{equation}
\end{lemma}

Before we prove this lemma, let us introduce a bit of notation.
Denote by $h\colon\verts\to\RR$ the \emph{heights} of the hyperideal pyramids
associated to $(\tri, \len, r)$, \ie, the (truncated) lengths of the edges incident
to the distinguished vertex $0$. If the vertex is ideal, we again compute the
length with respect to an auxiliary horosphere. Note that we make the same
choice of horosphere to determine the $\lambda$-lengths and the heights. Now,
the identity \eqref{eq:lambda_length_to_inversive_distance_spherical}
follows from a direct computation.

\begin{proof}[Proof of Lemma \ref{lemma:lambda_length_to_inversive_distance_spherical}]
    Using the relationship between the Poincar\'e disk model and the hyperboloid
    model of the hyperbolic plane (see, \eg, \cite[\S4.5]{Ratcliffe1994})
    together with the explicit formula for distances in the hyperboloid model
    (see, \eg, \cite[Eq.~3.2.2]{Ratcliffe1994}) it is immediate that
    \begin{equation}\label{eq:radii_height_relationship_spherical}
        \sin r_i \;=\; \frac{1}{\cosh h_i}.
    \end{equation}
    Therefore, the desired identity
    \eqref{eq:lambda_length_to_inversive_distance_spherical} for the
    $\lambda$-length follows from the
    trigonometric formulas for semi-hyperideal triangles
    (see, \eg, \cite[Thm.~3.5.11]{Ratcliffe1994}).
\end{proof}

This relationship to semi-hyperideal triangles provides us with a means to compute the
$\lambda$-lengths, even if some vertices are ideal. In general we have
\begin{equation}\label{eq:heights_lambda_lengths_relationship_spherical}
    \begin{aligned}
        2(\ee^{\lambda_{ij}}+\epsilon_i\epsilon_j\ee^{-\lambda_{ij}})
        &=
        \big(\ee^{h_i}-\epsilon_i\ee^{-h_i}\big)
        \big(\ee^{h_j}-\epsilon_j\ee^{-h_j}\big)\\
        &\phantom{xxxxxxxxxxxx}
        -\cos\len_{ij}
        \big(\ee^{h_i}+\epsilon_i\ee^{-h_i}\big)
        \big(\ee^{h_j}+\epsilon_j\ee^{-h_j}\big).
    \end{aligned}
\end{equation}
Here, $\epsilon_x=1$ if $r_x>0$ and $\epsilon_x=0$ otherwise, $x\in\{i,j\}$.
If $\epsilon_i\epsilon_j=1$, this is a reformulation of
\eqref{eq:lambda_length_to_inversive_distance_spherical} which remains valid
in the general case (see, \eg, the derivations in \cite[Sec.~2.3]{Prosanov2020a}).

Heights of the hyperideal pyramids are not invariant under discrete
conformal change. Yet, they are closely related. In fact, if
$(\tri, \tilde{\len}, \tilde{r})$ is a decorated discrete spherical metric
discrete conformally equivalent to $(\tri, \len, r)$, it follows from
\eqref{eq:conformal_change_formulas_spherical} and
\eqref{eq:radii_height_relationship_spherical} that
\begin{equation}\label{eq:log_scale_factor_to_heights_spherical}
    \tanh h_i \;=\; \sqrt{1-\ee^{2u_i}\sin^2\tilde{r}_i}
\end{equation}
for hyperideal vertices $i\in\verts$. The corresponding equation for ideal vertices
is given by $u_i = \tilde{h}_i-h_i$. It clearly depends on the
choice of horospheres. Still, a change of auxiliary horosphere only
results in a constant offset of all lengths adjacent to the respective vertex.
These considerations lead to the following characterization of the set
$\polyhedra_{\tri}^{+}(\len, r)\subset\RR^{\verts}$ of possible heights under
discrete conformal change.

\begin{proposition}
    \label{prop:possible_heights_spherical}
    Let $(\tri, \len, r)$ be a hyperideally decorated discrete spherical metric
    and $\lambda\colon\edges_{\tri}\to\RR$ the $\lambda$-lengths corresponding
    to its fundamental discrete conformal invariant. Then
    $h\in\polyhedra_{\tri}^{+}(\len, r)$ if and only if
    \begin{enumerate}[label=(\roman*)]
        \item\label{item:hyperideal_heights}
            $h_i>0$, for all hyperideal $i$,
        \item\label{item:heights_triangle_inequality}
            $\lambda_{ij} < h_i+h_j$, for each edge $ij\in\edges_{\tri}$, and
        \item\label{item:spherical_link_condition}
            the new $\tilde{\len}\colon\edges_{\tri}\to(0,\pi)$ determined
            by $h$ and the $\lambda$-lengths via
            \eqref{eq:heights_lambda_lengths_relationship_spherical} satisfy
            the triangle inequalities \eqref{eq:triangle_inequalities} and
            perimeter condition \eqref{eq:perimeter_condition} for each
            triangle $ijk\in\faces_{\tri}$.
    \end{enumerate}
\end{proposition}
\begin{proof}
    For a hyperideal pyramid, item \ref{item:hyperideal_heights} follows by
    definition, item \ref{item:heights_triangle_inequality} is a consequence of
    the triangle inequality for the distance function in hyperbolic $3$-space, and
    item \ref{item:spherical_link_condition} follows from considering the link at
    the distinguished vertex $0$. Conversely, given $h$ satisfying these conditions,
    we can always construct a hyperideal pyramid: first constructing the spherical
    link at $0$, by virtue of \ref{item:spherical_link_condition}, and then raising
    the pyramid over it, using items \ref{item:hyperideal_heights} and
    \ref{item:heights_triangle_inequality}. The relationship to
    discrete conformal equivalence, explored in the previous paragraphs, guarantees
    that the resulting $\lambda$-lengths coincide with the original ones.
\end{proof}

The collection of hyperideal pyramids determined by $(\tri, \len, r)$ and
the heights $h\in\polyhedra_{\tri}^{+}(\len, r)$ is homeomorphic to a topological
cone over $\toposurf_g\setminus\verts$. Moreover, the hyperbolic metrics on the
pyramids induce a piecewise hyperbolic metric with conical singularities corresponding
the edges incident to the distinguished vertex $0$. With respect to this metric, the
boundary of the cone is isometric to the fundamental discrete conformal invariant
$\surf_g$. The cone-angles at the singularities coincide with the cone-angles
$\theta_i$ of the discrete metric (see \figref{fig:hyperideal_pyramid}).
Thus, we call the collection of hyperideal pyramids a
\emph{hyperideal polyhedral cone} over the triangulated hyperbolic surface
$(\surf_g, \tri)$ and denote it by $P_{\len, r}^{+}(h)$.

\subsubsection{The hyperbolic case}
\label{sec:hyperbolic_polyhedra_hyperbolic}
Similar to the spherical case, there is a unique face-circle $C_{ijk}$ for
each hyperideally decorated hyperbolic triangle $ijk$. Again, we can use $C_{ijk}$
to associate a hyperideal triangle to the decorated triangle: $r_i>0$ and $r_i=0$
corresponding to hyperideal and ideal vertices, respectively. In doing so, we
also obtain a complete hyperbolic surface $\surf_g$ with ends homeomorphic to
$\toposurf_g\setminus\verts$ for each hyperideally decorated discrete hyperbolic
metric $(\tri, \len, r)$.

\begin{definition}[fundamental discrete conformal invariant --- hyperbolic case]
   \label{def:fundamental_discrete_conformal_invariant_hyperbolic}
   Let $(\tri, \len, r)$ be a hyperideally decorated discrete hyperbolic metric.
   The complete hyperbolic surface $\surf_g$ considered above is called the
   \emph{(hyperbolic) fundamental discrete conformal invariant} of $(\tri, \len, r)$.
\end{definition}

As the name suggests $\surf_g$ is invariant under discrete conformal change.
To see this, let us consider the upper half-space model of hyperbolic $3$-space, \ie,
$\HH^3_{\mathrm{up}} = \{(x,y,z)\in\RR^3:z>0\}$. This time the extended plane
$\RR^2\cup\{\infty\}$ models the boundary at infinity. The totally geodesic
planes in $\HH^3_{\mathrm{up}}$ correspond to circles and lines in $\RR^2$.
For circles it is given by the intersection of $\HH^3_{\mathrm{up}}$ with
the unique sphere which intersects $\RR^2$ orthogonally in this circle. Similarly,
a half-plane can be associated to each line in $\RR^2$ (see, \eg,
\cite[\S4.6]{Ratcliffe1994}). Using the upper-half plane model, \ie,
$\HH^2_{\mathrm{up}} = \{(x,y,0)\in\RR^3:y>0\}$, we can embed the hyperbolic plane
into the ideal boundary of $\HH^3_{\mathrm{up}}$. Now, the totally geodesic
planes over the edges of a hyperideally decorated hyperbolic triangle $ijk$ together
with the plane over $C_{ijk}$ bound a \emph{hyperideal tetrahedron}
(see \figref{fig:hyperideal_prism}). This
tetrahedron has one distinguished hyperideal vertex, say $\hat{\infty}$,
corresponding to the totally geodesic plane associated to the ideal boundary of
$\HH^2_{\mathrm{up}}$. To emphasise that
we distinguished a hyperideal vertex, we also call the tetrahedron a
\emph{hyperideal prism}. The \emph{base face} of this prism, \ie,
the face which is not incident to $\hat{\infty}$, is a hyperideal triangle. It
is related to the previously considered hyperideal triangle via
the well-known relationship between the Poincar\'e disk model
and the upper half-sphere model of the hyperbolic plane
(see, \eg, \cite[Sec.~7]{CFK+1997}). So both constructions induce the same hyperbolic
metric on the triangle $ijk$, and thus on $\toposurf_g\setminus\verts$.

\begin{figure}[t]
    \centering
    \includegraphics[width=0.49\textwidth]{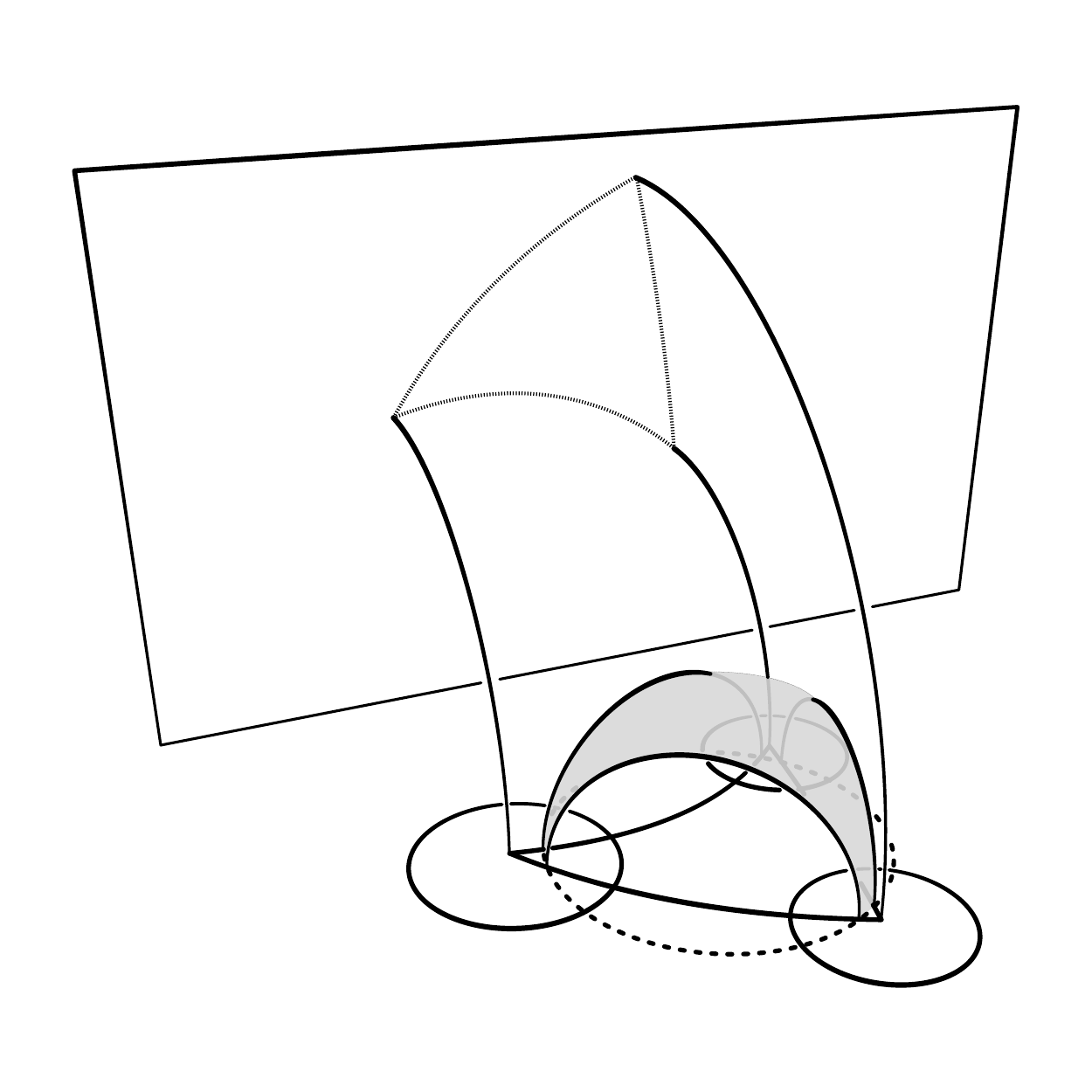}
    \caption{%
        Sketch of a hyperideal prism over a decorated hyperbolic triangle.
        The hyperideal triangle giving the base face of the prism is shaded.
        Note that the dihedral angles at the edges incident with distinguished
        hyperideal vertex $\hat{\infty}$ coincide with the interior
        angles $\theta_{ij}^k$ of the hyperbolic triangle.
    }
    \label{fig:hyperideal_prism}
\end{figure}

M\"obius-transformations of $\RR^2\cup\{\infty\}$ are in one-to-one correspondence
to the isometries of hyperbolic $3$-space. This is known as \emph{Poincar\'e-extension}
(see, \eg, \cite[\S~4.4]{Ratcliffe1994}).
Because discrete conformal equivalence is defined via M\"obius-transformations
acting on each triangle $ijk$ (\defref{def:dce_via_moebius}) we arrive at

\begin{proposition}
    \label{prop:dce_via_invariant_hyperbolic}
    Two hyperideally decorated discrete hyperbolic metrics on the same
    triangulation $\tri$ of the marked surface $(\toposurf_g, \verts)$ are
    discrete conformally equivalent if and only if they induce the same
    fundamental discrete conformal invariant.
\end{proposition}

As in the spherical case, we can introduce $\lambda$-lengths
$\lambda\colon\edges_{\tri}\to\RR$ and heights $h\colon\verts\to\RR$ for a
decorated discrete hyperbolic metric $(\tri, \len, r)$. The $\lambda$-lengths
describe the geometry of the fundamental discrete conformal invariant and the
heights are given by the (truncated) lengths of the remaining edges of the
hyperideal prisms, \ie, those incident with the distinguished hyperideal vertex
$\hat{\infty}$ (see \figref{fig:hyperideal_prism_notation}, left). If
both vertices of an edge $ij$ are hyperideal, $\lambda_{ij}$ is related to
the inversive distance  \eqref{eq:hyperbolic_inversive_distance} of the edge $ij$.

\begin{lemma}\label{lemma:lambda_length_to_inversive_distance_hyperbolic}
    Let $(\tri, \len, r)$ be a hyperideally decorated discrete hyperbolic metric.
    Consider an edge $ij\in\edges_{\tri}$ and suppose that both $r_i>0$ and
    $r_j>0$. Then
    \begin{equation}\label{eq:lambda_length_to_inversive_distance_hyperbolic}
        I_{ij}^{-} \;=\; \cosh\lambda_{ij}.
    \end{equation}
\end{lemma}
\begin{proof}
    Using the explicit formula for the hyperbolic arc length element in the
    upper half-plane model (see, \eg, \cite[Thm.~4.6.6]{Ratcliffe1994}) and
    the normalization shown in \figref{fig:hyperideal_prism_notation}, right,
    a computation shows that $\cosh(h_i) = \nicefrac{1}{\cos(\alpha)}$.
    Furthermore, $\cos(\alpha)=\nicefrac{\sinh(h_i)}{\cosh(h_i)}$ leading to
    \begin{equation}\label{eq:radii_height_relationship_hyperbolic}
        \sinh r_i \;=\; \frac{1}{\sinh h_i}.
    \end{equation}
    Hence, the desired identity
    \eqref{eq:lambda_length_to_inversive_distance_hyperbolic} for the
    $\lambda$-length follows from the
    trigonometric formulas for hyperideal triangles
    (see, \eg, \cite[Thm.~3.5.13]{Ratcliffe1994}).
\end{proof}

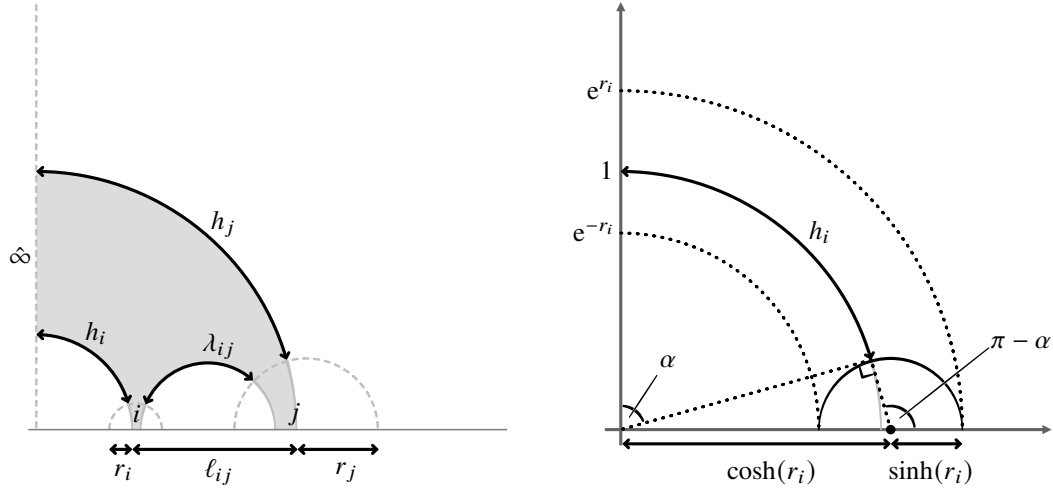
\begin{figure}[t]
    \centering
    \labellist
    \small\hair 2pt
    \pinlabel $\hat{\infty}$ at -10 280
    \pinlabel $i$ at 137 80
    \pinlabel $j$ at 335 80
    \pinlabel $h_i$ at 85 182
    \pinlabel $h_j$ at 245 320
    \pinlabel $\lambda_{ij}$ at 240 165
    \pinlabel $\len_{ij}$ at 240 6
    \pinlabel $r_i$ at 120 6
    \pinlabel $r_j$ at 400 6
    \endlabellist
    \includegraphics[width=0.45\textwidth]{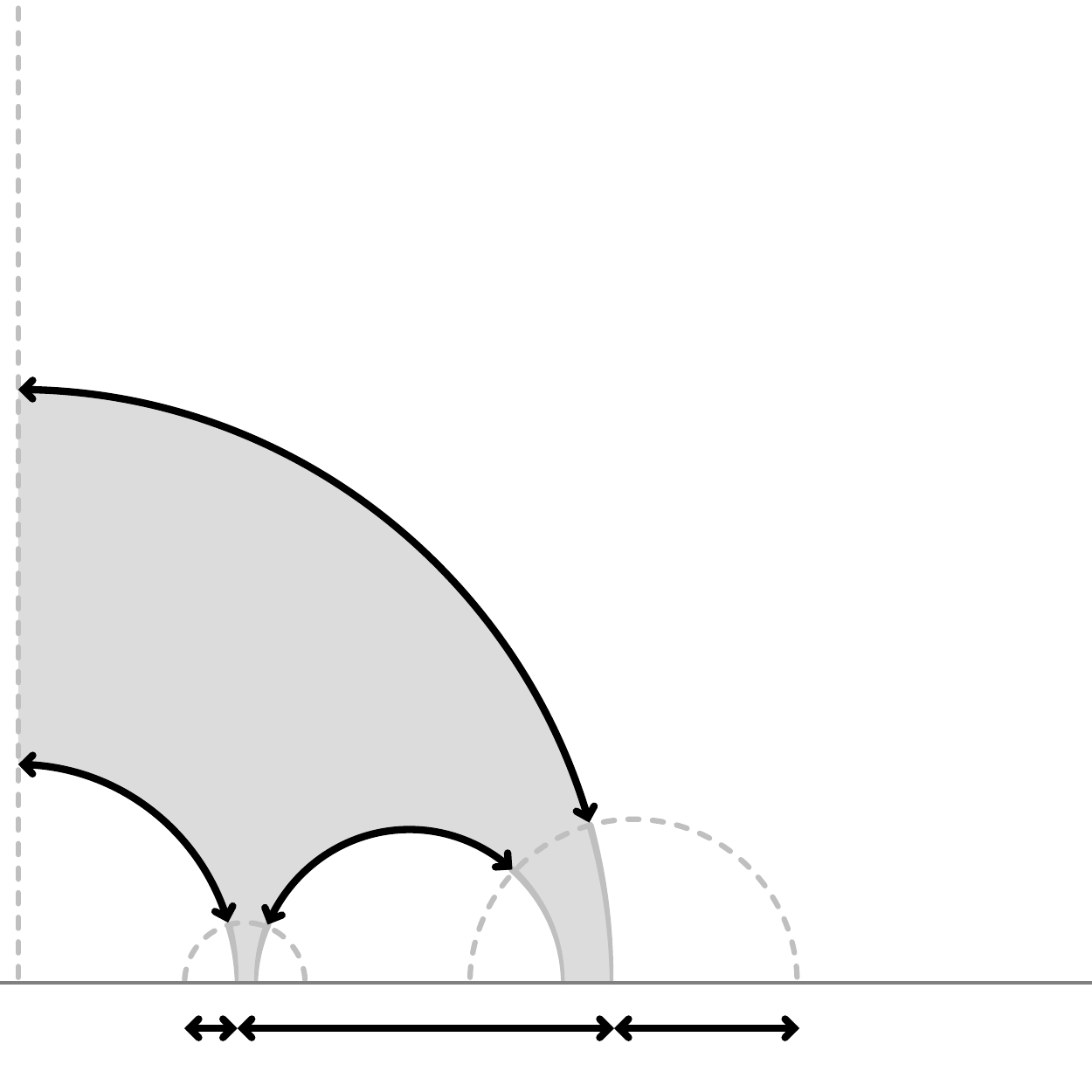}
    \qquad
    \labellist
    \small\hair 2pt
    \pinlabel $h_i$ at 310 310
    \pinlabel $1$ at 42 388
    \pinlabel $\ee^{-r_i}$ at 27 310
    \pinlabel $\ee^{r_i}$ at 36 490
    \pinlabel $\sinh(r_i)$ at 450 6
    \pinlabel $\cosh(r_i)$ at 250 6
    \pinlabel $\alpha$ at 120 147
    \pinlabel $\pi-\alpha$ at 563 170
    \endlabellist
    \includegraphics[width=0.45\textwidth]{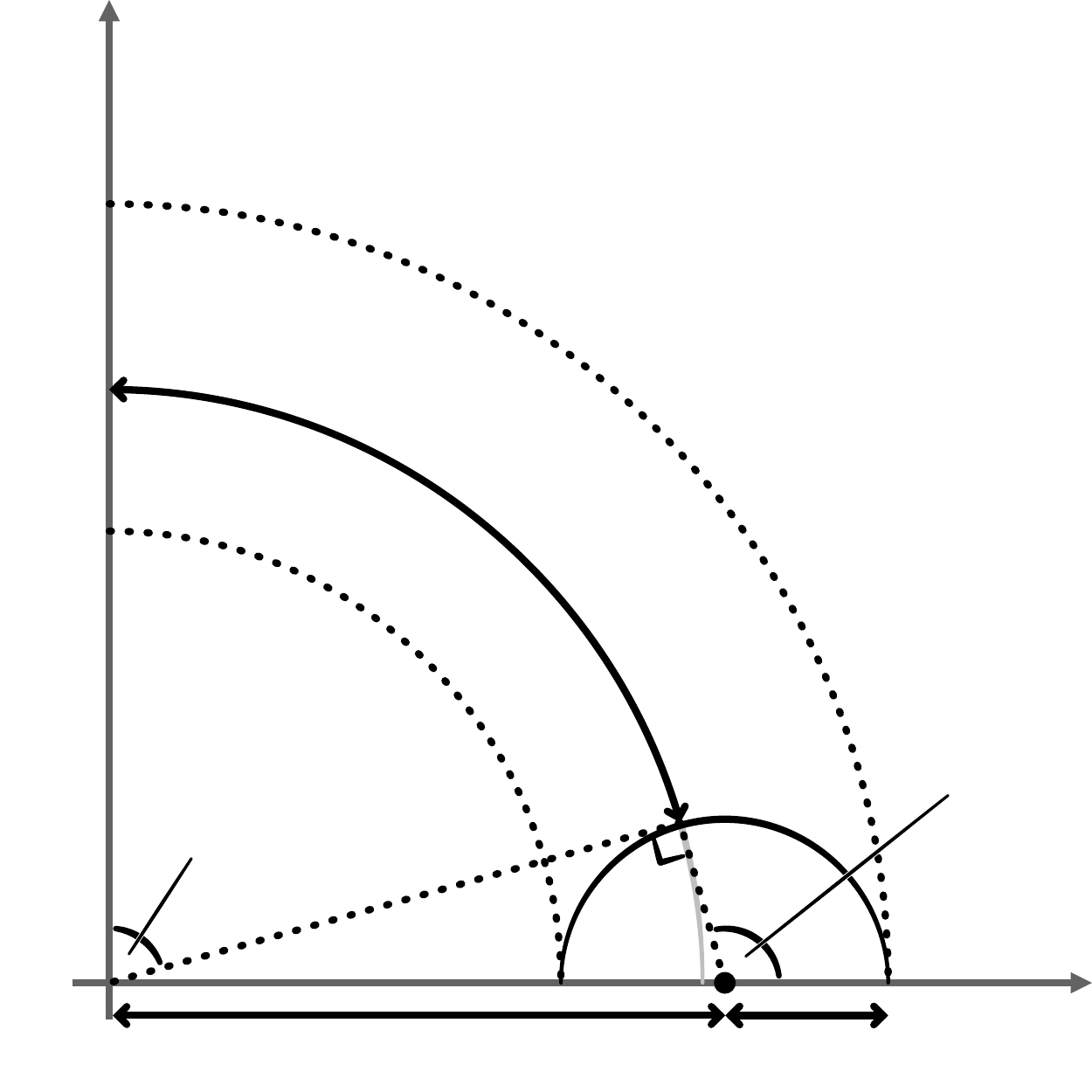}
    \caption{%
        \textsc{Left}: notation for hyperideal prism.
        \textsc{Right}: sketch of the construction used in
        \lemref{lemma:lambda_length_to_inversive_distance_hyperbolic}.
        It relates the euclidean parametrization of the upper half-plane to its
        hyperbolic metric. Auxiliary lines are dotted.
    }
    \label{fig:hyperideal_prism_notation}
\end{figure}

Again, this relationship to hyperideal triangles gives us
a way to compute the $\lambda$-lengths, even if some vertices are ideal
(compare to \eqref{eq:heights_lambda_lengths_relationship_spherical}):
\begin{equation}\label{eq:heights_lambda_lengths_relationship_hyperbolic}
    \begin{aligned}
        2(\ee^{\lambda_{ij}}+\epsilon_i\epsilon_j\ee^{-\lambda_{ij}})
        &\;=\;
        \cosh\len_{ij}
        \big(\ee^{h_i}-\epsilon_i\ee^{-h_i}\big)
        \big(\ee^{h_j}-\epsilon_j\ee^{-h_j}\big)\\
        &\phantom{xxxxxxxxxxxxxxx}
        -\big(\ee^{h_i}+\epsilon_i\ee^{-h_i}\big)
        \big(\ee^{h_j}+\epsilon_j\ee^{-h_j}\big).
    \end{aligned}
\end{equation}

The heights of the hyperideal prisms are not invariant under discrete conformal change.
But we get a similar relationship between the heights and the logarithmic scale
factors as in the spherical case. Let
$(\tri, \tilde{\len}, \tilde{r})$ be a decorated discrete hyperbolic metric
discrete conformally equivalent to $(\tri, \len, r)$. It follows from
\eqref{eq:conformal_change_formulas_hyperbolic} and
\eqref{eq:radii_height_relationship_hyperbolic} that
\begin{equation}\label{eq:log_scale_factor_to_heights_hyperbolic}
    \coth h_i \;=\; \sqrt{1+\ee^{2u_i}\sinh^2\tilde{r}_i}
\end{equation}
for hyperideal vertices $i\in\verts$. The corresponding equation for ideal vertices
is given by $u_i = \tilde{h}_i-h_i$. Thus, we again get a characterization of the set
$\polyhedra_{\tri}^{-}(\len, r)\subset\RR^{\verts}$ of possible heights under
discrete conformal change.

\begin{proposition}
    \label{prop:possible_heights_hyperbolic}
    Let $(\tri, \len, r)$ be a hyperideally decorated discrete hyperbolic metric
    and $\lambda\colon\edges_{\tri}\to\RR$ the $\lambda$-lengths corresponding
    to its fundamental discrete conformal invariant. Then
    $h\in\polyhedra_{\tri}^{-}(\len, r)$ if and only if
    \begin{enumerate}[label=(\roman*)]
        \item
            $h_i>0$, for all hyperideal $i$,
        \item
            the new $\tilde{\len}\colon\edges_{\tri}\to\RR_{>0}$ determined
            by $h$ and the $\lambda$-lengths via
            \eqref{eq:heights_lambda_lengths_relationship_hyperbolic} satisfy
            the triangle inequalities \eqref{eq:triangle_inequalities} for each
            triangle $ijk\in\faces_{\tri}$.
    \end{enumerate}
\end{proposition}
\begin{proof}
    The proof proceeds as for \propref{prop:possible_heights_spherical}.
\end{proof}

The collection of hyperideal prisms determined by $(\tri, \len, r)$ and
the heights $h\in\polyhedra_{\tri}^{-}(\len, r)$ is homeomorphic to the topological
cylinder $[0,1)\times(\toposurf_g\setminus\verts)$. It is endowed with a piecewise
hyperbolic metric with conical singularities corresponding the edges incident to the
distinguished hyperideal vertex $\hat{\infty}$. With respect to this metric, the
cylinder has one end of infinite volume corresponding to $\hat{\infty}$. Furthermore,
the boundary $\{0\}\times(\toposurf_g\setminus\verts)$ is isometric to the fundamental
discrete conformal invariant $\surf_g$. The cone-angles at the singularities
coincide with the cone-angles $\theta_i$ of the discrete metric
(see \figref{fig:hyperideal_prism}). Thus, we call the collection of hyperideal
prisms a \emph{hyperideal polyhedral end} over the triangulated hyperbolic surface
$(\surf_g, \tri)$ and denote it by $P_{\len, r}^{-}(h)$.

\section{The local variational principle for the discrete conformal mapping problem}
\label{sec:variational_principle}
\subsection{The discrete conformal mapping problem}
The notion of discrete conformal equivalence gives rise to the following
discrete conformal mapping problem:

\begin{problem}[prescribed angle sums]
   \label{problem:prescribed_angle_sums}
   \textbf{\emph{Given}}
   \begin{itemize}
       \item
           a triangulated marked surface $(\toposurf_g, \verts, \tri)$,
       \item
           a decorated discrete spherical/hyperbolic metric $(\tri, \len, r)$,
       \item
           and a desired angel sum $\Theta_i$ for each vertex $i\in\verts$.
   \end{itemize}
   \textbf{\emph{Find}} logarithmic scale factors $u\in\RR^{\verts}$ such that the
   discrete conformally changed discrete hyperbolic/spherical metric with respect
   to $u$ has angle sum $\Theta_i$ about each vertex $i\in\verts$.
\end{problem}

To find solutions to this problem one could try to directly evolve the
logarithmic scale factors according to some discrete curvature flow. Yet, it turns out
that it is favorable to consider the reparametrization by heights instead (see
\eqref{eq:log_scale_factor_to_heights_spherical} and
\eqref{eq:log_scale_factor_to_heights_hyperbolic}). Precisely, for a given
a decorated discrete spherical/hyperbolic metric $(\tri, \len, r)$, we consider
the evolution of $h\colon[0,T)\to\polyhedra_{\tri}^{\pm}(\len, r)$,
$T\in\RR_{>0}\cup\{\infty\}$, according to
\begin{equation}\label{eq:decorated_flow}
      \tdiff{t}h_i(t) \;=\; \Theta_i-\theta_i(t).
\end{equation}
The initial conditions are governed by $u_i(0)\equiv0$ and
\eqref{eq:log_scale_factor_to_heights_spherical} or
\eqref{eq:log_scale_factor_to_heights_hyperbolic}, respectively.
We call this the \emph{decorated $\Theta$-flow}. Here, $\theta_i(t)$ is the
angle-sum about the vertex $i\in\verts$ in the decorated discrete
spherical/hyperbolic metric induced by $h(t)$.

To better understand the local behavior of this flow, let us consider a single
hyperideally decorated spherical triangle $ijk$. Denote by $\alpha_{ij}^k$ the
interior intersection angle of the face-circle $C_{ijk}$ and the edge $ij$.
Moreover, let $r_{ij}$ be half of the distance between the two intersection-points
of $C_{ijk}$ with $ij$ (see \figref{fig:sketch_notation_decoration}). This is the
radius of the unique circle which intersects the edge $ij$ orthogonally in the same
points as $C_{ijk}$.

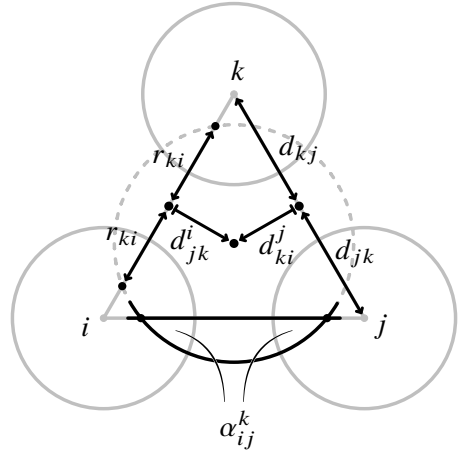
\begin{figure}[t]
    \centering
    \labellist
    \hair 2pt
    \pinlabel $i$ at 130 170
    \pinlabel $j$ at 470 170
    \pinlabel $k$ at 305 468
    \pinlabel $\alpha_{ij}^k$ at 305 50
    \pinlabel $d_{ki}^j$ at 350 265
    \pinlabel $d_{jk}^i$ at 250 265
    \pinlabel $r_{ki}$ at 226 365
    \pinlabel $r_{ki}$ at 172 275
    \pinlabel $d_{kj}$ at 375 375
    \pinlabel $d_{jk}$ at 440 255
    \endlabellist
    \includegraphics[width=0.49\textwidth]{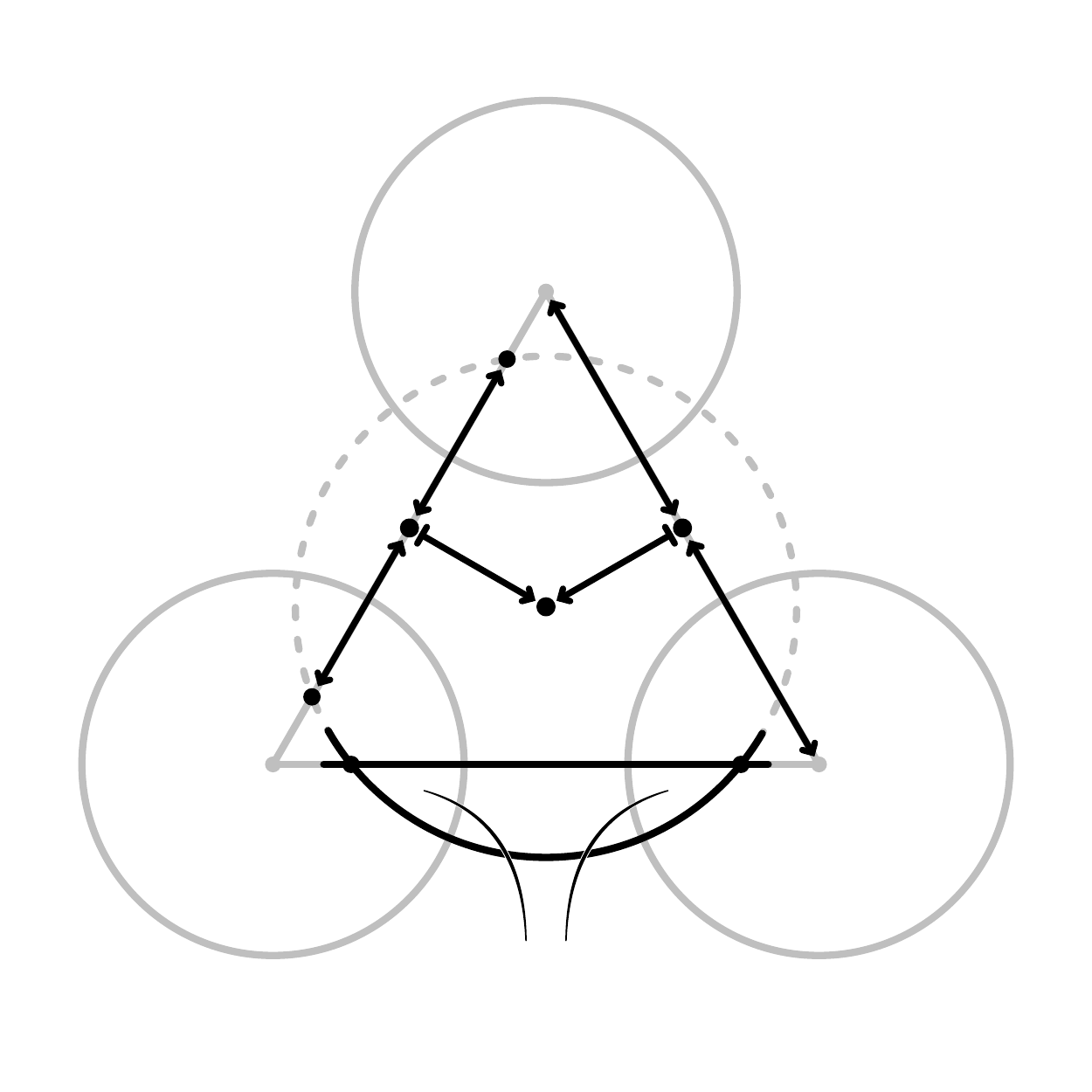}
    \caption{%
        Sketch of geometric quantities associated to a decorated triangle.
    }
    \label{fig:sketch_notation_decoration}
\end{figure}

\begin{lemma}\label{lemma:angle_derivative_spherical}
    Let $(\tri, \len, r)$ be a decorated discrete spherical metric. Under discrete
    conformal change the angles $\theta_{jk}^i$ in a triangle $ijk\in\faces_{\tri}$
    vary by
    \begin{equation}\label{eq:spherical_angle_derivative}
        \diff{}{\theta_{jk}^i}
        \;=\;
        \frac{\cot\alpha_{ij}^k\tan r_{ij}}{\sin\len_{ij}}\,
            (\diff{}{h_j}-\cos\len_{ij}\,\diff{}{h_i})
        \,+\, \frac{\cot\alpha_{ik}^j\tan r_{ik}}{\sin\len_{ik}}\,
            (\diff{}{h_k}-\cos\len_{ik}\,\diff{}{h_i}).
    \end{equation}
\end{lemma}
\begin{proof}
    From the spherical laws of sines and cosines (see, \eg,
    \cite[Thm.~2.5.2; Thm.~2.5.3]{Ratcliffe1994}) follows that the (oriented)
    distance $d_{ij}^k$ between the center of $C_{ijk}$ and the
    edge $ij$ can express as
    \begin{equation}\label{eq:distances_by_cotan_spherical}
        \tan d_{ij}^k \;=\; \cot\alpha_{ij}^k\,\sin r_{ij},
    \end{equation}
    where the orientation is chosen such that $d_{ij}^k$ is positive if the center
    lies on the same side of $ij$ as the triangle. Furthermore, the distance
    $d_{ij}$ between the vertex $i$ and
    the foot of the projection of the center of $C_{ijk}$ to
    $ij$ is given by
    \begin{equation}\label{eq:radii_relationship_spherical}
        \cos d_{ij} \;=\; \cos r_i\,\cos r_{ij}.
    \end{equation}

    Now this lemma is a reformulation of \cite[Thm.~9]{GT2017} using
    \eqref{eq:distances_by_cotan_spherical}, \eqref{eq:radii_relationship_spherical},
    and the reparametrization \eqref{eq:log_scale_factor_to_heights_spherical}
    via heights.
    In the following table we show how our notation is related to theirs.
    \begin{center}
        \begin{tabular}%
            {>{\centering\arraybackslash}m{4.6cm}|>{\centering\arraybackslash}m{4.6cm}}
            \textsc{Glickenstein} \& \textsc{Thomas} \cite{GT2017} & our notation\\
            \hline\hline
            $f_i$ & $u_i$\\
            $u_i$ & $-h_i$\\
            $h_{ij}$ & $d_{ij}^k$\\
            $\gamma_i$ & $\theta_{jk}^i$\\
            $\alpha_i$ & $\sin^2\tilde{r}_i$
        \end{tabular}
    \end{center}
\end{proof}

These considerations are not limited to discrete spherical metrics. We can define
$\alpha_{ij}$, $r_{ij}$, $d_{ij}^k$, and $d_{ij}$ for a hyperideally decorated
hyperbolic triangle in an analogous fashion.

\begin{lemma}\label{lemma:angle_derivative_hyperbolic}
    Let $(\tri, \len, r)$ be a decorated discrete hyperbolic metric. Under discrete
    conformal change the angles $\theta_{jk}^i$ in a triangle $ijk\in\faces_{\tri}$
    vary by
    \begin{equation}\label{eq:euclidean_angle_derivative}
        \diff{}{\theta_{jk}^i}
        \;=\;
        \frac{\cot\alpha_{ij}^k\tanh r_{ij}}{\sinh\len_{ij}}\,
            (\diff{}{h_j}-\cosh\len_{ij}\,\diff{}{h_i})
        \,+\, \frac{\cot\alpha_{ik}^j\tanh r_{ik}}{\sinh\len_{ik}}\,
            (\diff{}{h_k}-\cosh\len_{ik}\,\diff{}{h_i}).
    \end{equation}
\end{lemma}
\begin{proof}
    This time we can use the hyperbolic
    laws of sines and cosines (see, \eg, \cite[Thm.~3.5.2; Thm.~3.5.3]{Ratcliffe1994})
    to derive the relationships
    \begin{equation}\label{eq:distances_by_cotan_hyperbolic}
        \tanh d_{ij}^k \;=\; \cot\alpha_{ij}^k\,\sinh r_{ij}
    \end{equation}
    and
    \begin{equation}\label{eq:radii_relationship_hyperbolic}
        \cosh d_{ij} \;=\; \cosh r_i\,\cosh r_{ij}.
    \end{equation}
    Thus, we obtain the following result as a reformulation of \cite[Thm.~7]{GT2017}
    using \eqref{eq:distances_by_cotan_hyperbolic},
    \eqref{eq:radii_relationship_hyperbolic}, and the reparametrization
    \eqref{eq:log_scale_factor_to_heights_hyperbolic} via heights.
\end{proof}

\subsection{The local variational principle}
\label{sec:local_variational_principle}
The differential properties of the cone angles $\theta_i$ imply that the decorated
$\Theta$-flow \eqref{eq:decorated_flow} can be locally integrated
\cite[Thm.~3]{GT2017}. Indeed, the relationship to hyperbolic polyhedra discussed
in \secref{sec:connections_hyperbolic_polyhedra} provides us with a way to find an
explicit primitive. To this end we consider the volumes of these polyhedra. They are
intimately related to their dihedral angles \cite[Sec.~2]{Kellerhals1989}.

\begin{proposition}[Schl\"afli's differential formula]
   \label{prop:schlaflis_differential_formula}
   The differential of the volume $\vol$ on the space of (compact)
   convex hyperbolic polyhedra with fixed combinatorics is
   \begin{equation}\label{eq:schlaflis_differential_formula}
      \diff{}{\vol}
      \;=\;
      -\frac{1}{2}\sum_{ij}\lambda_{ij}\,\diff{}{\alpha_{ij}}.
   \end{equation}
   Here, we take the sum over all edges $ij$ of the polyhedron. For each edge
   $\lambda_{ij}$ denotes its length and $\alpha_{ij}$ the interior dihedral angle
   at it.
\end{proposition}

The volumes of hyperbolic polyhedra remain finite if we allow them to have ideal
vertices. Moreover, Schl\"afli's differential formula still holds
\cite{Milnor1994b}. The edge-lengths are now considered with respect to some
auxiliary horospheres at the ideal vertices (see
\secref{sec:connections_hyperbolic_polyhedra}). Should the polyhedron have hyperideal
vertices, its volume is no longer finite. Instead, we consider its
\emph{truncated volume}, \ie, the volume obtained by truncating the
polyhedron with the hyperbolic planes corresponding to the hyperideal vertices.
Since these planes remain orthogonal to their adjacent faces if we vary the
hyperideal vertices, the terms in \teqref{eq:schlaflis_differential_formula}
corresponding to these angles vanish. Hence, Schl\"afli's differential formula
also applies to polyhedra with hyperideal vertices.

\begin{remark}
    \label{remark:formula_for_hyperbolic_volume}
    In general the volumes of hyperbolic polyhedra are hard to compute. Yet, in
    the case of hyperbolic tetrahedra there are explicit formulas \cite{Ushijima2006}.
    They express $\vol$ as a function of the dihedral angles $\alpha_{ij}$
    using the \emph{dilogarithm function}
    \begin{equation}
        \mathrm{Li}_2(x)
        \coloneq
        -\int_0^x \frac{\log(1-t)}{t}\,\diff{}{t}.
    \end{equation}
    These formulas can be further simplified if the vertices of the base face
    are ideal (see \cite[Sec.~5.2]{Vinberg1993} in the spherical case and
    \cite[Sec.~7]{Springborn2008} in the hyperbolic case).
\end{remark}

Let $(\tri, \len, r)$ be a decorated discrete spherical/hyperbolic metric.
For $\Theta\in\RR^{\verts}$ the
\emph{discrete Hilbert--Einstein functional (dHE-functional)}
$\HE_{\len, r, \Theta}^{\pm}\colon\polyhedra_{\tri}^{\pm}(\len, r)\to\RR$
is given by
\begin{equation}\label{eq:he_functional_local}
    \HE_{\len, r, \Theta}^{\pm}(h)
    \;\coloneqq\;
    -2\vol(P_{\len, r}^{\pm}(h))
    \,+\, \sum_{i\in\verts}(\Theta_i-\theta_i)h_i
    \,+\, \sum_{ij\in\edges_{\tri}}(\pi-\alpha_{ij})\lambda_{ij}.
\end{equation}
Here, $\alpha_{ij}\coloneq\alpha_{ij}^k+\alpha_{ij}^l$ for the two adjacent
triangles $ijk, ilj\in\faces_{\tri}$ incident to the edge $ij\in\edges_{\tri}$.

\begin{proposition}[local variational principle]
    \label{prop:local_variational_principle}
    The decorated $\Theta$-flow
    \eqref{eq:decorated_flow} is the gradient
    flow of $\HE_{\len, r, \Theta}^{\pm}$. In particular,
    $h\in\polyhedra_{\tri}^{\pm}(\len, r)$ is a solution of
    \probref{problem:prescribed_angle_sums} if and only if
    $h$ is a critical point of $\HE_{\len, r, \Theta}^{\pm}$.
\end{proposition}
\begin{proof}
    This is a consequence of Schl\"afli's differential formula.
\end{proof}

Let us introduce the \emph{decorated cotan-weights}
$\cotw\colon\edges_{\tri}\to\RR$ by
\begin{equation}\label{eq:cotan_weights}
    \cotw_{ij}^{\pm}
    \;\coloneqq\;
    \begin{cases}
        \frac{\big(\cot\alpha_{ij}^k\,+\,\cot\alpha_{ij}^l\big)\tan r_{ij}}
            {\sin\len_{ij}} &
        \text{, spherical case,}\\
        \frac{\big(\cot\alpha_{ij}^k\,+\,\cot\alpha_{ij}^l\big)\tanh r_{ij}}
            {\sinh\len_{ij}} &
        \text{, hyperbolic case.}
    \end{cases}
\end{equation}
Here, $ijk$ and $ilj$ are the adjacent triangles with common edge $ij$. Using these
weights we define
\begin{equation}
    \laplacian_{\len,r}^{\pm}
    \coloneq
    \sum_{ij\in\edges_{\tri}}\cotw_{ij}^{\pm}\,\big(\diff{}{h_j}-\diff{}{h_i}\big)^2
\end{equation}
and
\begin{equation}
    \Diff_{\len,r}^{\pm}
    \coloneq
    \begin{cases}
        \sum_{i\in\verts}
        \sum_{j:ij\in\edges_{\tri}}\cotw_{ij}^{+}(\cos(\len_{ij})-1)
        \,\diff{}{h_i^2}
        &\text{, spherical case,}\\
        \sum_{i\in\verts}
        \sum_{j:ij\in\edges_{\tri}}\cotw_{ij}^{-}(\cosh(\len_{ij})-1)
        \,\diff{}{h_i^2}
        &\text{, hyperbolic case.}
    \end{cases}
\end{equation}

\begin{lemma}\label{lemma:local_hessian_formula}
    Let $(\tri, \len, r)$ be a decorated discrete spherical/hyperbolic metric
    and $\Theta\in\RR^{\verts}$. The dHE-functional $\HE_{\len, r, \Theta}^{\pm}$
    is analytic in the interior of $\polyhedra_{\tri}^{\pm}(\len, r)$.
    Its Hessian is given by the quadratic form
    \begin{equation}\label{eq:explicit_formula_jacobian}
        \sum_{i\in\verts}\diff{}{\theta_i}\diff{}{h_i}
        \;=\;
        -\big(\laplacian_{\len,r}^{\pm} \,+\, \Diff_{\len, r}^{\pm}\big).
    \end{equation}
\end{lemma}
\begin{proof}
    The Hessian of $\HE_{\len, r, \Theta}^{\pm}$ is given by
    $\sum_{i\in\verts}\diff{}{\theta_i}\diff{}{h_i}$ since the $\Theta$-flow
    \eqref{eq:decorated_flow} is its gradient flow. Now formula
    \eqref{eq:explicit_formula_jacobian} follows from a straightforward,
    but tedious, computation: first summing over all corners, using
    \lemref{lemma:angle_derivative_spherical} or
    \lemref{lemma:angle_derivative_hyperbolic}, respectively, then regrouping
    by edges.
\end{proof}

\begin{remark}
    The operator $\laplacian_{\len,r}$ is the quadratic form belonging
    to a \emph{graph Laplacian} on $\tri$. Both its construction and the
    formula of its weights are reminiscent of the \emph{discrete cotan-Laplacian}
    for piecewise euclidean surfaces \cite{PP1993} and its direct generalizations
    (see, \eg, \cite{BS2007, Glickenstein2011}).
\end{remark}

\subsection{Weighted Delaunay tessellations}
\label{sec:weighted_delaunay_tessellations}
The properties of the dHE-functional closely depend on the underlying triangulation
of the discrete metric. We say that the triangulation $\tri$ is a
\emph{weighted Delaunay triangulation} with respect to the decorated discrete
spherical/hyperbolic metric $(\tri, \len, r)$ if
\begin{equation}\label{eq:local_delaunay_condition_weights}
    \cotw_{ij}^{\pm}\;\geq\;0
\end{equation}
for all edges $ij\in\edges_{\tri}$. Using the geometry induced on the marked
surface $(\toposurf_g, \verts)$ by $(\tri, \len, r)$, \ie, the
piecewise spherical/hyperbolic metric $\dist_{\toposurf_g}$, this definition
has the following equivalent geometric interpretations.

\begin{lemma}
    \label{lemma:local_characterization_wDt}
    Fix a decorated discrete spherical/hyperbolic metric $(\tri, \len, r)$ and
    consider two adjacent triangles $ijk$ and $ilj$ with common edge $ij$.
    The following statements are equivalent:
    \begin{enumerate}[label=(\roman*)]
        \item\label{item:local_delaunay_weights}
            the edge $ij$ is \emph{local Delaunay}, \ie, $\cotw_{ij}^{\pm}\geq0$,
        \item\label{item:local_delaunay_distance}
            the center of the face-circle $C_{ijk}$ \enquote{lies to the left} of the
            center of the face-circle $C_{ilj}$, \ie,
            \begin{equation}\label{eq:local_delaunay_condition_distances}
                d_{ij}^k \,+\, d_{ij}^l \;\geq\; 0,
            \end{equation}
        \item\label{item:local_delaunay_angles}
            the face-circles $C_{ijk}$ and $C_{ilj}$ intersect with an angle less
            or equal to $\pi$, \ie,
            \begin{equation}
                \alpha_{ij}^k \,+\, \alpha_{ij}^l \;\leq\; \pi.
            \end{equation}
    \end{enumerate}
\end{lemma}
\begin{proof}
    This follows by direct computation from
    \eqref{eq:distances_by_cotan_spherical},
    \eqref{eq:distances_by_cotan_hyperbolic}, and
    \eqref{eq:cotan_weights}.
\end{proof}

This lemma suggests that weighted Delaunay triangulations only depend on
$(\dist_{\toposurf_g}, r)$. But they are not necessarily unique. Thus, we
define the \emph{weighted Delaunay tessellation} of $(\toposurf_g, \verts)$
with respect to $(\dist_{\toposurf_g}, r)$: let $\tri$ be a weighted Delaunay
triangulation with respect to a decorated discrete metric inducing
$(\dist_{\toposurf_g}, r)$. The corresponding weighted Delaunay tessellation is
the decomposition of $(\toposurf_g, \verts)$ obtained by removing all edges
with $\cotw_{ij}=0$.

\begin{lemma}\label{lemma:uniqueness_wDt}
    Let $(\dist_{\toposurf_g}, r)$ be a decorated piecewise
    spherical/hyperbolic metric on the marked surface $(\toposurf_g, \verts)$.
    If the weighted Delaunay tessellation defined above exists, then
    \begin{enumerate}[label=(\roman*)]
        \item
            it is indeed a tessellation of the marked surface $(\toposurf_g, \verts)$
            in the sense of \secref{sec:basic_definitions},
        \item
            it is uniquely determined by $(\dist_{\toposurf_g}, r)$, \ie,
            it does not depend on the discrete metric we used to define it.
    \end{enumerate}
\end{lemma}
\begin{proof}
    This follows from the connection of weighted Delaunay tessellations to
    canonical tessellations of complete hyperbolic surfaces established in
    \propref{prop:canonical_local_to_global} and
    \lemref{lemma:local_delaunay_convexity_equivalence}.
\end{proof}

A weighted Delaunay triangulation with respect to $(\dist_{\toposurf_g}, r)$
can now be recovered by taking any triangular refinement of its weighted Delaunay
tessellation. The following lemma addresses the existence of these tessellations.

\begin{lemma}\label{lemma:existence_wDt}
    For a hyperideally decorated piecewise spherical/hyperbolic metric
    $(\dist_{\toposurf_g}, r)$ on the marked surface $(\toposurf_g, \verts)$
    there exists a weighted Delaunay tessellation.
\end{lemma}
\begin{proof}
    For hyperbolic metrics this was proved in \cite[Thm.~3.3; Prop.~3.21]{Lutz2023}.
    We discuss the spherical counterpart in appendix \ref{sec:spherical_wDt}.
\end{proof}

Consider a weighted Delaunay triangulation $\tri$ with respect to the decorated
piecewise spherical/hyperbolic metric $(\dist_{\toposurf_g}, r)$. Let
$(\tri, \len, r)$ be the corresponding discrete metric. We denote
by $\polyhedra_{\tri}^{\pm}(\dist_{\toposurf_g}, r)$ the subset of
$\polyhedra_{\tri}^{\pm}(\len, r)$ for which $\tri$ remains a weighted
Delaunay triangulation after discrete conformal change of metric.

\begin{proposition}[local concavity of $\HE_{\len, r, \Theta}^{-}$]
    \label{prop:local_concavity_hyperbolic_he}
    Let $(\dist_{\toposurf_g}, r)$ be a decorated piecewise
    hyperbolic metric on the marked surface $(\toposurf_g, \verts)$. Consider
    a corresponding discrete metric $(\tri, \len, r)$ where $\tri$ is
    a weighted Delaunay triangulation.
    For any choice of $\Theta\in\RR^{\verts}$ the dHE-functional
    $\HE_{\len, r, \Theta}^{-}$ is locally strictly concave over
    $\polyhedra_{\tri}^{-}(\dist_{\toposurf_g}, r)$.
\end{proposition}
\begin{proof}
    By definition of a weighted Delaunay triangulation, all cotan-weights
    $\cotw_{ij}^{-}$ are non-negative. Thus, $\Diff_{\len,r}^{-}$ is
    positive definite and $\laplacian_{\len, r}^{-}$ is positive semi-definite.
    So the assertion follows from \lemref{lemma:local_hessian_formula}.
\end{proof}

\section{Global analysis of decorated discrete conformal equivalence}
\label{sec:global_analysis}
\subsection{Canonical tessellations of complete hyperbolic surfaces}
\label{sec:canonical_tessellations}
In the next two sections we are going to relate properties of the triangulations
of piecewise spherical/hyperbolic surfaces to those of their fundamental
discrete conformal invariant, \ie, complete hyperbolic surfaces with ends.

To start, let us discuss how a hyperideal polyhedral cone induces a decoration
of the hyperbolic surface $\surf_g$ at its base. We can equip the distinguished
vertex of this cone with a sphere of radius
$R\geq0$. Now, to each hyperideal vertex we can associate the hypersphere which
is orthogonal to this sphere (see \figref{fig:induced_decorations}, left). Similarly,
we find a horosphere with this property for each ideal vertex. The radii
$\rho_i$ of these spheres are related to the heights $h_i$ by
\begin{equation}\label{eq:heights_to_hyperbolic_radii_spherical}
    \ee^{h_i}-\epsilon_i\ee^{-h_i}
    \;=\;
    \cosh(R)\,
    (\ee^{\rho_i}-\epsilon_i\ee^{-\rho_i}).
\end{equation}
Here, $\epsilon_i$ is defined as for
\eqref{eq:heights_lambda_lengths_relationship_spherical} and
the radius of a horosphere is the distance to the auxiliary horosphere
we used to define its height. These hyper- and horospheres define
hyper- and horocycles in $\surf_g$, that is, a \emph{decoration} of the hyperbolic
surface. The decoration is determined by the \emph{weights}
$\nicefrac{(\ee^{\rho_i}-\epsilon_i\ee^{-\rho_i})}{2}\eqcolon\omega_i$
on $\surf_g$. We define the function $\Omega_R^{+}\colon h_i\mapsto \omega_i$
via \eqref{eq:heights_to_hyperbolic_radii_spherical}. It naturally extends to
subsets of $\RR^{\verts}$.

We can make a similar construction for hyperideal polyhedral ends. This time
we equip the distinguished vertex with a hypersphere of radius $R>0$. The
radii $\rho_i$ of the spheres induced at the vertices (see
\figref{fig:induced_decorations}, right) are related to the heights $h_i$ by
\begin{equation}\label{eq:heights_to_hyperbolic_radii_hyperbolic}
    \ee^{h_i}+\epsilon_i\ee^{-h_i}
    \;=\;
    \sinh(R)\,
    (\ee^{\rho_i}-\epsilon_i\ee^{-\rho_i}).
\end{equation}
Again we obtain a decoration of $\surf_g$ and corresponding weights $\omega_i$.
The resulting function between heights and weights determined by
\eqref{eq:heights_to_hyperbolic_radii_hyperbolic} is denoted by
$\Omega_R^{-}$.

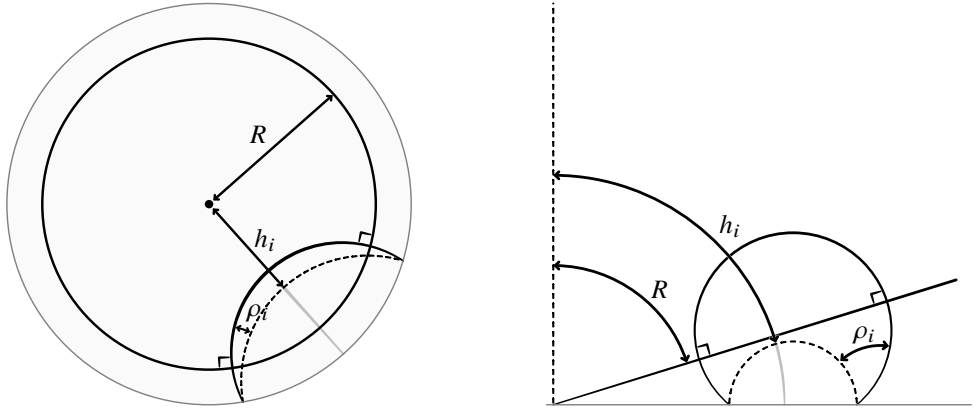
\begin{figure}[t]
    \centering
    \labellist
    \small\hair 2pt
    \pinlabel $R$ at 370 400
    \pinlabel $h_i$ at 380 250
    \pinlabel $\rho_i$ at 370 145
    \endlabellist
    \includegraphics[width=0.4\textwidth]{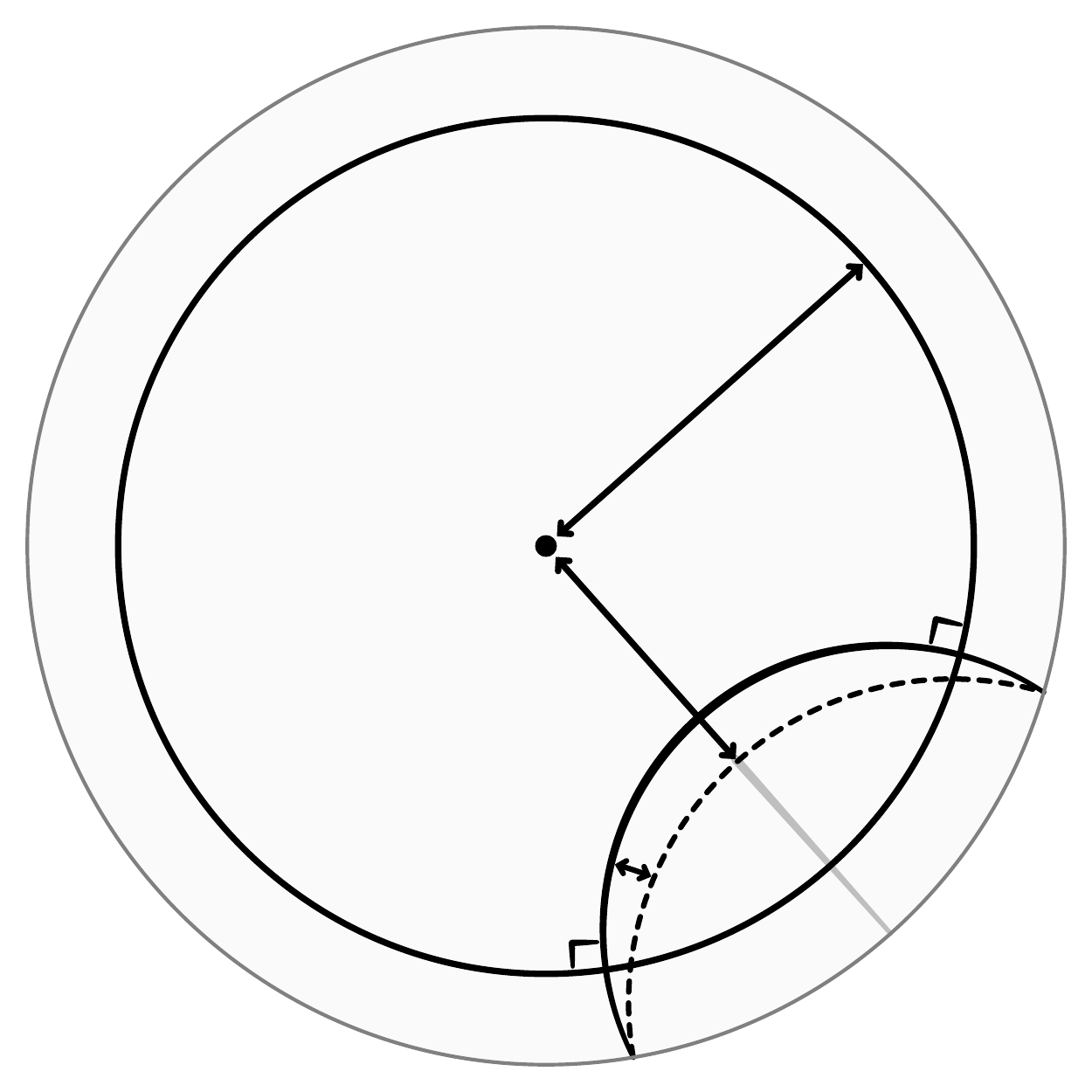}
    \qquad\qquad
    \labellist
    \small\hair 2pt
    \pinlabel $R$ at 160 180
    \pinlabel $h_i$ at 260 270
    \pinlabel $\rho_i$ at 450 115
    \endlabellist
    \includegraphics[width=0.4\textwidth]{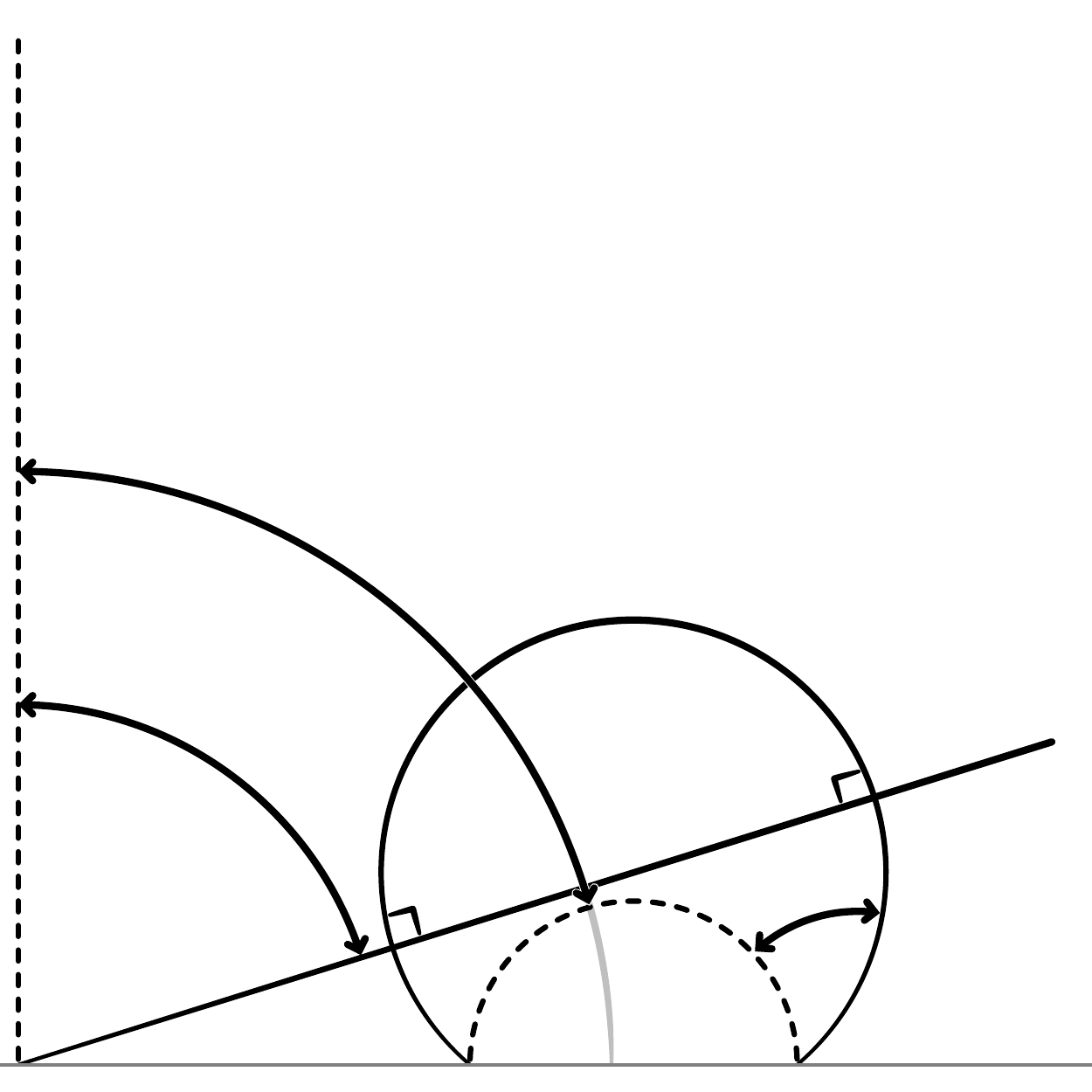}
    \caption{%
        Sketch depicting the relationship between the decoration of the distinguished
        and another vertex of of a hyperideal polyhedral cone (left) or end (right).
    }
    \label{fig:induced_decorations}
\end{figure}

We can associate a \emph{canonical (geodesic) tessellation} $\tri$
of $\surf_g$ to a decoration of $\surf_g$ with pairwise non-intersecting
vertex-cycles. It is the analog of a weighted Delaunay tessellation. Canonical
tessellations are defined via \emph{properly immersed disks}: for each cell of
$\tri$ there is
an isometric immersion of a hyperbolic disk into $\surf_g$ which is orthogonal to
the vertex-cycles of the cell and either does not, or at most with an angle of
$\nicefrac{\pi}{2}$, intersect any other vertex-cycle
(see \figref{fig:hyperbolic_proper_disks} and \cite[Sec.~3.2]{Lutz2023}).
Furthermore, this construction defines a geodesic embedding of the
$1$-skeleton of the combinatorial dual of the canonical tessellation, its
\emph{dual diagram}. The vertices of this dual diagram are given by the centers of
the properly immersed disks. Note that different decorations of $\surf_g$ can
induce the same canonical tessellation even though their dual diagrams do not
coincide (see \cite[Sec.~3.3]{Lutz2023}).

\begin{figure}[t]
   \centering
   \includegraphics[width=0.49\textwidth]{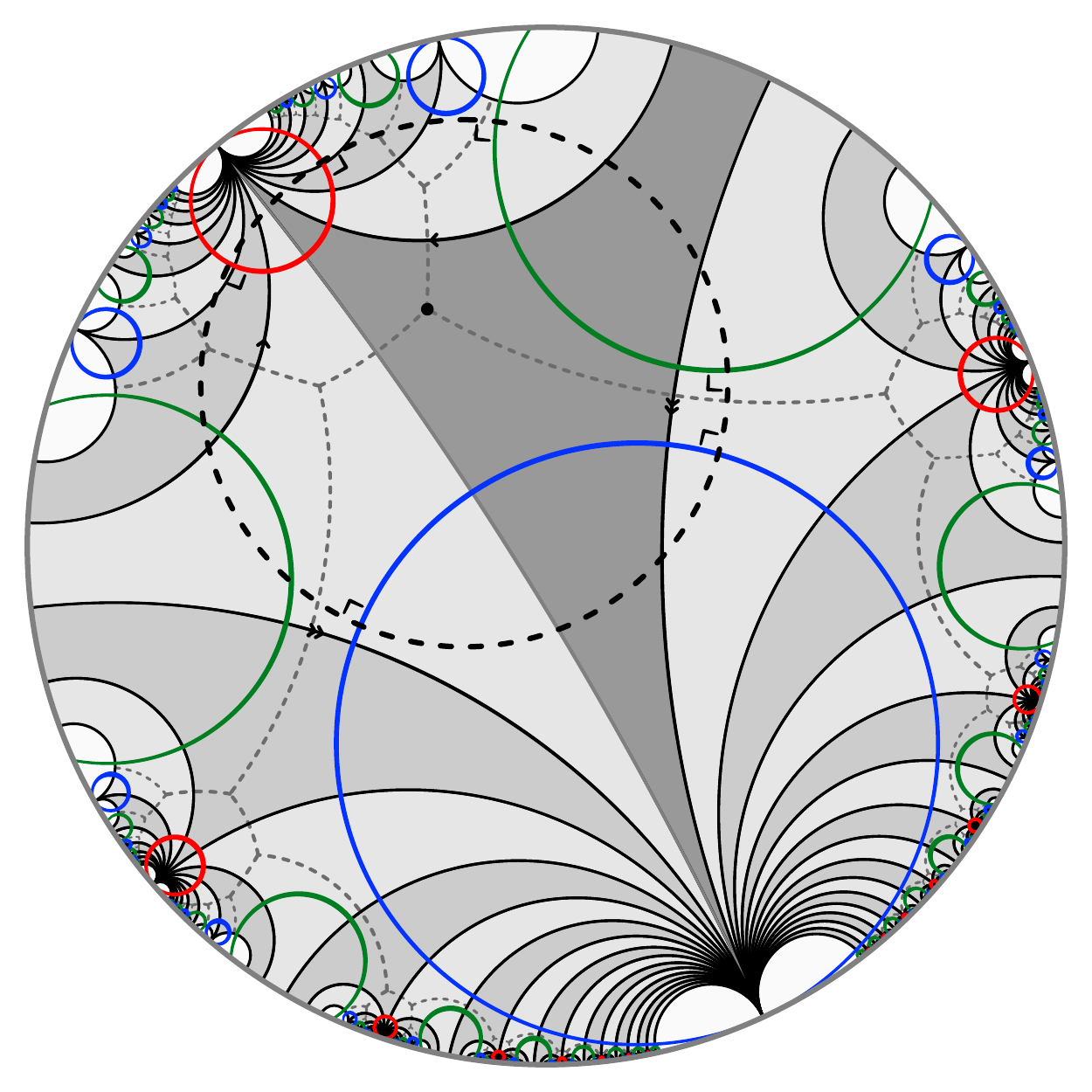}
   \caption{%
      Visualization of a proper hyperbolic disk (dashed black) of a decorated
      hyperbolic surface homeomorphic to the thrice punctured sphere in a lift
      to the hyperbolic plane. The associated canonical cell is depicted in dark gray.
      The checker pattern corresponds to the canonical tessellation of the surface.
      The dashed graph depicts its dual diagram. A fundamental polygon of the
      surface is bordered black. The identifications indicated by arrows
      correspond to the action of a Fuchsian group.
   }
   \label{fig:hyperbolic_proper_disks}
\end{figure}

More generally, a canonical tessellation and its dual diagram can be defined for
each choice of weights $\omega\in\RR_{>0}^{\verts}$. Denote the canonical tessellation
by $\tri_{\surf_g}^{\omega}$. For a given geodesic tessellation $\tri$ of
$\surf_g$ we define
\begin{equation}
   \mathcal{D}_{\tri}(\surf_g)
   \;\coloneqq\;
   \left\{\omega\in\RR_{>0}^{\verts}\,:\,
      \tri\text{ refines }\tri_{\surf_g}^{\omega}\right\}.
\end{equation}
Canonical tessellations and the sets $\mathcal{D}_{\tri}(\surf_g)$
have the following properties \cite[Sec.~4]{Lutz2023}:

\begin{proposition}[properties of the space of weights]
    \label{prop:properties_of_weightings}
    Let $\surf_g$ be a complete hyperbolic surface with ends which is homeomorphic to
    $\toposurf_g\setminus\verts$.
    \begin{enumerate}[label=(\roman*)]
        \item\label{item:tessellation_scale_invariance}
            Each choice of weights $\omega\in\RR_{>0}^{\verts}$ of $\surf_g$ induces
            a unique canonical tessellation. The induced tessellation and its dual
            diagram are invariant with respect to scaling of the weights, \ie, all
            $s\omega$, $s>0$, induce the same canonical tessellation and
            dual diagram.
        \item\label{item:canonical_cell_characterization}
            Each $\mathcal{D}_{\tri}(\surf_g)$ is either empty or the intersection of
            $\RR_{>0}^{\verts}$ with a closed polyhedral cone.
        \item\label{item:canonical_space_decomposition}
            For two geodesic tessellations $\tri_1$ and $\tri_2$ either
            $\mathcal{D}_{\tri_1}(\surf_g)\cap\mathcal{D}_{\tri_2}(\surf_g)=\emptyset$
            or there is another tessellation $\tri_3$ such that
            \(
                \mathcal{D}_{\tri_1}(\surf_g)\cap\mathcal{D}_{\tri_2}(\surf_g)
                =\mathcal{D}_{\tri_3}(\surf_g)
            \).
        \item\label{item:canonical_space_finite}
            There is only a finite number of geodesic tessellations
            $\tri_1,\dotsc,\tri_N$ of $\surf_g$ such that
            $\mathcal{D}_{\tri_n}(\surf_g)$
            is non-empty. In particular,
            $\RR_{>0}^{\verts}=\bigcup_{n=1}^N\mathcal{D}_{\tri_n}(\surf_g)$.
    \end{enumerate}
\end{proposition}

The following technical lemma will be useful for the analysis of the variational
principle. It provides us with information about canonical tessellations for
nearly degenerate weights.

\begin{lemma}[\!{\cite[proof of Thm.~4.3]{Lutz2023}}]
   \label{lemma:boundary_of_config_space}
   Consider a canonical tessellation $\tri$ of $\surf_g$. Let $\verts_0$ be
   the subset of vertices $i\in\verts$ such that there is an
   \(
      \bar{\omega}
      \in\partial\mathcal{D}_{\tri}(\surf_g)\setminus\{0\}
   \)
   with $\bar{\omega}_i=0$. The tessellation $\tri$ contains no edge $ij$ with
   both $i,j\in\verts_{0}$.
\end{lemma}

For weights $\omega\in\RR^{\verts}$ we call a geodesic triangulation $\tri$,
which refines the unique canonical tessellation $\tri_{\surf_g}^{\omega}$, a
\emph{canonical triangulation} of $\surf_g$ with respect to the weights $\omega$.
In general, the canonical tessellation $\tri_{\surf_g}^{\omega}$ is not a
triangulation. So canonical triangulations need not be unique. The next lemma
gives us a local characterization of canonical triangulations. It will allow us
to relate canonical triangulations to weighted Delaunay triangulations of
decorated piecewise spherical/hyperbolic surfaces in the next section.

\begin{lemma}[\!{\cite[Lem.~4.5]{BL2023}}]
   \label{lemma:local_characterization_canonical_tessellation}
   Let $\tri$ be a canonical triangulation of $\surf_g$ for the weights
   $\omega\in\RR^{\verts}$. Suppose that the decorating cycles corresponding to
   $\omega$ are pairwise non-intersecting. Denote by $t_{ij}^k$ the (oriented)
   distance of the center of this disk to the edge $ij$. The following
   properties are equivalent.

   \begin{enumerate}[label=(\roman*)]
      \item\label{item:hyperbolic_local_proper_angle}
         The proper hyperbolic disk of the decorated hyperideal triangle $ijk$
         intersects the vertex-cycle at $l$ either not at all or at an angle less
         than $\nicefrac{\pi}{2}$.
      \item\label{item:hyperbolic_local_proper_distance}
         The center of the proper hyperbolic disk of the hyperideal triangle $ijk$
         \enquote{lies to the left} of the center corresponding to $ilj$, \ie,
         \begin{equation}\label{eq:local_canonical_condition}
            t_{ij}^k \,+\, t_{ij}^l \;\geq\; 0.
         \end{equation}
   \end{enumerate}
\end{lemma}

The quantity $t_{ij}^k$ can be computed from the geometry of a decorated
hyperideal triangle only, \ie, it is determined by the lengths of the edges of
the triangle and the radii of the vertex-cycles. This gives us a way to test
whether a given geodesic triangulation of $\surf_g$ is canonical.

\begin{proposition}
    [local characterization of canonical tessellations, {\cite[Prop.~4.6]{BL2023}}]
    \label{prop:canonical_local_to_global}
    A given geodesic triangulation $\tri$ is a canonical tessellations for
    the complete hyperbolic surface $\surf_g$ with given weights
    $\omega\in\RR_{>0}^{\verts}$ if and only if all its edges satisfy the inequality
    \eqref{eq:local_canonical_condition}.
    Furthermore, two canonical triangulations for the same weights differ
    only on edges which satisfy \eqref{eq:local_canonical_condition} with equality.
\end{proposition}

\subsection{Convexity of hyperbolic polyhedra and discrete conformal equivalence}
\label{sec:convex_polyhedra}
Consider a decorated discrete spherical metric $(\tri, \len, r)$ and denote
by $\surf_g$ its fundamental discrete conformal invariant. The dihedral angle
at the edge $ij$ in the base face of the hyperideal pyramid over
the triangle $ijk\in\faces_{\tri}$ is given by $\alpha_{ij}^k$ (see
\figref{fig:sketch_notation_decoration} and \figref{fig:convexity_and_wDt}, top).
Therefore, we call the edge $ij$ \emph{convex} if
$\alpha_{ij}^k+\alpha_{ij}^l\leq\pi$. Here, $ilj$ is the other face incident to
$ij$ in $\tri$. The corresponding polyhedral cone is \emph{convex} if all edges of
$\tri$ are convex. Convexity of polyhedral ends is defined similarly.
We denote the set of all heights of convex polyhedral cones/ends over the triangulated
hyperbolic surface $(\surf_g, \tri)$ by
$\polyhedra_{\tri}^{\pm}(\surf_g)\subset\RR^{\verts}$.

\begin{figure}[!h]
    \centering
    \includegraphics[width=0.77\textwidth]{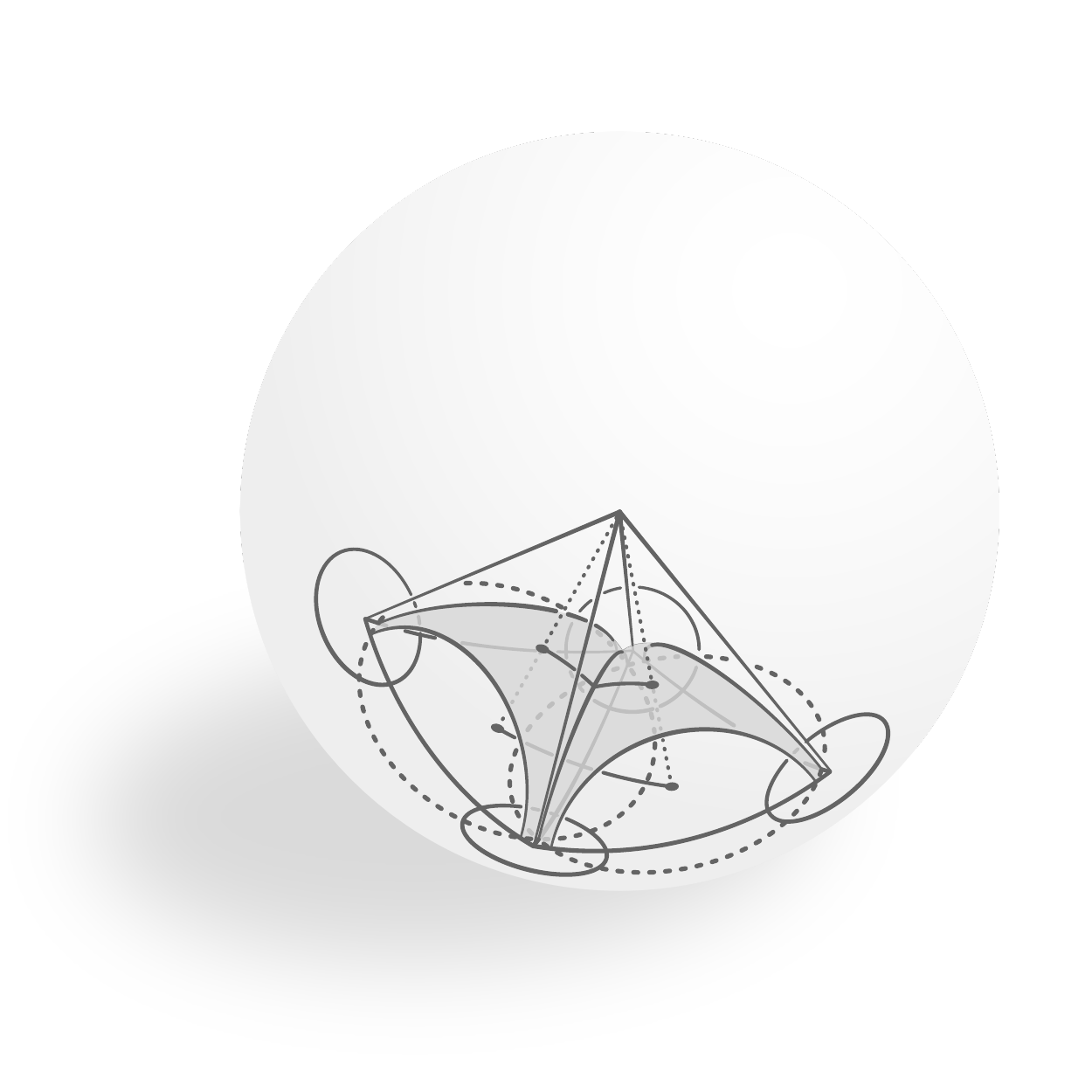}
    \includegraphics[width=0.67\textwidth]{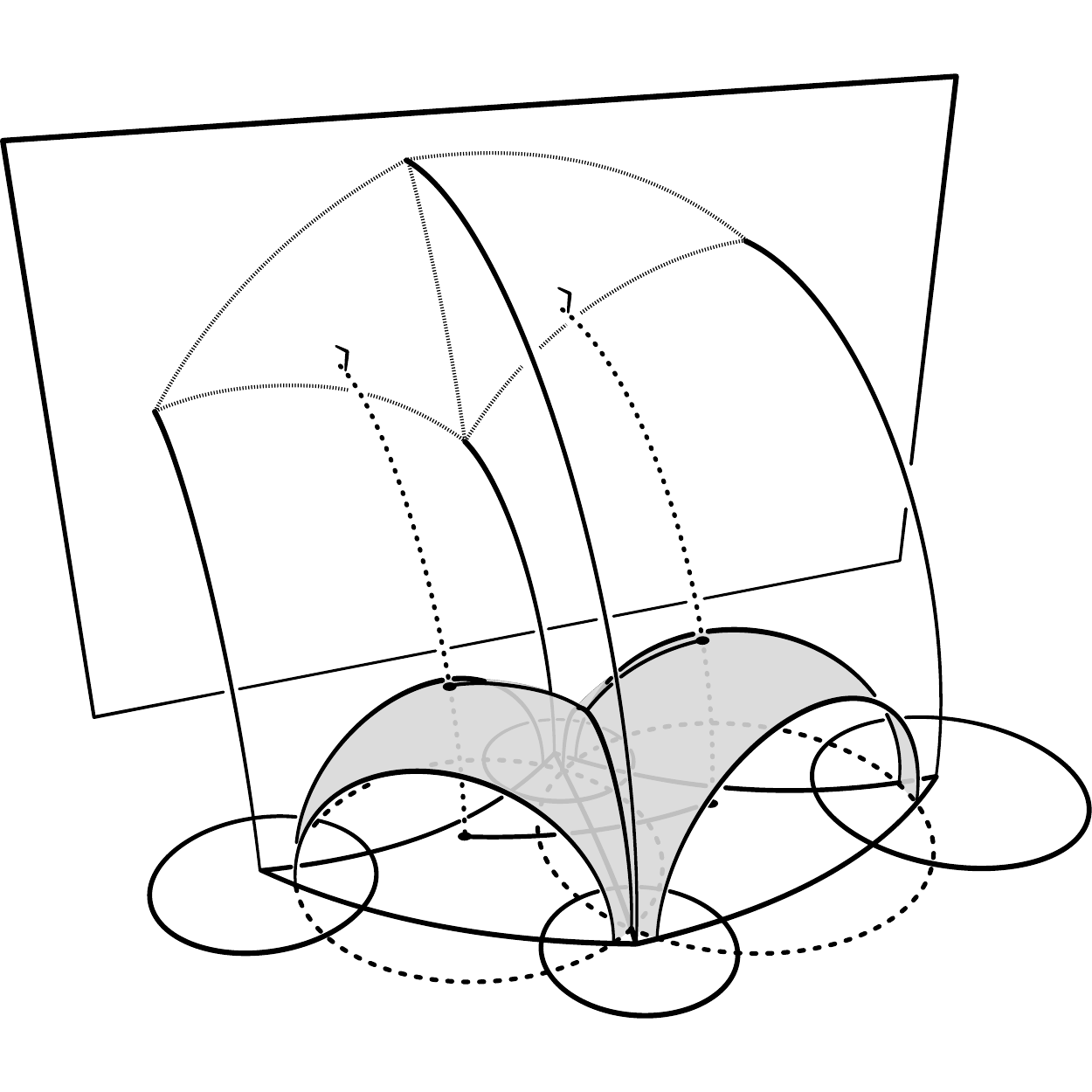}
    \caption{%
        Sketch of two adjacent hyperideal pyramids (top) and hyperideal prisms
        (bottom). Note that the centers of the proper hyperbolic disks project
        to the centers of the spherical/hyperbolic face-circles (dotted lines).
    }
    \label{fig:convexity_and_wDt}
\end{figure}

\begin{lemma}\label{lemma:local_delaunay_convexity_equivalence}
    Let $(\tri, \len, r)$ be a decorated discrete spherical/hyperbolic metric and
    denote by $\surf_g$ its fundamental discrete conformal invariant.
    Consider an edge $ij\in\edges_{\tri}$. The following statements are equivalent:
    \begin{enumerate}[label=(\roman*)]
        \item\label{item:delaunay_to_convexity_wDt}
            $ij$ is locally Delaunay with respect to
            $(\tri, \len, r)$, \ie, it satisfies one of the conditions given in
            \lemref{lemma:local_characterization_wDt},
        \item\label{item:delaunay_to_convexity_convex}
            $ij$ is convex, \ie, $\alpha_{ij}^k+\alpha_{ij}^l\leq\pi$,
        \item\label{item:delaunay_to_convexity_canonical}
            the diagonal $ij$ is locally canonical with respect to $\surf_g$, \ie,
            it satisfies one of the conditions given in
            \lemref{lemma:local_characterization_canonical_tessellation}.
    \end{enumerate}
\end{lemma}
\begin{proof}
    The equivalence of items \ref{item:delaunay_to_convexity_wDt}
    and \ref{item:delaunay_to_convexity_convex} was already considered in
    \lemref{lemma:local_characterization_wDt}. The other equivalence follows
    from an elementary geometric observation for the hyperideal pyramid/prism
    over a triangle $ijk\in\faces_{\tri}$: the distinguished vertex of the
    pyramid/prism, the vertex of the dual diagram corresponding to $ijk$, and the
    center of the face-circle $C_{ijk}$ are collinear, \ie, they lie on the same
    hyperbolic geodesic (see \figref{fig:convexity_and_wDt}).
    This observation, \eqref{eq:local_delaunay_condition_distances},
    and \eqref{eq:local_canonical_condition} imply the equivalence of
    items \ref{item:delaunay_to_convexity_wDt} and
    \ref{item:delaunay_to_convexity_canonical}.

    The proof of this observation proceeds similarly in the spherical and hyperbolic
    case. We will only elaborate the spherical case. Let us consider the hyperideal
    pyramid over $ijk$ in the Poincar\'{e} ball model of hyperbolic $3$-space
    normalized such that the distinguished vertex is $0$. Furthermore, by choosing the
    radius $R$ of the sphere at this vertex large enough we can assume that the
    decorating spheres of the base face are non-intersecting. Let $S$ be the hyperbolic
    sphere with the same hyperbolic center and hyperbolic radius as the hyperbolic
    proper disk defined by the corresponding decoration of the base face.
    By construction, the (Euclidean) sphere representing the hyperbolic plane over the
    face cycle $C_{ijk}$, $S$, and the sphere at $0$ intersect
    the spheres decorating the vertices of the base face orthogonally. Therefore,
    their Euclidean centers lie on a common Euclidean line
    (see \figref{fig:decorated_hyperideal_tetrahedra}): their \emph{radical line},
    also known as, \emph{power line}. In particular, it is the unique line mutually
    orthogonal to the three given spheres. Since this Euclidean line contains $0$ it is
    also a hyperbolic geodesic. Hence, the hyperbolic centers lie on this geodesic,
    too.
\end{proof}

\begin{remark}
   One can also define a radical line in hyperbolic geometry
   (see, \eg, \cite[Sec.~2]{Lutz2023}). This definition can be used to modify our
   proof of \lemref{lemma:local_delaunay_convexity_equivalence} to be completely
   intrinsic. In fact, the Euclidean and hyperbolic radical lines coincide
   for our choice of model and normalization.
\end{remark}

\begin{figure}[t]
    \centering
    \includegraphics[width=0.45\textwidth]{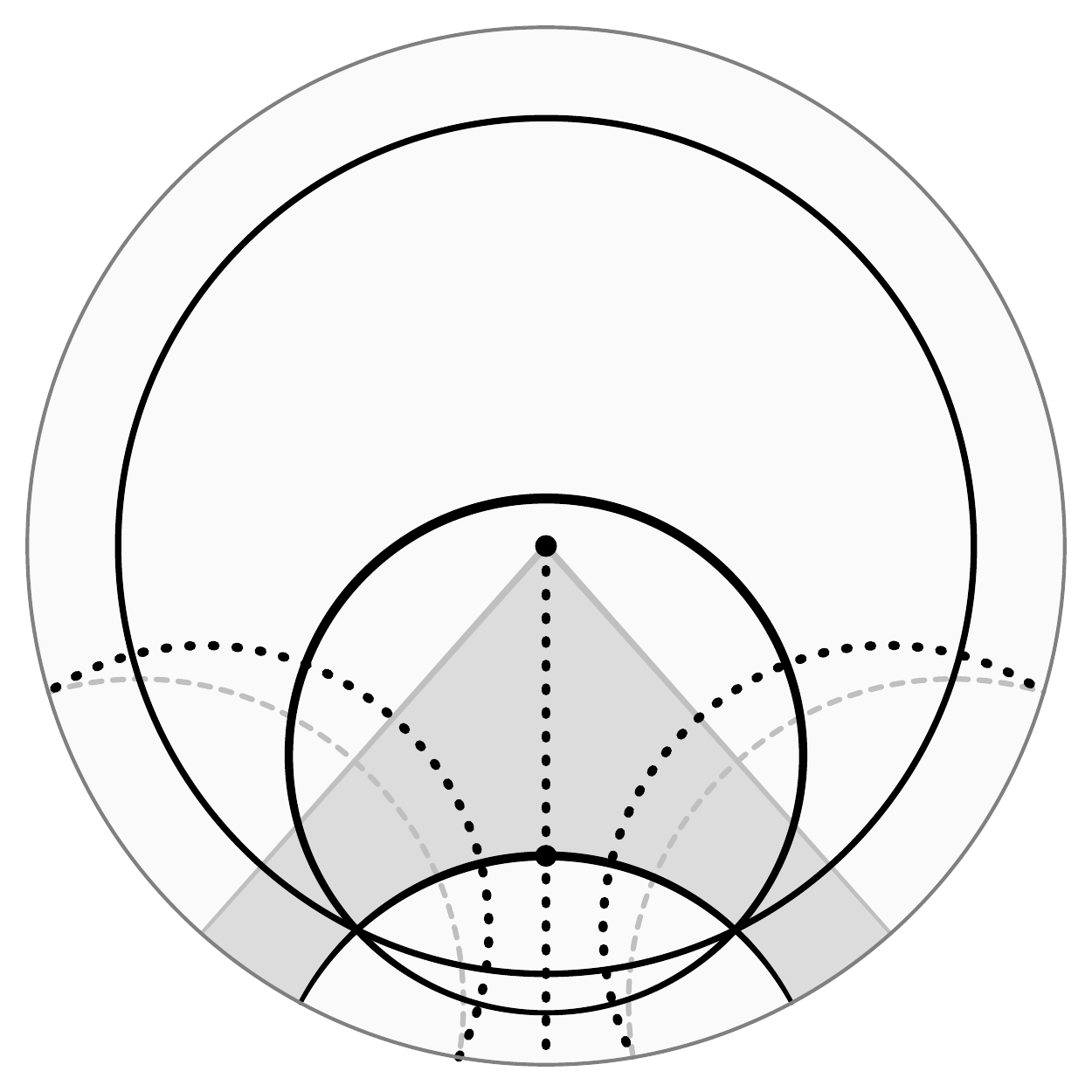}
    \quad\qquad
    \includegraphics[width=0.45\textwidth]{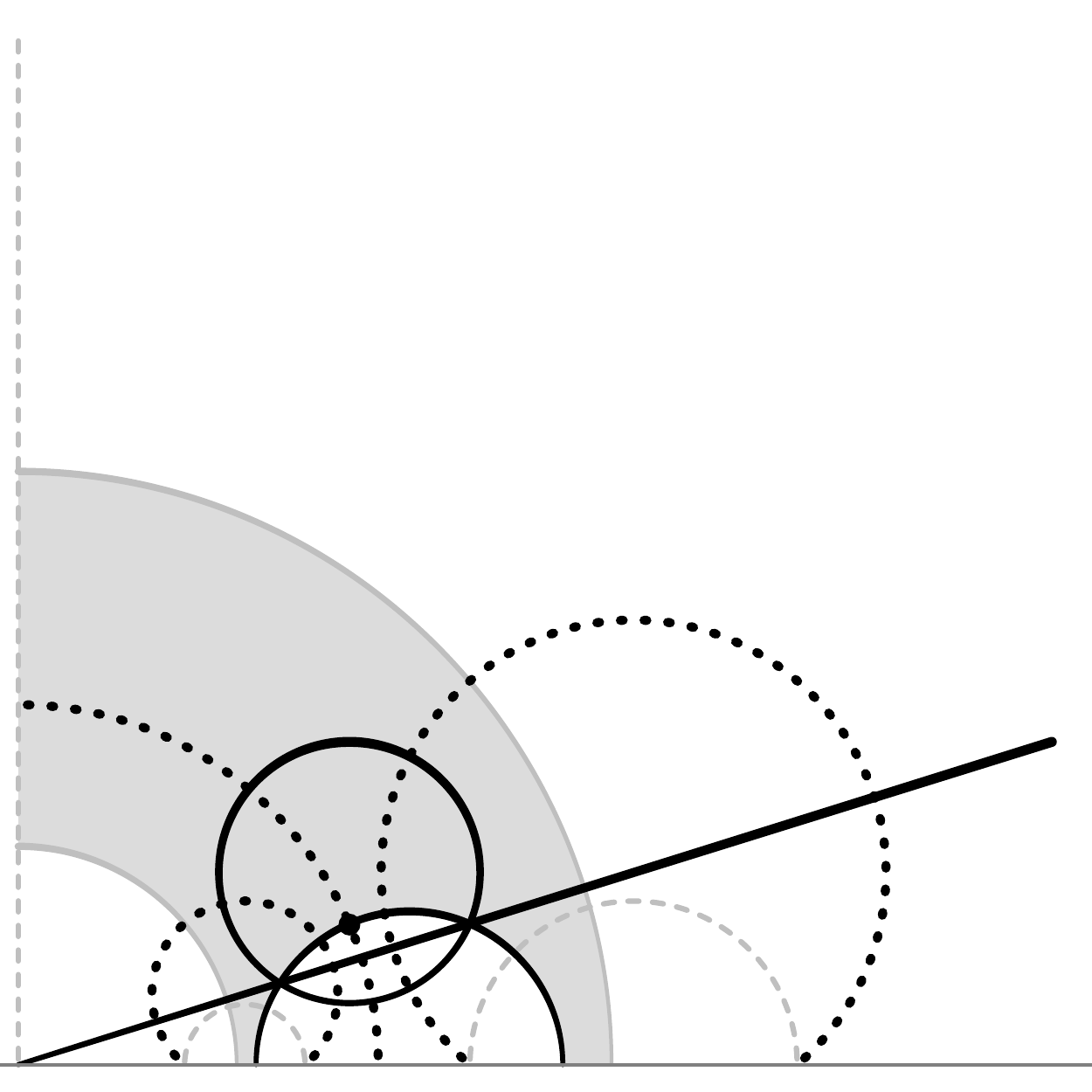}
    \caption{%
        Sketch of the 2-dimensional analogue of the main observation used to
        prove \lemref{lemma:local_delaunay_convexity_equivalence}
        (left: spherical case; right: hyperbolic case). The solid
        and dotted (sections of) circles are contained in \enquote{dual pencils},
        respectively. Thus, they are orthogonal. In particular,
        the center of the solid circle lies on the dotted geodesic.
    }
    \label{fig:decorated_hyperideal_tetrahedra}
\end{figure}

If $\tri$ and $\tilde{\tri}$ are both weighted Delaunay triangulations of
$(\dist_{\toposurf_g}, r)$ then \lemref{lemma:local_delaunay_convexity_equivalence}
shows that they induce the same complete hyperbolic surface $\surf_g$, \ie,
the same fundamental discrete conformal invariant. Thus, we call $\surf_g$ the
\emph{fundamental discrete conformal invariant} of the piecewise spherical/hyperbolic
metric $(\dist_{\toposurf_g}, r)$.

\begin{definition}[discrete conformal equivalence, variable combinatorics]
    \label{def:dce_decoration_hyperbolic_version}
    Let $(\dist_{\toposurf_g}, r)$ and
    $\big(\widetilde{\dist}_{\toposurf_g}, \tilde{r}\big)$ be
    two decorated piecewise spherical/hyperbolic metrics
    on the marked genus $g$ surface $(\toposurf_g, \verts)$. We say that they are
    \emph{discrete conformally equivalent} if they share the same fundamental discrete
    conformal invariant.
\end{definition}

\begin{proposition}
    \label{prop:space_polyhedra_fixed_triangulation}
    The spaces
    $\polyhedra_{\tri}^{\pm}(\dist_{\toposurf_g}, r)$,
    $\polyhedra_{\tri}^{\pm}(\surf_g)$, and
    $\mathcal{D}_{\tri}(\surf_g)$ are homeomorphic.
\end{proposition}
\begin{proof}
    \lemref{lemma:local_characterization_wDt} shows that
    $\tri$ is a weighted Delaunay triangulation of $(\dist_{\toposurf_g}, r)$ if
    and only if the corresponding polyhedral cone/end is convex. Thus,
    \(
        \polyhedra_{\tri}^{\pm}(\dist_{\toposurf_g}, r)
        = \polyhedra_{\tri}^{\pm}(\surf_g)
    \)
    by \propref{prop:dce_via_invariant_spherical} and
    \propref{prop:dce_via_invariant_hyperbolic}. Moreover,
    from \lemref{lemma:local_delaunay_convexity_equivalence} follows that a polyhedral
    cone/end is convex if and only if its heights define weights on $\surf_g$ such
    that $\tri$ is a canonical tessellation. Hence,
    \(
        \Omega_R^{\pm}\big(\polyhedra_{\tri}^{\pm}(\surf_g)\big)
        \subset \mathcal{D}_{\tri}(\surf_g).
    \)

    Finally, the homeomorphism follows from the following two properties of
    $\Omega_R^{\pm}\big(\polyhedra_{\tri}^{\pm}(\surf_g)\big)$:
    first, for all $\omega\in\mathcal{D}_{\tri}(\surf_g)$ there is
    $h\in\polyhedra_{\tri}^{\pm}(\surf_g)$ and $s>0$ such that
    $\Omega_R^{\pm}(h) = s\omega$; and second, from
    $\omega\in\Omega_R^{\pm}\big(\polyhedra_{\tri}^{\pm}(\surf_g)\big)$
    follows that
    $s\omega\in\Omega_R^{\pm}\big(\polyhedra_{\tri}^{\pm}(\surf_g)\big)$
    for all $s\geq1$. These properties can be deduced from the same geometric
    observation discussed in the proof of
    \lemref{lemma:local_delaunay_convexity_equivalence}. Indeed, the
    hyperbolic geodesics corresponding to dropping the perpendicular from the
    distinguished vertex to the base faces are the radical lines of the
    spheres decorating the vertices. They intersect the base faces in the vertices
    of the dual diagram corresponding to that face. As scaling does not change the
    canonical tessellation nor its dual diagram we can construct all polyhedral
    cones/ends corresponding to $s\omega$ from this data. A simple application
    of the hyperbolic law of cosines gives the properties.
\end{proof}

\begin{proposition}
    \label{prop:space_polyhedra_characterization}
    The space of convex polyhedral cones/ends is given by
    \begin{equation}
        \polyhedra^{\pm}(\surf_g)
        \;\coloneq\;
        (\Omega_R^{\pm})^{-1}\big(\RR_{>0}^{\verts}\big)
        \;=\;
        \bigcup\polyhedra_{\tri}^{\pm}(\surf_g).
    \end{equation}
    It is homeomorphic to $\RR_{>0}^{\verts}$. Furthermore, its decomposition
    inherits the properties \ref{item:canonical_space_decomposition}
    and \ref{item:canonical_space_finite} from the space of weights
    (\propref{prop:properties_of_weightings}). In particular, each
    $h\in\polyhedra^{\pm}(\surf_g)$ corresponds to
    a unique convex polyhedral cone/end $P_{\surf_g}^{\pm}(h)$.
\end{proposition}
\begin{proof}
    This is a direct consequence of the previous
    \propref{prop:space_polyhedra_fixed_triangulation} and its proof.
    In the case of convex polyhedral ends it even implies that
    \begin{equation}
        \Omega_R^{-}\big(\polyhedra^{-}(\surf_g)\big)
        \;=\;\big\{\omega\in\RR_{>0}^{\verts} \,:\,
        \omega_i>\nicefrac{1}{\sinh(R)} \text{, if $i$ is hyperideal}\big\}.
        \qedhere
    \end{equation}
\end{proof}

The definition of discrete conformal equivalence in terms of the fundamental
discrete conformal invariant can also be reformulated using sequences of weighted
Delaunay triangulations.

\begin{propdef}\label{prop:dce_surgery_to_hyperbolic}
    Let $(\dist_{\toposurf_g}, r)$ and $(\widetilde{\dist}_{\toposurf_g}, \tilde{r})$
    be two decorated piecewise spherical/hyperbolic metrics on
    the marked surface $(\toposurf_g, \verts)$. They are
    \emph{discrete conformally equivalent} in the sense of
    \defref{def:dce_decoration_hyperbolic_version} if and only if there is a
    sequence of decorated discrete spherical/hyperbolic metrics
    \begin{equation}
        (\tri^0,\len^0,r^0), \dotsc , (\tri^N,\len^N,r^N)
    \end{equation}
    such that:
    \begin{enumerate}[label=(\roman*)]
        \item
            the piecewise spherical/hyperbolic metrics corresponding to
            $(\tri^0,\len^0)$ and $(\tri^N,\len^N)$
            are $\dist_{\toposurf_g}$ and $\widetilde{\dist}_{\toposurf_g}$,
            respectively,
        \item
            each $\tri^n$ is a weighted Delaunay triangulation
            with respect to $(\tri^n, \len^n, r^n)$,
        \item
            if $\tri^n=\tri^{n+1}$, then there are logarithmic scale factors
            $u\in\RR^{\verts}$ such that $(\tri^n,\len^n, r^n)$ and
            $(\tri^{n+1}, \len^{n+1}, r^{n+1})$ are related via
            \eqref{eq:conformal_change_formulas_spherical}
            respectively \eqref{eq:conformal_change_formulas_hyperbolic},
        \item
            if $\tri^n\neq\tri^{n+1}$, then $\tri^n$ and $\tri^{n+1}$ are
            two different weighted Delaunay triangulations of the
            same decorated piecewise spherical/hyperbolic metric.
    \end{enumerate}
\end{propdef}
\begin{proof}
    This is a combination of \propref{prop:dce_via_invariant_spherical},
    \propref{prop:dce_via_invariant_hyperbolic},
    and \lemref{lemma:local_delaunay_convexity_equivalence}.
\end{proof}

\subsection{Proof of the spherical discrete uniformization theorem}
Our proof of the spherical discrete uniformization theorem relies on the following
characterization of complete hyperbolic metrics with ends on the sphere.

\begin{proposition}
    \label{prop:metrics_on_hyperideal_polyhedra}
    Suppose $(\toposurf_0, \verts)$ is a marked genus $0$ surface. Each
    complete hyperbolic metric on $\toposurf_0\setminus\verts$ with cusps and ends
    of infinite area can be realized as a unique hyperideal polyhedron in
    hyperbolic $3$-space, up to hyperbolic isometries.
\end{proposition}

For hyperbolic surfaces with cusps this was proved by \textsc{I.~Rivin}
\cite{Rivin1994a}. Later, \textsc{J.-M.~Schlenker} extended this result to surfaces
with ends of infinite area \cite{Schlenker1998}. See also \cite{Fillastre2008} for
a discussion of related results.

\begin{proof}[Proof of \thmref{thm:discrete_uniformization_spherical}]
    Let $(\dist_{\toposurf_0}, r)$ be a hyperideally decorated piecewise spherical
    metric. By \lemref{lemma:uniqueness_wDt} and \lemref{lemma:existence_wDt},
    there exists a weighted Delaunay triangulation of $(\toposurf_0, \verts)$
    with respect to $(\dist_{\toposurf_0}, r)$. Therefore,
    $(\dist_{\toposurf_0}, r)$ possesses a fundamental discrete conformal invariant
    $\surf_0$. It is a complete hyperbolic surface with cusps and ends of
    infinite area. Hence, \propref{prop:metrics_on_hyperideal_polyhedra} grants that
    we can realize $\surf_0$ as the boundary of a hyperideal polyhedron. This
    polyhedron induces a hyperideally decorated piecewise spherical metric
    $(\widetilde{\dist}_{\toposurf_0}, \tilde{r})$ on the ideal
    boundary of hyperbolic $3$-space, \ie, $\SS^2$. Naturally, the cone angles
    with respect to this metric are all $2\pi$ and it is discrete conformally
    equivalent to $(\dist_{\toposurf_0}, r)$ by construction (see
    \defref{def:dce_decoration_hyperbolic_version}). Finally, the only ambiguity
    is given by isometries of hyperbolic $3$-space, which extend to M\"obius
    transformations of $\SS^2$.
\end{proof}

\subsection{Proof of the hyperbolic prescribed cone-angle theorem}
Let $\surf_g$ be a complete hyperbolic surface with ends.
For $\Theta\in\RR^{\verts}$ the discrete Hilbert--Einstein functional
$\HE_{\surf_g,\Theta}^{\pm}$ over $\polyhedra^{\pm}(\surf_g)$ is given by
\begin{equation}\label{eq:global_he_functional}
    \HE_{\surf_g,\Theta}^{\pm}(h)
    \;\coloneq\;
    \HE_{\surf_g,\Theta,\tri}^{\pm}(h)
    \;\coloneq\;
    \HE_{\len, r, \Theta}^{\pm}(h).
\end{equation}
Here, $\tri$ is a canonical triangulation determined by the weights
$\Omega_R^{\pm}(h)$ and $(\tri, \len, r)$ is the corresponding decorated discrete
metric.

\begin{proposition}[Properties of the dHE-functional]
    \label{prop:he_functional_properties}
    Let $\surf_g$ be a genus $g$ complete hyperbolic surface
    and $\Theta\in\RR_{>0}^{\verts}$.
    \begin{enumerate}[label=(\roman*)]
        \item\label{item:he_functional_differentiability}
            The dHE-functional $\HE_{\surf_g,\Theta}^{\pm}$ is twice continuously
            differentiable over $\polyhedra^{\pm}(\surf_g)$ and analytic in
            each $\polyhedra_{\tri}^{\pm}(\surf_g)$.
        \item\label{item:he_functional_gradient_flow}
            The decorated $\Theta$-flow \eqref{eq:decorated_flow} is the gradient
            flow of $\HE_{\surf_g,\Theta}^{\pm}$.
        \item\label{item:hyperbolic_he_functional_strictly_concave}
            The hyperbolic dHE-functional $\HE_{\surf_g,\Theta}^{-}$ is
            strictly concave.
    \end{enumerate}
\end{proposition}
\begin{proof}
    In \lemref{lemma:well_definedness_he_functional} we will show that
    $\HE_{\surf_g,\Theta}^{\pm}$ is well-defined. The differentiability of
    $\HE_{\surf_g,\Theta}^{\pm}$ is analysed in
    \lemref{lemma:he_functional_total_diff} and
    \lemref{lemma:he_functional_hessian}. This proves item
    \ref{item:he_functional_differentiability}. In particular,
    item \ref{item:he_functional_gradient_flow} is a reformulation of
    \lemref{lemma:he_functional_total_diff}. Finally, the strict concavity of the
    hyperbolic dHE-functional $\HE_{\surf_g,\Theta}^{-}$
    follows from \lemref{lemma:he_functional_hessian} and
    \propref{prop:local_concavity_hyperbolic_he}.
\end{proof}

With these preparations in place we are ready to solve the prescribed cone-angle
problem for piecewise hyperbolic surfaces.

\begin{proof}[%
    Proof of \thmref{theorem:realisation_hyperbolic} (variational principle \&
    uniqueness).
    ]
    Item \ref{item:he_functional_gradient_flow} of
    \propref{prop:he_functional_properties} shows that realizations of
    $\Theta$ are given by critical points of the dHE-functional
    $\HE_{\surf_g,\Theta}^{-}$.
    Since  $\HE_{\surf_g,\Theta}^{-}$ is strictly concave
    (\propref{prop:he_functional_properties} item
    \ref{item:hyperbolic_he_functional_strictly_concave}), it possesses at most
    one critical point, \ie, a maximum point.
\end{proof}

\begin{proof}[%
    Proof of \thmref{theorem:realisation_hyperbolic} (existence)
    ]
    Its left to show that $\HE_{\surf_g,\Theta}^{-}$ has a maximum.
    First, it is clear that $\HE_{\surf_g,\Theta}^{-}$ can only have a maximum
    if $\Theta$ satisfies the hyperbolic Gau{\ss}--Bonnet condition
    \eqref{eq:hyperbolic_gauss_bonnet_condition} since each
    $h\in\polyhedra^{-}(\surf_g)$ induces a piecewise hyperbolic surface.

    So suppose $\Theta$ satisfies these conditions. We will show that there
    is an $M>0$ such that
    \begin{equation}
        \sup_{h\in\polyhedra^{-}(\surf_g)}\HE_{\surf_g,\Theta}^{-}(h)
        \;=\;
        \sup_{\substack{h\in\polyhedra^{-}(\surf_g)\\\|h\|\leq M}}
        \HE_{\surf_g,\Theta}^{-}(h),
    \end{equation}
    Since the ball $\{\|h\|\leq M\}$ is compact, the supremum is attained at some
    $\bar{h}\in\cl\{h\in\polyhedra^{-}(\surf_g):\|h\|\leq M\}$. Now,
    \lemref{lemma:hyperbolic_he_hight_bound} guarantees that
    $\bar{h}\in\polyhedra^{-}(\surf_g)$.

    To see that such an $M$ exists we will first prove in
    \lemref{lemma:hyperbolic_he_hight_bound} that the
    heights can be bounded from below. Finally, the existence of an upper bound
    for the heights follows from the structure of $\polyhedra^{-}(\surf_g)$
    (\propref{prop:space_polyhedra_characterization}) together with the
    \enquote{scaling-behavior} of $\HE_{\surf_g, \Theta}^{-}$
    (\lemref{lemma:weight_relationship} and \lemref{lemma:hyperbolic_he_fiber_limit}).
\end{proof}

\subsection{Properties of the discrete Hilbert--Einstein functional}
\begin{lemma}[well-definedness of $\HE_{\surf_g, \Theta}^{\pm}$]
    \label{lemma:well_definedness_he_functional}
    The dHE-functional is well-defined, that is, if $h\in\polyhedra^{\pm}(\surf_g)$ and
    $\tri$ and $\tilde{\tri}$ are triangulations refining the canonical tessellation
    corresponding to the weights $\Omega_R^{\pm}(h)$, then
    \begin{equation}
        \HE_{\surf_g,\Theta,\tri}^{\pm}(h)
        \;=\;
        \HE_{\surf_g,\Theta,\tilde{\tri}}^{\pm}(h).
    \end{equation}
\end{lemma}
\begin{proof}
    The volume and cone-angles $\theta_i$ are intrinsic quantities of
    $P_{\surf_g}^{\pm}(h)$. Furthermore, if $ij$ is an edge of $\tri$
    but not $\tilde{\tri}$, then $ij$ is contained in a face of the canonical
    tessellation corresponding to $\Omega_R^{\pm}(h)$. It follows from
    \lemref{lemma:local_delaunay_convexity_equivalence}
    that $\alpha_{ij}=\pi$. So this edge is not contributing to
    $\HE_{\surf_g,\Theta,\tri}^{\pm}$.
\end{proof}

\begin{lemma}[first derivative of $\HE_{\surf_g,\Theta}^{\pm}$]
\label{lemma:he_functional_total_diff}
    The derivative of the dHE-functional $\HE_{\surf_g,\Theta}^{\pm}$
    is given by
    \begin{equation}\label{eq:he_functional_total_diff}
        \diff{}{\HE_{\surf_g,\Theta}^{\pm}}
        \;=\; \sum(\Theta_i-\theta_i)\,\diff{}{h_i}.
    \end{equation}
\end{lemma}
\begin{proof}
    We obtain \eqref{eq:he_functional_total_diff} from
    \propref{prop:local_variational_principle} by considering any
    canonical triangulation corresponding to $\Omega_R^{\pm}(h)$.
    Since $\theta_i$ does not depend on the canonical triangulation we chose,
    so does not \eqref{eq:he_functional_total_diff}.
\end{proof}

\begin{lemma}[second derivative of $\HE_{\surf_g,\Theta}^{\pm}$]
   \label{lemma:he_functional_hessian}
   Consider $h\in\polyhedra^{\pm}(\surf_g)$. Choose any canonical
   triangulation $\tri$ determined by $\Omega_R^{\pm}(h)$ and denote the
   corresponding decorated discrete metric by $(\tri, \len, r)$.
   The second derivative of the dHE-functional $\HE_{\surf_g,\Theta}^{\pm}$
   at $h$ is given by \eqref{eq:explicit_formula_jacobian} with respect to
   $(\tri, \len, r)$.
\end{lemma}
\begin{proof}
    We only have to show that formula \eqref{eq:explicit_formula_jacobian} does
    not depend on the particular choice of canonical triangulation.
    We observe that $\cotw_{ij}=0$ if and only if $\alpha_{ij}=\pi$. So we can
    argue as in the proof of \lemref{lemma:well_definedness_he_functional}.
\end{proof}

\begin{lemma}[lower bound for the heights]
    \label{lemma:hyperbolic_he_hight_bound}
    Denote by $\verts_0$ and $\verts_1$ the sets of ideal and hyperideal vertices of
    $\surf_g$, respectively.
    There are $\varepsilon_0, \varepsilon_1>0$ such that
    \begin{equation}
       \sup_{h\in\polyhedra^{-}(\surf_g)}\HE_{\surf_g,\Theta}^{-}(h)
        \;\leq\;
        \sup_{h\in\polyhedra_{\varepsilon_0, \varepsilon_1}^{-}(\surf_g)}
        \HE_{\surf_g,\Theta}^{-}(h),
    \end{equation}
    Here,
    \(
        h
        \in
        \polyhedra_{\varepsilon_0,\varepsilon_1}^{-}(\surf_g)
        \subset
        \polyhedra^{-}(\surf_g)
    \)
    if
    $h_i\geq-\varepsilon_0$ for all $i\in\verts_0$ and
    $h_i\geq\varepsilon_1$ for all $i\in\verts_1$.
\end{lemma}
\begin{proof}
    We are going to show in \lemref{lemma:bound_on_hyperideal_heights} and
    \lemref{lemma:bound_on_ideal_heights} that for each $i\in\verts$ and each
    canonical triangulation
    $\tri$ there is an $\varepsilon^i_{\Theta, \tri}$ such that the cone-angle
    $\theta_i(h)<\Theta_i$ for all $h\in\polyhedra_{\tri}^{-}(\surf_g)$ with
    $h_i<\varepsilon^i_{\Theta,\tri}$. Then \lemref{lemma:he_functional_total_diff}
    grants that we can increase $\HE_{\surf_g,\Theta,\tri}^{-}$ at each of these $h$
    by increasing $h_i$. This completes the proof since
    there are only finitely many canonical triangulations of $\surf_g$ by
    \propref{prop:properties_of_weightings}.
\end{proof}

\begin{lemma}\label{lemma:bound_on_hyperideal_heights}
    Let $\tri$ be a canonical triangulation of $\surf_g$ and
    $(h^n)_{n\geq0}\in\polyhedra_{\tri}^{-}(\surf_g)$. If $i\in\verts$ is
    hyperideal and  $\lim_{n\to\infty}h_i^n=0$, then
    $\lim_{n\to\infty}\theta_i(h^n)=0$.
\end{lemma}
\begin{proof}
    It is enough to prove this for a single triangle $ijk\in\faces_{\tri}$.
    The angle $\theta_{jk}^i(h^n)$
    is the dihedral angle of the hyperideal tetrahedron over $ijk$ at the edge
    connecting the vertex $i$ with the distinguished hyperideal vertex
    (see \secref{sec:hyperbolic_polyhedra_hyperbolic}). We can compute it
    by considering the link at the vertex $i$. It is a hyperbolic triangle.
    The opposite of $\theta_{jk}^i(h^n)$ in this triangle is completely determined by
    the geometry of the base face, \ie, the $\lambda$-lengths. Thus, its length
    is fixed. The length $\varphi_i^j(h^n)$ of the adjacent corresponding to the
    edge $ij$ can be computed from the law of cosines for hyperideal triangles.
    In particular,
    \begin{equation}\label{eq:link_length_hyperideal}
        \cosh\varphi_i^j(h^n)
        \;=\;
        \frac{
            \tau_{\epsilon_j}(h_j^n)
            \,+\,\cosh(h_i^n)\,\tau_{\epsilon_i\epsilon_j}(\lambda_{ij})
        }{
            \sinh(h_i^n)\,\tau'_{\epsilon_i\epsilon_j}(\lambda_{ij})
        }.
    \end{equation}
    Here, $\tau_{y}(x)\coloneq\nicefrac{(\ee^{x}+y\ee^{-x})}{2}$ and
    $\tau'_{y}$ is its derivative. It follows that
    $\lim_{h_i^n\to0}\varphi_i^j(h^n)=\infty$.
    The length $\varphi_i^k(h^n)$ of the other adjacent exhibits the same
    limiting behavior. Hence, again by using the law of cosines, we
    get $\lim_{h_i\to0}\theta_{jk}^i(h^n) = 0$.
\end{proof}

\begin{lemma}\label{lemma:bound_on_ideal_heights}
    Let $\tri$ be a canonical triangulation of $\surf_g$ and
    $(h^n)_{n\geq0}\in\polyhedra_{\tri}^{-}(\surf_g)$.
    Define the subset $\verts_{\infty}$ of
    vertices $i\in\verts$ with $\lim_{n\to\infty}h_i^n=-\infty$.
    For each $i\in\verts_{\infty}$ follows $\lim_{n\to\infty}\theta_i(h^n)=0$.
\end{lemma}
\begin{proof}
    The idea of this proof is similar to the proof of
    \lemref{lemma:bound_on_hyperideal_heights}. We observe that $i\in\verts_{\infty}$
    if and only if $i$ is ideal. So this time the law of cosines reads
    \begin{equation}\label{eq:link_length_ideal}
        \cosh\varphi_i^j(h^n)
        \;=\;
        \frac{
            2\tau_{\epsilon_j}(h_j^n)
        }{
            \ee^{h_i^n}\,\tau'_{\epsilon_i\epsilon_j}(\lambda_{ij})
        }
        \;+\;
        \frac{
            \tau_{\epsilon_i\epsilon_j}(\lambda_{ij})
        }{
            \tau'_{\epsilon_i\epsilon_j}(\lambda_{ij})
        }.
    \end{equation}
    Let $\omega^n\coloneq\Omega_R^{-}(h^n)$. Then
    \(
        \nicefrac{2\tau_{\epsilon_j}(h_j^n)}{\ee^{h_i^n}}
        =
        \nicefrac{\omega_j^n}{\omega_i^n}
    \).
    Thus, \lemref{lemma:boundary_of_config_space} grants that
    $\lim_{n\to\infty}\varphi_i^j(h^n)=\infty$. It follows
    $\lim_{h_i\to0}\theta_{jk}^i(h^n) = 0$ as in
    \lemref{lemma:bound_on_hyperideal_heights}.
\end{proof}

\begin{lemma}\label{lemma:weight_relationship}
    Let $\tri$ be a canonical triangulation of $\surf_g$ and write
    $\omega \coloneq \Omega_R^{-}(h)$ for $h\in\polyhedra_{\tri}^{-}(\surf_g)$.
    For every $i\in\verts$ and $\varepsilon>0$ there is an $M_{i, \varepsilon}>0$
    such that $\theta_i(h)<\varepsilon$ whenever
    $\nicefrac{\omega_j}{\omega_i}>M_{i, \varepsilon}$ for all $j\in\verts$ adjacent
    to $i$ with respect to $\tri$.
\end{lemma}
\begin{proof}
    The assertions of this lemma lie at the heart of the proofs of
    \lemref{lemma:bound_on_hyperideal_heights} and
    \lemref{lemma:bound_on_ideal_heights}. We will not repeat the details.
\end{proof}

\begin{lemma}\label{lemma:hyperbolic_he_fiber_limit}
    Let $\Theta\in\RR_{>0}^{\verts}$ satisfy the hyperbolic
    Gau{\ss}--Bonnet condition \eqref{eq:hyperbolic_gauss_bonnet_condition}.
    Define $h^t \coloneq (\Omega_R^{-})^{-1}\big(t\,\Omega_R^{-}(h^1)\big)$ for
    $h^1\in\polyhedra^{-}(\surf_g)$ and all $t\geq1$.
    Then $\lim_{t\to\infty}\HE_{\surf_g,\Theta}^{-}(h^t)=-\infty$.
\end{lemma}
\begin{proof}
    We begin by observing that $|h_i^t-h_j^t|$ is bounded from above by
    \(
        |\ln\big(\nicefrac{\tau_{\epsilon_i}(h_i^1)}{\tau_{\epsilon_j}(h_j^1)}\big)|
    \)
    for all $t\geq1$. This follows from a direct computation and the identity
    \begin{equation}
        \exp\big((\Omega_R^{-})^{-1}(\omega_i)\big)
        \;=\;
        \omega_i\sinh(R)\,+\, \sqrt{\omega_i^2\sinh^2(R)-\epsilon_i}.
    \end{equation}
    Thus, we find a $c_1>0$ and for each $t\geq1$ an $H^t\in\RR$ such that
    $|H^t-h_i^t|<c_1$ for all $i\in\verts$ and $t\geq1$. Note that
    necessarily $\lim_{t\to\infty}H^t=\infty$.

    Furthermore, all $\theta(h^t)$ satisfy the hyperbolic Gau{\ss}--Bonnet
    condition \eqref{eq:hyperbolic_gauss_bonnet_condition} by construction.
    So there is $c_2>0$ such that $|\Theta_i-\theta_i(h^t)|<c_2$ for all
    $i\in\verts$ and $t\geq1$.
    We compute
    \begin{align}
        \HE_{\surf_g,\Theta}^{-}(h^t)
        \;&=\;
        \HE_{\surf_g,\Theta,\tri}^{-}(h^t)\\
        \;&=\;
        -2\vol(P_{\surf_g}^{-}(h^t))
        \,+\, \sum_{i\in\verts}\;\,(\Theta_i-\theta_i^t)h_i^t
        \,+\, \sum_{ij\in\edges_\tri}(\pi-\alpha_{ij}^t)\lambda_{ij}\\
        \;&\leq\; -2\vol(P_{\surf_g}^{-}(h^t))
        \,+\, \Big(\sum_{i\in\verts}(\Theta_i-\theta_{i}^t)\Big)\,H^t
        \,+\, c_1c_2\,|\verts|
        \,+\, \pi|\edges_\tri|\max_{ij\in\edges_\tri}\lambda_{ij}.
    \end{align}
    Here, we abbreviated the notation of the cone-angles and
    intersection angles of face-circles determined by $h^t$ to
    $\theta_i^t$ and $\alpha_{ij}^t$, respectively.
    Moreover, \propref{prop:properties_of_weightings} allowed us
    to consider a fix canonical triangulation $\tri$.
    Consequently, we only have to show that
    \[
        \lim_{t\to\infty}\Big(\sum(\Theta_i-\theta_{i}^t)\Big)\,H^t
        \;=\;-\infty
    \]
    to prove this lemma. This follows from \lemref{lemma:he_angle_limit}:
    on one hand $\Theta$ satisfies the hyperbolic Gau{\ss}--Bonnet
    condition, by assumption. On the other hand $\theta(h^t)$ satisfies
    \eqref{eq:euclidean_gauss_bonnet_condition} in the limit. It follows
    that $\lim_{t\to\infty}\sum(\Theta_i-\theta_{i}^t) < 0$.
\end{proof}

\begin{lemma}\label{lemma:he_angle_limit}
    Define $h^t \coloneq (\Omega_R^{\pm})^{-1}\big(t\,\Omega_R^{\pm}(h^1)\big)$ for
    $h^1\in\polyhedra^{\pm}(\surf_g)$ and all $t\geq1$.
    \begin{enumerate}[label=(\roman*)]
        \item\label{item:limit_non_degenerate}
            $\lim_{t\to\infty}\theta_{jk}^i(h^t)\eqcolon\vartheta_{jk}^i>0$
            for all corners $\corner{i}{jk}$ in triangles $ijk\in\faces_{\tri}$.
        \item\label{item:limit_euclidean_triangles}
            For each triangle $ijk\in\faces_{\tri}$
            the limits $(\vartheta_{jk}^i, \vartheta_{ki}^j, \vartheta_{ij}^k)$ are
            the angles of a euclidean triangle, \ie,
            $\vartheta_{jk}^i+\vartheta_{ki}^j+\vartheta_{ij}^k=\pi$.
        \item\label{item:limit_to_gauss_bonnet}
            The limit of the cone angles
            $\lim_{t\to\infty}\theta(h^t)\eqcolon\vartheta\in\RR_{>0}^{\verts}$
            satisfies the \emph{euclidean Gau{\ss}--Bonnet condition}
            \begin{equation}\label{eq:euclidean_gauss_bonnet_condition}
                \frac{1}{2\pi}\sum\vartheta_i
                \;=\;
                2g-2\,+\,|\verts|.
            \end{equation}
    \end{enumerate}
\end{lemma}
\begin{proof}
    We will only elaborate the hyperbolic case. Its spherical counterpart
    can be proved using exactly the same ideas. Furthermore, note that
    we can fix a canonical triangulation $\tri$ for the rest of the proof
    because of \propref{prop:properties_of_weightings}.

    Let us begin by consider the lengths $\varphi_i^j(h^t)$
    of the link at $i\in\verts$ in the polyhedral end $P_{\surf_g}^{-}(h^t)$.
    Using \eqref{eq:link_length_hyperideal} and \eqref{eq:link_length_ideal}
    we see that $t\mapsto\varphi_i^j(h^t)$ is monotonically decreasing and bounded
    from below. Thus, the law of cosines implies item
    \ref{item:limit_non_degenerate}.

    Next, we remember that the $\lambda$-lengths are fixed and note that
    $\lim_{t\to\infty}h_i^t=\infty$ for all $i\in\verts$. Hence,
    $\lim_{t\to0}\len_{ij}=0$ follows from
    \eqref{eq:heights_lambda_lengths_relationship_hyperbolic} for all edges
    $ij\in\edges_{\tri}$. This proves item \ref{item:limit_euclidean_triangles}
    since the area of a hyperbolic triangle is given by $\pi$ minus the sum
    of its interior angles.

    Finally, item \ref{item:limit_to_gauss_bonnet} is a simple implication of
    item \ref{item:limit_euclidean_triangles}.
\end{proof}

\section{Connections to the Euclidean theory and geometric transitions}
\label{sec:geometric_transitions}
So fare we only considered piecewise non-Euclidean metrics. Yet, decorations can
naturally be considered in the case of piecewise Euclidean metrics. Their
discrete conformal equivalence can be defined in a similar way, both via
M\"obius transformations and discrete logarithmic scale factors
\cite{BL2023}. Moreover, there is also a Euclidean version of the discrete
Hilbert--Einstein functional. Again it is the Legendre-transformation of the
volume of certain piecewise hyperbolic $3$-manifolds, so-called
\emph{convex polyhedral cusps}. This time, they can be constructed from hyperideal
tetrahedra with one distinguished ideal vertex, \ie, \emph{hyperideal horoprisms}
(see \figref{fig:hyperideal_horoprisms}, left).

The constructions presented in the previous section suggest a close
connection between the Euclidean and non-Euclidean theory of decorated
discrete conformal equivalence. Indeed, we can obtain the Euclidean theory by
taking certain limits of its non-Euclidean counterpart. Two key observations lead to
this claim: the angle defect of spherical/hyperbolic triangles decreases with their
edge-lengths and hyperideal pyramids/prisms approximate hyperideal horoprisms if we
let their heights increase. These observations correspond to the intuition that
infinitesimally the sphere and hyperbolic space \enquote{look like} Euclidean space.

To make this precise let us introduce the \emph{ideal boundary}
$\partial_{\infty}\polyhedra^{\pm}(\surf_g)$ of the space of convex
hyperideal cones/ends $\polyhedra^{\pm}(\surf_g)$. It consists of equivalence classes
of strictly diverging sequences $(h^n)_{n\geq0}\in\polyhedra^{\pm}(\surf_g)$, \ie,
sequences with $\lim_{n\to\infty}h_i^n=\infty$ for all $i\in\verts$. Two such
sequences are said to be equivalent if their associated sequences
$\nicefrac{\Omega_R^{\pm}(h^n)}{\|\Omega_R^{\pm}(h^n)\|}$ have the same limit
in the unit sphere in $\RR^{\verts}$.

\begin{theorem}[geometric transitions]
    \label{theorem:geometric_transitions}
    Let $\surf_g$ be a complete hyperbolic surface with ends and let
    $\Theta\in\RR_{>0}^{\verts}$ satisfy the Euclidean Gau{\ss}--Bonnet condition
    \eqref{eq:euclidean_gauss_bonnet_condition}.

    There exists a unique convex polyhedral cusp over $\surf_g$ with
    cone-angles $\Theta$. It corresponds to the unique maximizer
    $\mathfrak{h}\in\partial_{\infty}\polyhedra^{-}(\surf_g)$ of
    $\HE_{\surf_g,\Theta}^{-}$ over
    $\polyhedra^{-}(\surf_g)\cup\partial_{\infty}\polyhedra^{-}(\surf_g)$.
    In particular, $\HE_{\surf_g, \Theta}^{-}$ extends to
    $\partial_{\infty}\polyhedra^{-}(\surf_g)$ (at least to second order). There
    it coincides with the Euclidean dHE-functional $\HE_{\surf_g,\Theta}^0$.
\end{theorem}
\begin{proof}
    In \propref{prop:infty_bdy_polyhedral_cusps} we will explain the correspondence of
    $\partial_{\infty}\polyhedra^{-}(\surf_g)$ to convex polyhedral cusps.
    Afterwards we are going to show that
    $\HE_{\surf_g, \Theta}^{-}$ extends to $\partial_{\infty}\polyhedra^{-}(\surf_g)$
    in \lemref{lemma:he_extension_continuous}. The differential properties
    of this extension are discussed in \lemref{lemma:he_extension_hessian}.

    Now note that $\HE_{\surf_g, \Theta}^{-}$ can not attain its supremum
    in $\polyhedra^{-}(\surf_g)$ because each point of $\polyhedra^{-}(\surf_g)$
    corresponds to a convex polyhedral end, \ie, its associated cone-angles satisfy the
    hyperbolic Gau{\ss}--Bonnet condition \eqref{eq:hyperbolic_gauss_bonnet_condition}.
    Neither can it correspond to a
    boundary-point of $\Omega_R^{-}\big(\polyhedra^{-}(\surf_g)\big)$. This can
    be seen by repeating the argument given in
    \lemref{lemma:hyperbolic_he_hight_bound} and
    \lemref{lemma:weight_relationship}. So we obtain that there is some
    $\mathfrak{h}\in\partial_{\infty}\polyhedra^{-}(\surf_g)$
    with $\sup\HE_{\surf_g, \Theta}^{-}(h) = \HE_{\surf_g, \Theta}^{-}(\mathfrak{h})$.
    Finally, this $\mathfrak{h}$ is uniquely determined because
    $\HE_{\surf_g, \Theta}^{-}$ depends continuously on $\Theta$ and has unique
    maximizers in the case of $\Theta$'s satisfying the hyperbolic
    Gau{\ss}--Bonnet condition.
\end{proof}

\begin{remark}
    The strict concavity of $\HE_{\surf_g, \Theta}^{-}$ is essential for proving
    the uniqueness above. Nonetheless, the spherical dHE-functional exhibits a
    limiting-behavior similar to its hyperbolic counterpart (see subsequent
    discussion). Actually, this allows us to analyse the signature of the
    Hessian of $\HE_{\surf_g, \Theta}^{+}$ in a neighborhood of the ideal boundary
    $\partial_{\infty}\polyhedra^{+}(\surf_g)$.
    Using \lemref{lemma:local_hessian_formula} and \lemref{lemma:he_extension_hessian}
    we see that it is $(1, |\verts|-1)$ close to the Euclidean limit.
    In particular, the Hessian is not degenerate.
\end{remark}

In \cite[Thm.~A]{BL2023} the authors gave an explicit variational solution to
the Euclidean prescribed cone-angle problem via the Euclidean dHE-functional
$\HE_{\surf_g, \Theta}^0$. Because of the relationship between convex polyhedral cusps
and decorated piecewise Euclidean metrics, our previous theorem provides yet
another variational solution to this problem, this time via geometric transitions.

\begin{corollary}[Euclidean prescribed cone-angle problem]
    Let $(\dist_{\toposurf_g}, \mathfrak{r})$ be a hyperideally decorated piecewise
    Euclidean metric on the marked genus $g\geq0$ surface $(\toposurf_g, \verts)$
    and let $\Theta\in\RR_{>0}^{\verts}$ satisfy the Euclidean Gau{\ss}--Bonnet
    condition \eqref{eq:euclidean_gauss_bonnet_condition}.

    There exists a unique decorated piecewise Euclidean metric, up to scale,
    discrete conformally equivalent to $(\dist_{\toposurf_g}, \mathfrak{r})$ realizing
    $\Theta\in\RR_{>0}^{\verts}$. The logarithmic scale factors, which give
    to the change of metric, correspond to the maximum point of a strictly
    concave functional.
\end{corollary}

\begin{figure}[t]
    \centering
    \includegraphics[width=0.575\textwidth]{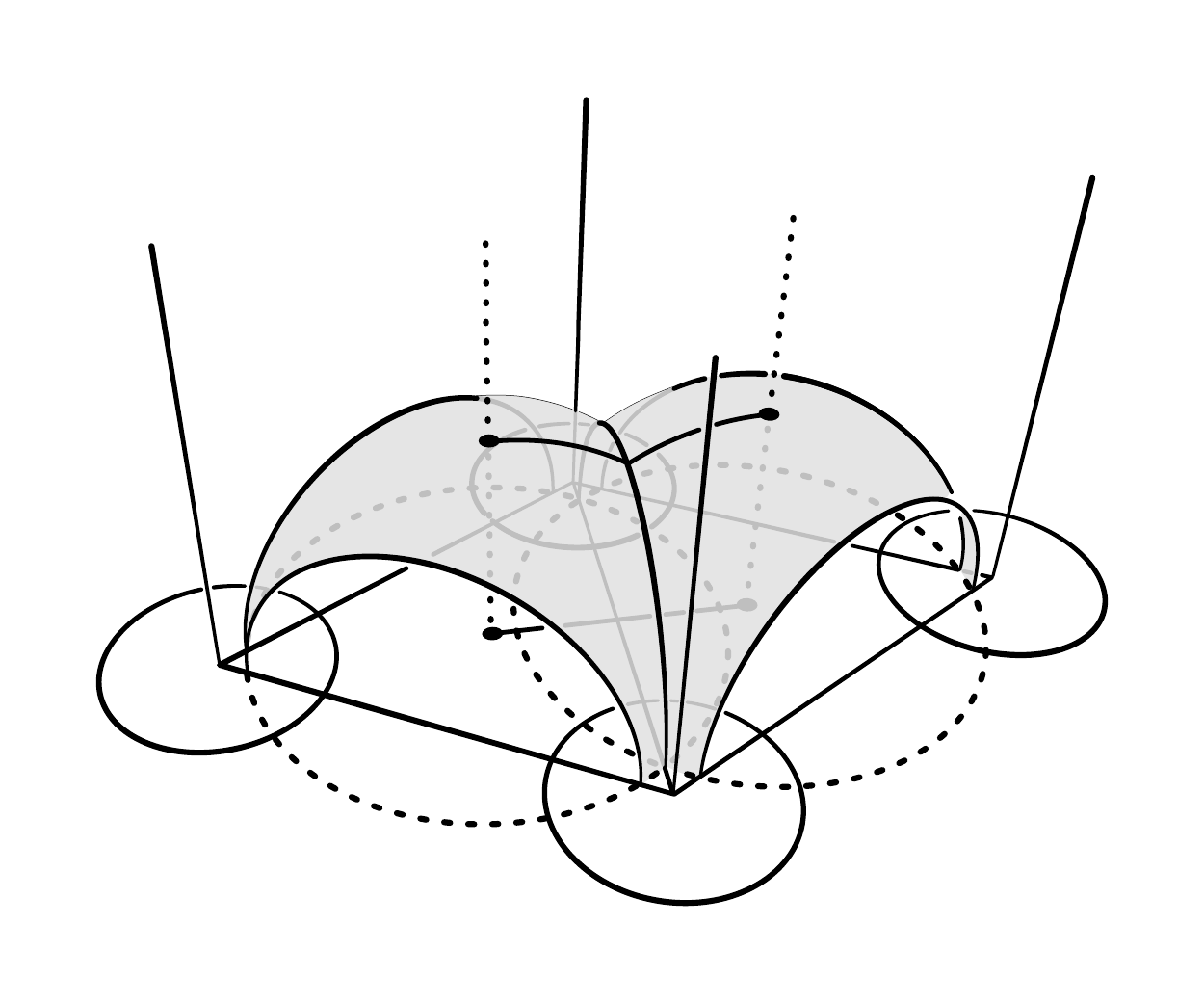}
    \quad
    \labellist
    \small\hair 2pt
    \pinlabel $i$ at 160 50
    \pinlabel $j$ at 440 50
    \pinlabel $\infty$ at 300 545
    \pinlabel $\mathfrak{h}_i$ at 105 250
    \pinlabel $\mathfrak{h}_j$ at 495 250
    \pinlabel $\mathfrak{h}_{ij}$ at 335 250
    \pinlabel $\lambda_{ij}$ at 305 110
    \pinlabel $\mathfrak{l}_{ij}$ at 305 390
    \endlabellist
    \includegraphics[width=0.39\textwidth]{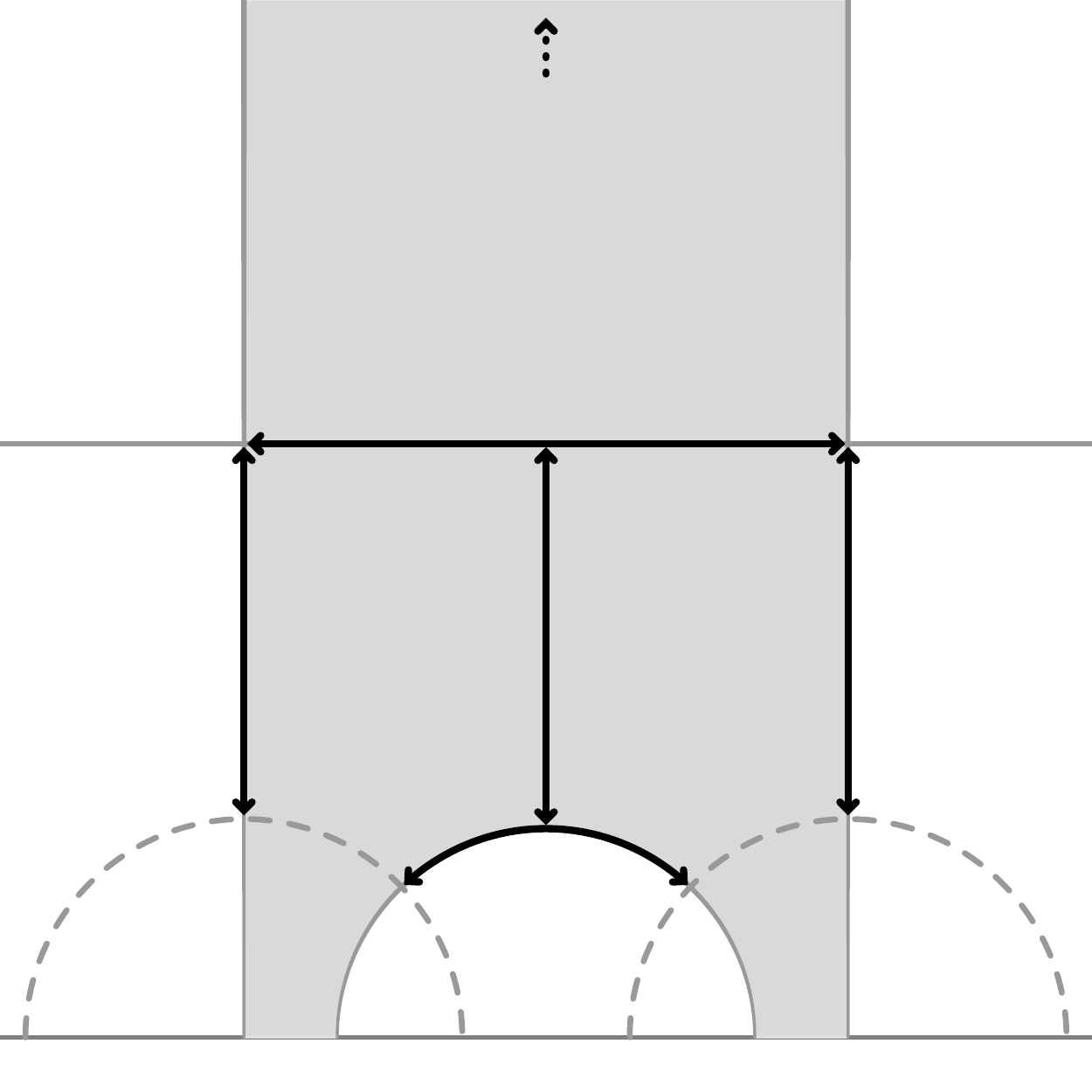}
    \caption{%
        \textsc{Left}: two adjacent hyperideal horoprisms in the upper half-plane
        of hyperbolic $3$-space. The common distinguished vertex is given by $\infty$.
        As in the case hyperideal prisms the angle at the common edge
        $\leq\pi$ if and only if the edge is locally canonical.
        \textsc{Right}: notation for a hyperideal horoprism.
    }
    \label{fig:hyperideal_horoprisms}
\end{figure}
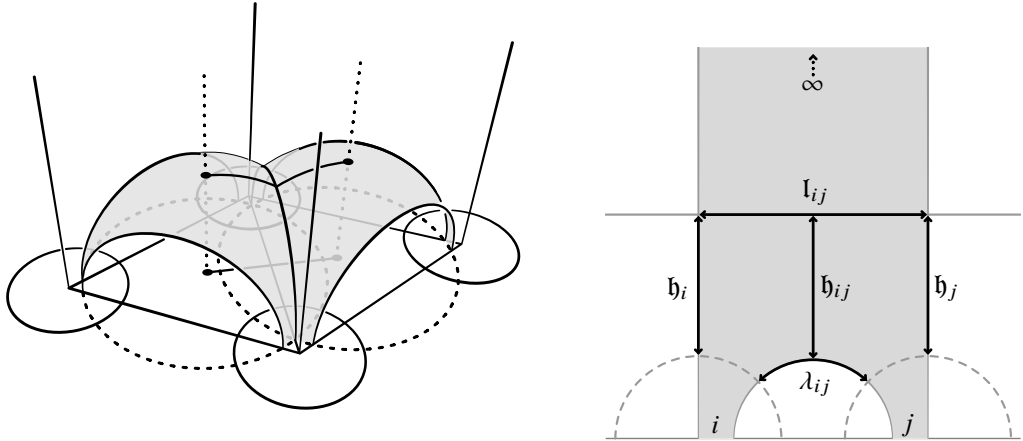

Let us review some facts about hyperideal polyhedral cusps and decorated piecewise
Euclidean surfaces (see \cite[Sec.~4]{BL2023} for more details). Similar to the
non-Euclidean case we can associate a complete hyperbolic surface with ends $\surf_g$
to a hyperideally decorated piecewise Euclidean metric
$(\dist_{\toposurf_g}, \mathfrak{r})$ on the marked surface $(\toposurf_g, \verts)$,
\ie, its fundamental discrete conformal invariant. For a geodesic triangulation
$\tri$ of the surface the corresponding discrete Euclidean metric
$(\tri, \mathfrak{l}, \mathfrak{r})$ is related to the $\lambda$-lengths of $\surf_g$
via
\begin{equation}\label{eq:lambda_length_to_inversive_distance_euclidean}
    \cosh\lambda_{ij}
    \;=\;
    \frac{\mathfrak{l}_{ij}^2\,-\,\mathfrak{r}_i^2\,-\,\mathfrak{r}_j^2}
        {2\mathfrak{r}_i\mathfrak{r}_j}.
\end{equation}
Furthermore, the heights $\mathfrak{h}\in\RR^{\verts}$ of the polyhedral cusp
over $\surf_g$ determined by $(\tri, \mathfrak{l}, \mathfrak{r})$ are given by
\begin{equation}\label{eq:euclidean_radii_to_heights}
    \mathfrak{r}_i
    \;=\;
    \ee^{-\mathfrak{h}_i}.
\end{equation}
These heights are measured with respect to an auxiliary horosphere at the
distinguished vertex of the polyhedral cusp (see \figref{fig:hyperideal_horoprisms},
right). Thus, they determine the polyhedral cusp only up to shifts by
$\bm{1}_{\verts}$. Here, $\bm{1}_{\verts}\in\RR^{\verts}$ is the constant vector.
As in the non-Euclidean case, the polyhedral cusp is convex if and only if $\tri$
is a canonical triangulation of $\surf_g$ (with respect to the weights
$\ee^{\mathfrak{h}_i}$). Therefore, one can show that the space of convex
polyhedral cusps is given by the quotient space $\RR^{\verts}/\lin(\bm{1}_{\verts})$.

\begin{proposition}
    \label{prop:infty_bdy_polyhedral_cusps}
    To each element of $\partial_{\infty}\polyhedra^{\pm}(\surf_g)$
    corresponds a convex polyhedral cusp. Conversely, each convex polyhedral cusp
    determines a point in $\partial_{\infty}\polyhedra^{\pm}(\surf_g)$. This
    correspondence gives a homeomorphism between
    $\partial_{\infty}\polyhedra^{\pm}(\surf_g)$ and the quotient space
    $\RR^{\verts}/\lin(\bm{1}_{\verts})$.
\end{proposition}
\begin{proof}
    This follows from a simple geometric observation. It is similar in the
    spherical and hyperbolic case. We only elaborate the hyperbolic one.

    For each hyperideal prism we can construct a unique hyperideal horoprism
    with the same base face and orthogonal projection of the apex to that face.
    So we can associate to each convex polyhedral end $h\in\polyhedra^{-}(\surf_g)$
    a convex polyhedral cusp with heights
    $\mathfrak{h}\in\RR^{\verts}/\lin(\bm{1}_{\verts})$.
    Using the hyperbolic law of cosines we see that
    \begin{equation}\label{eq:ends_to_cusps}
        \frac{\tau_{\epsilon_i}(h_i)}{\sinh(h_{ij})}
        \;=\;
        \tau_{\epsilon_i}'(d_i)
        \;=\;
        \frac{\ee^{\mathfrak{h}_i}}{\ee^{\mathfrak{h}_{ij}}}.
    \end{equation}
    Here $d_i$ is the distance between the vertex $i$ and the foot of the
    projection of the apex to the edge $ij$. Furthermore, $h_{ij}$ and
    $\mathfrak{h}_{ij}$ are the distances between the apex and edge $ij$ in the
    prism and horoprism, respectively (see \figref{fig:hyperideal_horoprisms},
    right). Now consider
    $h^t \coloneq (\Omega_R^{-})^{-1}\big(t\,\Omega_R^{-}(h^1)\big)$ for
    $h^1\in\polyhedra^{-}(\surf_g)$ and all $t\geq1$.
    In the proof of \lemref{lemma:hyperbolic_he_fiber_limit} we showed that
    $h_i^t - h_j^t$ converges to
    \(
        \ln\big(\nicefrac{\tau_{\epsilon_i}(h_i^1)}{\tau_{\epsilon_j}(h_j^1)}\big)
    \).
    In light of \eqref{eq:ends_to_cusps} and the properties of the
    map $\Omega_R^{-}$ this yields the claim.
\end{proof}

The Euclidean discrete Hilbert--Einstein functional
$\HE_{\surf_g,\Theta}^0\colon\RR^{\verts}\to\RR$ can be defined by the same formula
as its non-Euclidean counterparts \eqref{eq:he_functional_local}. It is a twice
continuously differentiable function which descends to the quotient space
$\RR^{\verts}/\lin(\bm{1}_{\verts})$ if $\Theta\in\RR_{>0}^{\verts}$ satisfies the
Euclidean Gau{\ss}--Bonnet condition.

\begin{lemma}
    \label{lemma:he_extension_continuous}
    Given $\Theta\in\RR_{>0}^{\verts}$ satisfying the Euclidean Gau{\ss}--Bonnet
    condition \eqref{eq:euclidean_gauss_bonnet_condition}. The dHE-functional
    $\HE_{\surf_g, \Theta}^{\pm}$ extends continuously
    to $\partial_{\infty}\polyhedra^{\pm}(\surf_g)$. Indeed,
    \(
        \HE_{\surf_g, \Theta}^{\pm}(\mathfrak{h})
        =
        \HE_{\surf_g, \Theta}^{0}(\mathfrak{h})
    \)
    for $\mathfrak{h}\in\partial_{\infty}\polyhedra^{\pm}(\surf_g)$.
\end{lemma}
\begin{proof}
    Again we only discuss the hyperbolic case.
    Let $(h^n)_{n\geq0}\in\polyhedra^{-}(\surf_g)$ with
    \(
        \lim_{n\to\infty}h^n
        \eqcolon
        \mathfrak{h}
        \in
        \partial_{\infty}\polyhedra^{-}(\surf_g)
    \).
    We wish to show that
    \(
        \lim_{n\to\infty}\HE_{\surf_g, \Theta}^{-}(h^n)
        =
        \HE_{\surf_g, \Theta}^0(\mathfrak{h})
    \).
    \propref{prop:infty_bdy_polyhedral_cusps} grants that
    $(h^n-\min_ih_i^n)_{n\geq0}$ converges to the heights of a convex polyhedral
    cusp which is associated to $\mathfrak{h}$. Hence, following the argument in
    \lemref{lemma:hyperbolic_he_fiber_limit}, we only have to show that
    \begin{equation}
        \label{eq:angle_defect_vs_heights}
        \lim_{n\to\infty}\Big(\sum(\Theta_i-\theta_{i}^n)\Big)\,\min_ih_i^n
        \;=\;
        0.
    \end{equation}
    As always, we can assume that the triangulation is fixed.
    Write $H^n\coloneq\min_ih_i^n$ and $L^n\coloneq\max_{ij}\cosh\len_{ij}^n$.
    Using \eqref{eq:heights_lambda_lengths_relationship_hyperbolic}, we see that
    $L^n-1$ is asymptotically dominated by
    $\nicefrac{1}{\ee^{2H^n}}$, up to a constant factor. Furthermore,
    \begin{equation}
        \pi \,-\, \sum\theta_{jk}^i(h^n)
        \;=\;
        \area(ijk)(h^n)
        \;\leq\;
        2\pi(L^n-1)
    \end{equation}
    for all triangles $ijk\in\faces$. The equality is given by the
    Gau{\ss}--Bonnet theorem. The estimate follows by comparison to the area
    of a hyperbolic disk of radius $\acosh(L^n)=\max_{ij}\len_{ij}$. This implies
    \eqref{eq:angle_defect_vs_heights} since $\Theta$ satisfies the
    Euclidean Gau{\ss}--Bonnet condition by assumption.
\end{proof}

Schl\"afli's differential formula (\propref{prop:schlaflis_differential_formula})
shows that the first derivative of the Euclidean dHE-functional is given by
$\diff{}{\HE_{\surf_g,\Theta}^0}=\sum(\Theta_i-\theta_i)$. The Hessian of
$\HE_{\surf_g,\Theta}^0$ can also be computed explicitly. Its associated
quadratic form is $-\sum\cotw_{ij}^0(\diff{}{h_i}-\diff{}{h_j})^2$. The
Euclidean decorated cotan-weights are
\begin{equation}\label{eq:euclidean_cotan_weights}
    \cotw_{ij}^0
    \;\coloneq\;
    \big(\cot\alpha_{ij}^k+\cot\alpha_{ij}^l\big)\,
    \frac{\mathfrak{r}_{ij}}{\mathfrak{l}_{ij}}.
\end{equation}

\begin{lemma}
    \label{lemma:he_extension_hessian}
    Given $\Theta\in\RR_{>0}^{\verts}$ satisfying the Euclidean Gau{\ss}--Bonnet
    condition \eqref{eq:euclidean_gauss_bonnet_condition}. The first and
    second derivative of the dHE-functional $\HE_{\surf_g, \Theta}^{\pm}$ extend
    continuously to $\partial_{\infty}\polyhedra^{\pm}(\surf_g)$. In particular,
    they coincide with the derivatives of $\HE_{\surf_g, \Theta}^{0}$ over
    $\partial_{\infty}\polyhedra^{-}(\surf_g)$.
\end{lemma}
\begin{proof}
    The convergence of the first derivative follows from the convergence of the
    cone-angles (\lemref{lemma:he_angle_limit}).
    We are going to argue that the decorated cotan-weights
    (see \eqref{eq:cotan_weights}) corresponding to a sequence
    $(h^n)_{n\geq0}\in\polyhedra^{\pm}(\surf_g)$ with
    \(
        \lim_{n\to\infty}h^n
        \eqcolon
        \mathfrak{h}
        \in
        \partial_{\infty}\polyhedra^{\pm}(\surf_g)
    \)
    converge to the decorated cotan-weights \eqref{eq:euclidean_cotan_weights}
    corresponding to $\mathfrak{h}$. This implies the convergence of the Hessian.

    It is clear from the construction presented in
    \propref{prop:infty_bdy_polyhedral_cusps} that the dihedral
    angles $\alpha_{ij}^k$ of the hyperideal prisms associated to $h^n$ converge to
    the dihedral angles of the polyhedral cusp determined by $\mathfrak{h}$.
    So we only have to show that $\nicefrac{\tanh r_{ij}^n}{\sinh \len_{ij}^n}$
    converges to $\nicefrac{\mathfrak{r}_{ij}}{\mathfrak{l}_{ij}}$. A
    computation similar to the one in
    \lemref{lemma:lambda_length_to_inversive_distance_hyperbolic} shows that
    \begin{equation}
        \frac{\tanh r_{ij}^n}{\sinh \len_{ij}^n}
        \;=\;
        \frac{1}{\sinh\len_{ij}^n\cosh h_{ij}^n}.
    \end{equation}
    Expanding the right hand side of this equation using
    \eqref{eq:heights_lambda_lengths_relationship_hyperbolic},
    \eqref{eq:lambda_length_to_inversive_distance_euclidean},
    \eqref{eq:euclidean_radii_to_heights}, and \eqref{eq:ends_to_cusps},
    the result follows from a direct, but tedious, computation.
\end{proof}

\begin{remark}
    The discussion in this section shows that we can identify
    the ideal boundaries of $\polyhedra^{+}(\surf_g)$ and $\polyhedra^{-}(\surf_g)$.
    So the space
    \begin{equation}
        \polyhedra(\surf_g)
        \;\coloneq\;
        \polyhedra^{+}(\surf_g)
        \cup
        \partial_{\infty}\polyhedra^{\pm}(\surf_g)
        \cup
        \polyhedra^{-}(\surf_g).
    \end{equation}
    is connected. Actually, \propref{prop:space_polyhedra_characterization}
    shows that $\polyhedra(\surf_g)$ is homeomorphic to $\RR^{\verts}$ and
    supports a \enquote{cell-decomposition} corresponding to canonical tessellations
    of $\surf_g$ (see, \propref{prop:properties_of_weightings}).
    Furthermore, the discrete Hilbert--Einstein functional can be interpreted
    as a twice continuously differentiable function over the space
    $\polyhedra(\surf_g)$.
\end{remark}


\appendix
\section{Spherical weighted Delaunay tessellations}
\label{sec:spherical_wDt}
This appendix is dedicated to proving the following

\begin{proposition}[spherical weighted Delaunay tessellations]
    \label{prop:spherical_weighted_Delaunay_tessellations}
    Let $(\dist_{\toposurf_g}, r)$ be a hyperideally decorated piecewise spherical
    metric on the marked surface $(\toposurf_g, \verts)$. The weighted
    Delaunay tessellation of $(\toposurf_g, \verts, \dist_{\toposurf_g}, r)$ exists and
    is unique.
\end{proposition}

The uniqueness of weighted Delaunay tessellations was already discussed in
\lemref{lemma:uniqueness_wDt}. So we are going to focus on their existence. Indeed,
we are describing a method to compute weighted Delaunay triangulations from an
arbitrary starting geodesic triangulation. In the next
section we will collect some necessary facts from spherical geometry. Afterwards we
will construct weighted Delaunay triangulations with the \emph{flip algorithm}.
Our approach follows closely the exposition in \cite[Sec.~3.4]{Lutz2023}.
We refer the interested reader to this paper for more information on the related
literature.

\subsection{Geometric preliminaries}
\begin{lemma}\label{lemma:isosceles_triangles}
    Let $(\tri, \len, r)$ be a hyperideally decorated discrete spherical metric.
    If the triangle $ijk\in\faces_{\tri}$ is glued to itself along the edge $ij$,
    then $ij$ is local Delaunay.
\end{lemma}
\begin{proof}
    By assumption $ijk$ has to be an isosceles triangle and $ij$ one of its legs.
    Furthermore, the vertex cycles at the base have the same radius.
    By symmetry, the center of a circle which intersects the vertex-circles at
    $j$ and $k$ orthogonally lies on the perpendicular bisector of the edge $jk$.
    Hence, $d_{ij}^k>0$ and $d_{ki}^j>0$. The result follows from
    \lemref{lemma:local_characterization_wDt}.
\end{proof}

\begin{lemma}\label{lemma:concave_quads}
    Let $(\tri, \len, r)$ be a hyperideally decorated discrete spherical metric.
    If $ijk, ilj\in\faces_{\tri}$ are adjacent triangles with the common edge
    $ij$ such that $\theta_{jk}^i+\theta_{jl}^i\geq\pi$, then $ij$ is local Delaunay.
\end{lemma}
\begin{proof}
    Realize the two decorated triangles in $\SS^2$. Now, stereographically project
    from the antipodal point of $i$ to $\RR^2$. This projection preserves
    angles and circles, sends $i$ to $0$, and spherical geodesics through $i$
    to euclidean straight lines. Hence, this assertion follows from its euclidean
    counterpart.
\end{proof}

\begin{lemma}\label{lemma:center_distance_estimate}
    Let $C_1$ and $C_2$ be two orthogonally intersecting circles in $\SS^2$.
    The distance between their centers is $\leq\nicefrac{\pi}{2}$.
\end{lemma}
\begin{proof}
    This is a consequence of the spherical law of cosines. If $d_{12}$ denotes
    the distance between the centers and $r_1$ and $r_2$ the radii, respectively,
    then $\cos(d_{12}) = \cos(r_1)\cos(r_2)$. By definition
    $r_1,r_2\in[0,\nicefrac{\pi}{2})$. So the result follows.
\end{proof}

\begin{lemma}\label{lemma:bounded_geodesics}
    Let $(\dist_{\toposurf_g}, r)$ be a hyperideally decorated piecewise spherical
    metric on the marked surface $(\toposurf_g, \verts)$. For each pair
    $(i,j)\in\verts^2$ and $L>0$ there is only a finite number of geodesic arcs which
    start in $i$, end in $j$, and have a length not exceeding $L$.
\end{lemma}
\begin{proof}
    This is a classical fact about geodesics of a compact surface
    (see, \eg, \cite[Lem.\,3.2]{Lutz2023} for a proof).
\end{proof}

\subsection{The flip algorithm}
Consider a hyperideally decorated discrete spherical metric $(\tri, \len, r)$.
A triangle $ijk\in\faces_{\tri}$ can be realized in $\SS^2$.
The ambiguity of this realization is given by spherical isometries, \ie, the action
of $\SO(3)$. To the M\"obius-lift \eqref{eq:spherical_moebius_lift} corresponds
the affine representation
\begin{equation}
    C_i \,=\, \frac{p_i}{\cos r_i}.
\end{equation}
Similarly we can represent the face-circle $C_{ijk}$ as an element of $\RR^3$.
The \emph{support function}
$S_{ijk}\colon\SS^2\setminus\{\ip{x}{C_{ijk}}=0\}\to\RR$ of the
decorated triangle $ijk$ is given by
\begin{equation}
    1
    \;=\;
    \ip{S_{ijk}(x)\,x}{C_{ijk}}.
\end{equation}
A small computation shows that $\min|S_{ijk}| = \cos r_{ijk}$.
Moreover, by \lemref{lemma:center_distance_estimate},
$ijk\subset\SS^2\setminus\{\ip{x}{C_{ijk}}\leq0\}$. Hence, $S_{ijk}$
is a well-defined positive function on $ijk\subset\toposurf_g$.
For two adjacent triangles the support functions agree on their common edge.
It follows that $(\tri, \len, r)$ induces a continuous \emph{support function}
$S_{\tri, \len, r}\colon\toposurf_g\to\RR_{>0}$ restricting to
$S_{ijk}$ on each triangle $ijk\in\faces_{\tri}$.

\begin{proposition}[flip algorithm]
    Let $(\dist_{\toposurf_g}, r)$ be a hyperideally decorated piecewise spherical
    metric on the marked surface $(\toposurf_g, \verts)$. Start with any
    geodesic triangulation. Consecutively flipping edges violating
    the strict local Delaunay condition (see \lemref{lemma:local_characterization_wDt})
    terminates after a finite number of steps. The computed triangulation is a
    weighted Delaunay triangulation with
    respect to $(\dist_{\toposurf_g}, r)$.
\end{proposition}
\begin{proof}
    We only have to prove that the algorithm stops. Using
    \lemref{lemma:isosceles_triangles} and \lemref{lemma:concave_quads} show
    that an edge violating the strict local Delaunay
    condition can always be flipped. Let $\tri$ and $\tilde{\tri}$ be geodesic
    triangulations. Suppose $\tilde{\tri}$ can be obtained from $\tri$ by flipping
    an edge $ij\in\edges_{\tri}$ which violates the strict local Delaunay
    condition. Locally this is equivalent to changing to the outer convex hull
    of four points in $\RR^3$, \ie, the four representatives of the corresponding
    vertex-circles (see also the angle condition in
    \lemref{lemma:local_characterization_wDt}). It follows that
    $S_{\tilde{\tri}, \tilde{\len}, r}\geq S_{\tri, \len, r}$. In particular,
    $S_{\tilde{\tri}, \tilde{\len}, r}(x)>S_{\tri, \len, r}(x)$ for all
    $x\in ij$.

    Our previous discussion of the support function shows that
    \begin{equation}
        2\arccos(\min S_{\tri, \len, r})\,+\,2\max r_i
    \end{equation}
    is an upper bound for the lengths of the edges of $\tri$ only determined by.
    Therefore, there is only a finite number of geodesic triangulations of
    $(\toposurf_g, \dist_{\toposurf_g})$ whose edges satisfy
    this length-constraint, by \lemref{lemma:bounded_geodesics}. This implies that
    the number of geodesic
    triangulations $\tilde{\tri}$ with
    $S_{\tilde{\tri}, \tilde{\len}, r}\geq S_{\tri, \len, r}$
    is finite.
\end{proof}

\section*{Acknowledgements}
This work was funded by the \emph{Deutsche Forschungsgemeinschaft}
(\emph{DFG -- German Research Foundation}) -- Project-ID 195170736 -- SFB/TRR109
\enquote{Discretization in Geometry and Dynamics}.

\bibliographystyle{plain}
\bibliography{references.bib}

~\\
\begin{flushright}
   \noindent
   \textit{%
      Technische Universit\"at Berlin\\
      Institut f\"ur Mathematik\\
      Str.\ des 17.\ Juni 136\\
      10623 Berlin\\
      Germany
   }\\
   ~\\
   \noindent
   \texttt{bobenko@math.tu-berlin.de}\\
   \url{https://page.math.tu-berlin.de/~bobenko},\\
   ~\\
   \noindent
   \texttt{clutz@math.tu-berlin.de}\\
   \url{https://page.math.tu-berlin.de/~clutz}
\end{flushright}

\end{document}